\documentclass{amsart}

\usepackage{amssymb, amsmath,  amsfonts }
\usepackage{color }
\usepackage{slashbox, longtable}
\usepackage[all,cmtip]{xy}

\newtheorem{thm}{Theorem}[section]
\newtheorem{prop}[thm]{Proposition}
\newtheorem{lemma}[thm]{Lemma}
\newtheorem{cor}[thm]{Corollary}
\newtheorem{defn}[thm]{Definition}
\newtheorem{example}[thm]{Example}
\newtheorem{remark}[thm]{Remark}

\newcommand{\circline}{\!\;\!\!\circ\!\!\;\!\!-\!\!\!-\!}
\newcommand{\bulletline}{\bullet\!\!\;\!\!-\!\!\!-\!}

\newcommand{\typeb}{ \circ\!\!\;\!\!-\!\!\!-\!\circ\!\!\!\Rightarrow\!\!=\!\!\!\circ}
\newcommand{\typebbb}{\!\;\!\!\circ\!\!\!\Rightarrow\!\!=\!\!\!\circ}
\newcommand{\typec}{\noindent\circ\!\!\;\!\!-\!\!\!-\!\circ\!\!\!=\!\!\Leftarrow\!\!\!\circ}
\newcommand{\typeccc}{\!\;\!\!\circ\!\!\!=\!\!\Leftarrow\!\!\!\circ}

\numberwithin{equation}{section}

\begin{document}{\allowdisplaybreaks[3]

\title[Functorial relationships  between $QH^*(G/B)$ and $QH^*(G/P)$, (II)]{Functorial relationships \\ between $QH^*(G/B)$ and $QH^*(G/P)$, (II)}

 \author{Changzheng Li}
\address{Kavli Institute for the Physics and Mathematics of the Universe (WPI),
Todai Institutes for Advanced Study, The University of Tokyo,
5-1-5 Kashiwa-no-Ha,Kashiwa City, Chiba 277-8583, Japan}
\email{changzheng.li@ipmu.jp}

\date{
      }




\begin{abstract}

We show a canonical injective   morphism from the quantum cohomology ring $QH^*(G/P)$ to   the associated graded algebra of   $QH^*(G/B)$, which is with respect to a nice filtration on $QH^*(G/B)$ introduced by Leung and the author. This tells us the vanishing of a lot of genus zero, three-pointed Gromov-Witten invariants of flag varieties $G/P$.
\end{abstract}

\maketitle

\section{Introduction}

 The (small) quantum cohomology ring $QH^*(G/P)$   of a flag variety $G/P$   is a deformation of the ring structure on  the classical cohomology $H^*(G/P)$ by incorporating three-pointed, genus zero Gromov-Witten invariants of $G/P$. Here  $G$ denotes    a
simply-connected complex simple Lie group, and $P$ denotes a    parabolic subgroup   of $G$.
 There has been a lot of intense studies on  $QH^*(G/P)$  (see e.g.  the survey \cite{fu11} and references therein). In particular,  there was an insight on $QH^*(G/P)$ in the unpublished work \cite{peterson} of D. Peterson, which, for instance, describes a surprising connection  between $QH^*(G/P)$ and the so-called Peterson subvariety.
    When $P=B$ is the Borel subgroup of $G$, Lam and Shimozono  \cite{lamshi}  proved that $QH^*(G/B)$ is isomorphic to the homology of the group of the based loops in   a maximal compact Lie subgroup of $G$ with the ring structure given by the Pontryagin product,  after equivariant extension and localization  (see also \cite{peterson}, \cite{leungli22}).
   Woodward proved a comparison formula \cite{wo} of Peterson that  all genus zero, three-pointed    Gromov-Witten invariants of $G/P$ are contained in those of $G/B$. As a consequence, we can define  a canonical (injective) map   $QH^*(G/P)\hookrightarrow QH^*(G/B)$ as vector spaces.
  In \cite{leungli33}, Leung and the author constructed  a natural filtration $\mathcal{F}$ on $QH^*(G/B)$ which comes from a quantum analog of the Leray-Serre spectral sequence for the natural fibration $P/B\rightarrow G/B\longrightarrow G/P$. The next theorem is our main result in the present paper,  precise descriptions of which
      will be given in    Theorem \ref{thmfirstmainthm}. 
  \bigskip


 \noindent{\textbf{Main Theorem.}}
   {\itshape   There is a  canonical injective morphism of   algebras from the quantum cohomology ring $QH^*(G/P)$ to the associated graded algebra of  $QH^*(G/B)$ with respect to the filtration $\mathcal{F}$. }
\bigskip

 \noindent
 The above statement was proved by Leung and the author
under an additional assumption on $P/B$.
Here we do not require any constraint on $P/B$. That is, we  prove
  Conjecture 5.3 of \cite{leungli33}.
  Combining the   main results therein with the above theorem, we can tell a complete story as follows.
\begin{thm}\label{thmgenthm11} Let $r$  denote the   semisimple rank of the Levi subgroup of $P$   containing a maximal torus $T\subset B$.
\begin{enumerate}
  \item  There exists a   $\mathbb{Z}^{r+1}$-filtration $\mathcal{F}$ on $QH^*(G/B)$,   
             respecting the quantum product structure.
   \item  There exist an ideal  $\mathcal{I}$ 
             of $QH^*(G/B)$ and  a canonical algebra isomorphism
      $$ QH^*(G/B)/\mathcal{I}  \overset{\simeq}{\longrightarrow}QH^*(P/B).$$

   \item   There exists a subalgebra $\mathcal{A}$ of $QH^*(G/B)$ together with an ideal $\mathcal{J}$ of $\mathcal{A}$, 
       such that               $QH^*(G/P)$ is canonically isomorphic to $\mathcal{A}/\mathcal{J}$  as algebras.

  \item   There exists a canonical
                     injective  morphism of graded algebras
            $$\Psi_{r+1}:\,\,  \,\, QH^*(G/P)\hookrightarrow Gr^{\mathcal{F}}_{(r+1)}\subset Gr^{\mathcal{F}}(QH^*(G/B)) $$
       together with an isomorphism of graded algebras after localization
    $$Gr^{\mathcal{F}}(QH^*(G/B))\cong\big(\bigotimes_{j=1}^{r}QH^*(P_{j}/P_{j-1})\big)\bigotimes Gr^{\mathcal{F}}_{(r+1)},$$
   where $P_j$'s are parabolic subgroups constructed in  a canonical way, forming   a     chain  $B:=P_0\subsetneq P_1\subsetneq\cdots\subsetneq P_{r-1}\subsetneq P_r=P\subsetneq G$.
     Furthermore,  $\Psi_{r+1}$ is    an isomorphism if and only if the next hypothesis
       {\upshape (Hypo1)} holds: \qquad  $P_{j}/P_{j-1}$ is a  projective space for any $1\leq j\leq r$.
\end{enumerate}
\end{thm}
\noindent All the relevant ideals, subalgebras and morphisms  above will be described precisely  in Theorem \ref{thmgenthm22}.   To get a clearer idea of them here, we
use  the same toy example of the variety of complete flags in $\mathbb{C}^3$   as in \cite{leungli33}.
  \begin{example}\label{exampfiltrforA2}
Let $G=SL(3, \mathbb{C})$ and   $B\subsetneq P\subsetneq G$. Then we have
          $G/B=\{V_1\leqslant    V_2\leqslant \mathbb{C}^{3}~|~ \dim_\mathbb{C}V_1=1, \dim_\mathbb{C}V_2=2\}$, and the  natural projection  $\pi: G/B{\longrightarrow} G/P$ is given  by forgetting the vector subspace $V_1$ in the complete flag $V_1\leqslant V_2\leqslant\mathbb{C}^3$.
          In particular,
$P/B\cong\mathbb{P}^1$,   $ G/P\cong\mathbb{P}^2$, and the semisimple rank $r$ of the Levi subgroup of $P$ containing a maximal torus $T\subset B$ equals 1.
In this case, the quantum cohomology ring $QH^*(G/B)=(H^*(G/B)\otimes\mathbb{Q}[q_1, q_2], \star)$ has a $\mathbb{Q}$-basis $\sigma^wq_1^aq_2^b$, indexed by $(w, (a, b))\in W\times \mathbb{Z}_{\geq 0}^2$, and we define a grading map $gr(\sigma^wq_1^aq_2^b):=(2a-b, 3b)+gr(\sigma^w)\in\mathbb{Z}^2$.
Here  $W :=\{1, s_1, s_2, s_1s_2, s_2s_1, s_1s_2s_1\}$ is the Weyl group (isomorphic to the permutation group $S_3$). The grading $gr(\sigma^w)$ is the usual one from the Leray-Serre spectral sequence,     respectively given by $(0, 0)$, $(1, 0)$, $(0, 1)$, $(0, 2)$, $(1, 1)$, $(1, 2)$. Using the above gradings together with the lexicographical order on $\mathbb{Z}^2$ (i.e.,
     $(x_1,   x_2)<(y_1, y_2)$   if and only if either     $x_1<y_1$ or {\upshape $(x_{1}=y_{1}$} and {\upshape $x_2<y_2)$}), we have the following conclusions.

    \begin{enumerate}
    \item   There is a
$\mathbb{Z}^2$-filtration $\mathcal{F}=\{F_{\mathbf{x}}\}_{\mathbf{x}\in\mathbb{Z}^2}$ on $QH^*(G/B)$, defined by  $F_{\mathbf{x}}:=\bigoplus\limits_{gr(\sigma^wq_1^aq_2^b)\leq \mathbf{x}}\mathbb{Q}\sigma^wq_1^aq_2^b\subset QH^*(G/B)$.
  Furthermore, $\mathcal{F}$ respects   the quantum product structure. That is, $F_{\mathbf{x}}\star F_{\mathbf{y}}\subset F_{\mathbf{x}+\mathbf{y}}$.

   \item $\mathcal{I}:=\!\!\!\!\bigoplus\limits_{gr(\sigma^wq_1^aq_2^b)\in\mathbb{Z}\times \mathbb{Z}^+}\!\!\!\!\mathbb{Q}\sigma^wq_1^aq_2^b$ is an ideal of $QH^*(G/B)$.
                    We take the standard ring presentation $QH^*(\mathbb{P}^1)=\mathbb{Q}[x, q]/\langle x^2-q\rangle$. Note $P/B\cong \mathbb{P}^1$.  Then  $\sigma^{s_1}+\mathcal{I}\mapsto x$ and $q_1+\mathcal{I}\mapsto q$ define an isomorphism of algebras from  $QH^*(G/B)/\mathcal{I}$ to  $QH^*(P/B)$.
    \item $\mathcal{A}:=\sum_{k\in \mathbb{Z}} F_{(0, k)}$ is a subalgebra of $QH^*(G/B)$, and  $\mathcal{J}:=F_{(0, -1)}$ is an ideal of $\mathcal{A}$.
     Write $QH^*(\mathbb{P}^2)=\mathbb{Q}[z, t]/\langle z^3-t\rangle$. Note $G/P\cong \mathbb{P}^2$. Then $z\mapsto \sigma^{s_2}+\mathcal{J}, z^2\mapsto \sigma^{s_1s_2}+\mathcal{J}$ and
     $t\mapsto \sigma^{s_1}q_2+\mathcal{J}$ define an isomorphism of algebras from $QH^*(G/P)$ to     $\mathcal{A}/\mathcal{J}$.
     \item    $Gr_{(2)}^{\mathcal{F}}:=\bigoplus_{k\in \mathbb{Z}} F_{(0, k)}\big/\sum_{\mathbf{x}<(0, k)}F_{\mathbf{x}}$ is a graded subalgebra of $Gr^{\mathcal{F}}(QH^*(G/B))$, and it is
                       canonically isomorphic to  $\mathcal{A}/\mathcal{J}$ as algebras. Combining it with (3), we have  an isomorphism of (graded) algebras
              $\pi_{\mathbf{q}}^*: QH^*(G/P)\overset{\cong}{\longrightarrow} Gr_{(2)}^{\mathcal{F}}$ (which, in general, is an injective morphism only).
 \end{enumerate}
 In addition,  by taking the classical limit, $\mathcal{F}|_{\mathbf{q}=\mathbf{0}}$ gives the usual $\mathbb{Z}^2$-filtration on $H^*(G/B)$ from the Leray-Serre spectral sequence.
    The classical limit of $\pi_{\mathbf{q}}^*$  also coincides with the induced morphism $\pi^*: H^*(G/P)\hookrightarrow H^*(G/B)$ of algebras.

\end{example}

In the present paper, we will prove Theorem \ref{thmgenthm11} in a combinatorial way. It will be very interesting to explore a conceptual explanation of the theorem. Such an explanation may involve the notion of vertical quantum cohomology in \cite{assa}. As an evidence, part (2) of Theorem \ref{thmgenthm11} turns out to coincide with equation (2.17) of \cite{assa} in the special case when $G=SL(n+1, \mathbb{C})$. In a future project, we plan to investigate the relation between our results and those from \cite{assa}.
We would like to remind that a sufficient condition  {\upshape (Hypo2)}  for  $\Psi_{r+1}$ to be an isomorphism  was provided in \cite{leungli33}, which says that $P/B$ is isomorphic to a product of complete flag varieties of type $A$. It is not a strong constraint,  satisfied for  all flag varieties $G/P$ of type $A$, $G_2$ as well as  for most of flag varieties $G/P$ of each remaining Lie type.
The  necessary and sufficient condition in the above theorem is slightly more general. For instance for $G$ of type $F_4$, there are $16$ flag varieties $G/P$ in total (up to isomorphism together with the  two extremal cases $G/B$, $\{\mbox{pt}\}$  being   counted). Among them,  there are 13 flag varieties  satisfying both hypotheses  {\upshape (Hypo1)} and  {\upshape (Hypo2)}, while   one more flag variety satisfies      {\upshape (Hypo1)}. Precisely, for $G$ of type $F_4$, {\upshape (Hypo1)}  holds for all $G/P$ except for the two (co)adjoint Grassmannians that respectively correspond to (the complement of) the two ending nodes of the Dynkin diagram of type $F_4$.

The notion of quantum cohomology was introduced by the physicists in 1990s, and it can be defined over a smooth projective variety $X$.
   It is a quite challenging problem to study the quantum cohomology ring $QH^*(X)$, partially because of the lack of functorial property. Namely, in general, a reasonable morphism between two smooth projective varieties does not induce a morphism on the level of quantum cohomogy.   However,
 Theorem \ref{thmgenthm11}   tells us a beautiful story on the ``functoriality" among the special case of the quantum cohomology of flag varieties. We may even expect nice applications of it in future research.
Despite lots of interesting studies of $QH^*(G/P)$, they are mostly   for the varieties of
    partial flags of subspaces of  $\mathbb{C}^{n+1}$, i.e., when   $G=SL(n+1, \mathbb{C})$.
        For $G$ of general Lie type, ring presentations of the quantum cohomology are better understood for either complete flag varieties $G/B$ \cite{kim} or  most of Grassmannians, i.e., when $P$ is maximal (cf. \cite{BKT2}, \cite{cmn} and references therein).
        The special case of the functorial property   \cite{leungli33} when $P/B\cong \mathbb{P}^1$ has led to  nice applications on the ``quantum to classical" principle \cite{leungliQtoC},
     as further applications of which Leung and the author obtained certain quantum Pieri rules \cite{leungli44} as well as alternative proofs of the main results of \cite{BKT1}.
  On the other hand, our main result could also be treated as a kind of application of the ``quantum to classical" principle.
 As we can see later, the proof requires knowledge on the vanishing of a lot of Gromov-Witten invariants as well as explicit calculations of certain non-vanishing Gromov-Witten invariants that all turn out to be   equal to 1. Although Leung and the author have showed an explicit combinatorial formula for those Gromov-Witten invariants (with sign cancelation involved) in \cite{leungli22}, it would exceed  the capacity of a   computer in some cases if we use the formula directly. Instead, we will apply the ``quantum to classical" principle developed in \cite{leungliQtoC}.

 The paper is organized as follows. In section 2, we introduce  a (non-recursively defined) grading map $gr$   and state the main results of the present paper.   The whole of section 3 is devoted to a proof of Main Theorem when the Dynkin diagram of the Levi subgroup of  $P$ containing a maximal torus $T\subset B$ is connected, the outline of which is given at the beginning the section.    The proofs of some propositions in section 3 require arguments case by case. Details for all those cases not covered in the section are given in section 5.
 In section 4, we describe  Theorem \ref{thmgenthm11} in details and provide a sketch proof of it therein, in which there is no constraint on $P/B$. We also
 greatly clarify the grading map defined recursively in \cite{leungli33}, by showing the coincidence between it and
  the map $gr$ defined in section 2.  Both the definition of $gr$ and the conjecture of the coincidence between the two grading maps were due to  the anonymous referee of \cite{leungli33}. It is quite worth to prove the coincidence,   because the grading map was used to establish a nice filtration on $QH^*(G/B)$, which is the heart of  the whole story of the functoriality.

\section{Main results}\label{mainresults}

\subsection{Notations}\label{subsecnotations}
We will follow   most of the notations   used in \cite{leungli33}, which are repeated here  for the sake of completeness. Our readers can refer to   \cite{humalg} and \cite{fupa} for more details.

Let $G$ be a simply-connected  complex simple Lie group of rank $n$, $B$ be a
Borel subgroup, $T\subset B$ be a maximal complex torus with Lie algebra    $\mathfrak{h}=\mbox{Lie}(T)$, and  $P\supsetneq B$ be
   a proper  parabolic subgroup  of $G$.
Let $\Delta=\{\alpha_1, \cdots, \alpha_n\}\subset \mathfrak{h}^*$ be a basis of simple roots and $\{\alpha_1^\vee, \cdots, \alpha_n^\vee\}\subset\mathfrak{h}$ be
 the  simple coroots.
   Each parabolic subgroup $\tilde P\supset B$ is  in one-to-one correspondence with a subset $\Delta_{\tilde P}\subset \Delta$. Conversely,  by $P_{\tilde \Delta}$ we mean the parabolic subgroup containing $B$ that corresponds to a given subset $\tilde\Delta\subset \Delta$. Here $B$ contains the one-parameter unipotent subgroups $U_{\alpha}$, $\alpha\in \tilde \Delta$.
Clearly,
       $P_\Delta=G$, $\Delta_B=\emptyset$ and $\Delta_P\subsetneq \Delta$.
  Let    $\{\omega_1, \cdots, \omega_n\}$ (resp. $\{\omega_1^\vee, \cdots, \omega_n^\vee\}$) denote   the fundamental (co)weights, and $\langle\cdot, \cdot\rangle
                    :\mathfrak{h}^*\times\mathfrak{h}\rightarrow \mathbb{C}$ denote the natural pairing.
Let $\rho:=\sum_{i=1}^n\omega_i$.

 The Weyl group $W$ is generated by $\{s_{\alpha_i}~|~ \alpha_i\in \Delta\}$, where each
                     simple reflection $s_i:=s_{\alpha_i}$ maps  $\lambda\in\mathfrak{h}$ and $\beta\in\mathfrak{h}^*$ to
        $s_{i}(\lambda)=\lambda-\langle \alpha_i, \lambda\rangle\alpha_i^\vee$ and $s_{i}(\beta)=\beta-\langle \beta, \alpha_i^\vee\rangle\alpha_i$ respectively. Let $\ell: W\rightarrow \mathbb{Z}_{\geq 0}$ denote the standard length function.
  Given a parabolic subgroup $\tilde P\supset B$,  we denote by   $W_{\tilde P}$
     the subgroup of $W$ generated by $\{s_\alpha~|~ \alpha \in \Delta_{\tilde P}\}$, in which
              there is a unique element of maximum length,
   say $w_{\tilde P}$.  Given another parabolic subgroup $\bar P$ with $B\subset \bar P\subset
      \tilde P$, we have $\Delta_{\bar P}\subset \Delta_{\tilde P}$.
       Each coset in $W_{\tilde P}/W_{\bar P}$ has
        a unique minimal length  representative. The set of all these minimal length  representatives is denoted by  $W_{\tilde P}^{\bar P}(\subset W_{\tilde P}\subset W)$.
        Note that $W_B=\{\mbox{id}\}$,   $W_{\tilde P}^{B}=W_{\tilde P}$ and  $W_{G}=W$. We simply denote  $w_0:=w_G$ and $W^{\bar P}:=W_{G}^{\bar P}$.

    The root system is given by $R=W\cdot\Delta=R^+\sqcup (-R^+)$, where
                       $R^+=R\cap \bigoplus_{i=1}^n{\mathbb{Z}_{\geq 0}}\alpha_i$ is the set of positive roots. It is a well-known fact that $\ell(w)=|\mbox{Inv}(w)|$ where   $\mbox{Inv}(w)$ is the \textit{inversion set} of $w\in W$ given by
          $$\mbox{Inv}(w):=\{\beta\in R^+  ~|~ w(\beta)\in -R^+\}.$$
 Given  $\gamma=w(\alpha_i)\in R$, we
                have the coroot $\gamma^\vee:=w(\alpha_i^\vee)$ in the coroot lattice   $Q^\vee:=\bigoplus_{i=1}^n\mathbb{Z}\alpha_i^\vee$  and the reflection $s_\gamma:=ws_iw^{-1}\in W$, which are independent of the expressions of
           $\gamma$.
  For the given $P$, we denote   by $R_P=R^+_P\sqcup (-R^+_P)$ the root subsystem, where
      $R_{P}^+=R^+\cap \bigoplus_{\alpha\in \Delta_P}\mathbb{Z}\alpha$, and denote              $Q^\vee_P:=\bigoplus_{\alpha_i\in \Delta_P}\mathbb{Z}\alpha_i^\vee$.

   The (co)homology of the flag variety $G/P$ has an additive basis of Schubert (co)homology classes $\sigma_u$  (resp. $\sigma^u$) indexed by $W^P$.
   In particular,
       we can identify    $H_2(G/P,\mathbb{Z})=\bigoplus_{\alpha\in \Delta\setminus\Delta_P}
         \mathbb{Z}\sigma_{s_{\alpha}}$
               with $Q^\vee/Q^\vee_P$ canonically,
            by mapping $\sum_{\alpha_j\in \Delta\setminus\Delta_P}a_j\sigma_{s_{\alpha_j}}$ to $\lambda_P=\sum_{\alpha_j\in \Delta\setminus\Delta_P}a_j\alpha_j^\vee+Q^\vee_P$.
        For each $\alpha_j\in \Delta\setminus\Delta_P$, we introduce a formal variable $q_{\alpha_j^\vee+Q^\vee_P}$.
          For such $\lambda_P$, we denote $q_{\lambda_P}=\prod_{\alpha_j\in \Delta\setminus\Delta_P}q_{\alpha_j^\vee+Q^\vee_P}^{a_j}$.
The (\textbf{small}) \textbf{quantum cohomology ring}   $QH^*(G/P)=(H^*(G/P)\otimes\mathbb{Q}[\mathbf{q}],  \star_P)$ of $G/P$ also has
       a natural $\mathbb{Q}[\mathbf{q}]$-basis of Schubert classes $\sigma^u=\sigma^u\otimes 1$, for which
                $$\sigma^u\star_P \sigma^v = \sum_{w\in W^P, \lambda_P\in Q^\vee/Q^\vee_P}    N_{u,v}^{w, \lambda_P}q_{\lambda_P}\sigma^w.$$
 The quantum product $\star_P$ is associative and commutative.
  The  \textit{quantum Schubert structure constants}  $N_{u, v}^{w, \lambda_P}$  are all \textit{non-negative}, given by
      genus zero, 3-pointed Gromov-Witten invariants of $G/P$.
       When $P=B$, we have $Q^\vee_P=0$, $W_P=\{1\}$ and $W^P=W$. In this case, we simply denote
     $\star=\star_P$,    $\lambda=\lambda_P$ and $q_j=q_{\alpha_j^\vee}$.

   \subsection{Main results}\label{subsecMainResult}
     We will assume the Dynkin diagram $Dyn(\Delta_P)$ to be connected throughout the paper except in section \ref{sectiongeneralcase}.  Denote $r:=|\Delta_P|$. Note $1\leq r<n$.

   Recall that a natural $\mathbb{Q}$-basis of $QH^*(G/B)[q_1^{-1}, \cdots, q_n^{-1}]$ is given by  $q_\lambda\sigma^{w}$ labelled by $(w, \lambda)\in W\times Q^\vee$.     Note that
   $q_\lambda \sigma^w\in QH^*(G/B)$ if and only if $q_\lambda\in \mathbb{Q}[\mathbf{q}]$ is a polynomial.
   In \cite{leungli33}, Leung and the author introduced a grading map
              $$gr: W\times Q^\vee \longrightarrow \mathbb{Z}^{r+1}=\bigoplus\nolimits_{i=1}^{r+1}\mathbb{Z}\mathbf{e}_{i}.$$
Due to Lemma 2.12 of \cite{leungli33}, the following subset
          $$ S :=\{gr(w, {\lambda})~|~ q_{\lambda} \sigma^{w}\in QH^*(G/B)\}  \footnote{$(w, \lambda)$ is simply denoted as $wq_\lambda$ in \cite{leungli33}.}$$
    is a totally-ordered sub-semigroup of $\mathbb{Z}^{r+1}$.
Here we are using  the  \textbf{lexicographical order} on elements $\mathbf{a}=(a_1, \cdots, a_{r+1})=\sum_{i=1}^{r+1}a_i\mathbf{e}_{i}$ in  $\mathbb{Z}^{r+1}$. Namely      $\mathbf{a}<\mathbf{b}$  
   if and only if there exists $1\leq j\leq r+1$ such that   $a_j<b_j$ and $a_{i}=b_{i}$ for all $1\leq i< j$.
    We can define
 a family    $\mathcal{F}=\{F_{\mathbf{a}}\}_{\mathbf{a}\in S}$ of subspaces of $QH^*(G/B)$, in which
   $$        F_{\mathbf{a}}:= \bigoplus_{gr(w, {\lambda})\leq \mathbf{a}}\mathbb{Q}q_{\lambda}\sigma^w\subset QH^*(G/B).$$
It is one of the main theorems in \cite{leungli33} that

\begin{prop}[Theorem 1.2 of \cite{leungli33}]\label{propmainthm}
   $QH^*(G/B)$ is an $S$-filtered algebra with filtration $\mathcal{F}$. Furthermore, this  $S$-filtered algebra structure
   is naturally   extended to a $\mathbb{Z}^{r+1}$-filtered algebra structure on $QH^*(G/B)$.
\end{prop}

As a consequence,
          we obtain 
             the associated $\mathbb{Z}^{r+1}$-graded algebra
       $$Gr^{\mathcal{F}}(QH^*(G/B)):=\bigoplus_{\mathbf{a}\in \mathbb{Z}^{r+1}} Gr_\mathbf{a}^{\mathcal{F}}, \mbox{ where }
                Gr_{\mathbf{a}}^{\mathcal{F}}:=F_{\mathbf{a}}\big/\sum_{\mathbf{b}<\mathbf{a}}F_{\mathbf{b}}.$$
 In particular, we have a  graded subalgebra
    $$Gr^{\mathcal{F}}_{(r+1)} :=\bigoplus_{i\in \mathbb{Z}}Gr_{i\mathbf{e}_{r+1}}^{\mathcal{F}}. $$

  Recall the next Peterson-Woodward comparison formula  \cite{wo} (see also \cite{lamshi}).

  \begin{prop}\label{propcomparison}
     For any $\lambda_P\in Q^\vee/Q_P^\vee$,  there exists a unique $\lambda_B\in Q^\vee$ such that $\lambda_P=\lambda_B+Q_P^\vee$ and
                $\langle \alpha, \lambda_B\rangle  \in \{0, -1\}$ for all $\alpha\in R^+_P$.
                Furthermore for every $u, v, w\in W^P$, we have  $$N_{u,v}^{w, \lambda_P }=N_{u, v}^{ww_Pw_{P'},   \lambda_B},$$
               where  
               $\Delta_{P'}=\{\alpha\in \Delta_P~|~\langle  \alpha, \lambda_B\rangle =0\}$.
    \end{prop}
The above formula, comparing Gromov-Witten invariants for $G/P$ and for $G/B$,
  induces   an injective map
             \begin{align*}\psi_{\Delta, \Delta_P}:\,\, &\,\, QH^*(G/P)\longrightarrow QH^*(G/B);\\
                               &\!\!\!\sum a_{w, \lambda_P}q_{\lambda_P}\sigma^{w}\mapsto
       \sum a_{w, \lambda_P}q_{\lambda_B}\sigma^{ww_Pw_{P'}},
     \end{align*}
and we call $\lambda_B$ the \textit{Peterson-Woodward lifting} of $\lambda_P$.
The next proposition is another one of the main theorems in \cite{leungli33} (see Proposition 3.24 and Theorem 1.4 therein).
\begin{prop}\label{propanothermainthm}
   Suppose that  $\Delta_P$ is of $A$-type. Then   the following map
     \begin{align*}\Psi_{r+1}:\,\, &\,\, QH^*(G/P)\hookrightarrow Gr^{\mathcal{F}}_{(r+1)}; \\
                               &\,\quad\quad q_{\lambda_P}\sigma^w\mapsto \overline{\psi_{\Delta, \Delta_P}(q_{\lambda_P}\sigma^w)}
     \end{align*}
   is well-defined, and it is    an isomorphism of  (graded)  algebras.
\end{prop}
Conjecture 5.3 of \cite{leungli33} tells us the counterpart of the above proposition, and it is the  main result  of the present paper that such a conjecture does hold. Namely

 \begin{thm}\label{thmfirstmainthm}
     Suppose that  $\Delta_P$ is not of $A$-type. Then   the  map $\Psi_{r+1}$ given in Proposition \ref{propanothermainthm} is
        well-defined, and it is    an injective morphism  of (graded) algebras. Furthermore, $\Psi_{r+1}$ becomes an isomorphism if and only if
         $r=2$ together with either case {\upshape \mbox{C1}$B)$} or case {\upshape \mbox{C9)}} of Table \ref{tabrelativeposi} occurring.
    \end{thm}

  \begin{remark}
  The algebra $QH^*(G/P)$ is equipped with a  natural $\mathbb{Z}$-grading:   a Schubert class $\sigma^w$ is of grading $\ell(w)$, and a quantum variable $q_{\alpha^\vee+Q^\vee_P}$ is of grading $\langle \sigma_{s_{\alpha}}, c_1({G/P})\rangle$. Once we show that $\Psi_{r+1}$ is an morphism of algebras, the   way of defining
   $\Psi_{r+1}$ automatically tells us that it preserves the $\mathbb{Z}$-grading as well.
  \end{remark}
  We will provide the proof in the next section, one point of which is to compute certain Gromov-Witten invariants explicitly. 

In order  to define the grading map $gr$ in \cite{leungli33}, Leung and the author introduced an ordering on the subset $\Delta_P$ first. In our case when $\Delta_P$ is not of type $A$, such an ordering is equivalent to the assumption that  $\Delta_P=\{\alpha_1, \cdots, \alpha_r\}$ with all the possible  Dynkin diagrams $Dyn(\Delta)$ being listed in Table \ref{tabrelativeposi}. These are precisely the cases for which Theorem \ref{thmfirstmainthm} is not covered in \cite{leungli33}. In addition, Table \ref{tabrelativeposi} has exhausted all the possible cases of fiberations $G/B\to G/P$ such that $Dyn(\Delta_P)$ is connected but not of type $A$. Therein the cases are basically numbered according to those for  $Dyn(\{\alpha_1, \cdots, \alpha_{r-1}\})$ in Table 2 of \cite{leungli33}.

 \begin{table}[h]
  \caption{\label{tabrelativeposi}  
    }
   \begin{tabular}{|c|c|}
     \hline
      & Dynkin diagram of $\Delta$   \\
     \hline\hline
               \vspace{-0.034cm}          &\\
      \raisebox{1.1ex}[0pt]{ $\mbox{C} 1B)$}      &  \begin{tabular}{l} \raisebox{-0.4ex}[0pt]{$  \bullet\!\;\!\cdots\!\bulletline\!\!\;\!\bulletline\!\;\!\!\circ\!\cdots  \circline\typebbb $}\\
                 \raisebox{1.1ex}[0pt]{${\hspace{0.9cm}\scriptstyle{\alpha_{r+1}}\hspace{0.15cm}\alpha_1\hspace{0.65cm}\alpha_{r-1}\hspace{0.1cm} \alpha_{r} } $}
                \end{tabular}         \\ \hline
         \vspace{-0.034cm}      &\\
       \raisebox{1.1ex}[0pt]{ $\mbox{C} 1C)$} &  \begin{tabular}{l} \raisebox{-0.4ex}[0pt]{$  \bullet\!\;\!\cdots\!\bulletline\!\!\;\!\bulletline\!\;\!\!\circ\!\cdots  \circline\typeccc $}\\
                 \raisebox{1.1ex}[0pt]{${\hspace{0.9cm}\scriptstyle{\alpha_{r+1}}\hspace{0.15cm}\alpha_1\hspace{0.65cm}\alpha_{r-1}\hspace{0.1cm} \alpha_{r} } $}
                \end{tabular}        \\ \hline
   { $\mbox{C} 2)$} &  \begin{tabular}{l}  \vspace{-0.15cm} $\hspace{2.769cm}\circ \scriptstyle{\,\alpha_{r}}$ \\
                     \vspace{-0.20cm}{ $\hspace{2.763cm}\!\!\;\!\big|$} \\
                   \raisebox{-0.4ex}[0pt]{$  \bullet\!\;\!\cdots\!\bulletline\!\!\;\!\bulletline\!\;\!\!\circ\!\cdots  \circline\circline\!\!\;\circ$}\\
                 \raisebox{1.1ex}[0pt]{${\hspace{0.9cm}\scriptstyle{\alpha_{r+1}}\hspace{0.15cm}\alpha_1\hspace{0.6cm}\alpha_{r-2}\hspace{0.04cm} \alpha_{r-1} }  $}
                \end{tabular}          \\ \hline
   &   \begin{tabular}{l}  \vspace{-0.168cm} $\hspace{2.15cm}\!\circ \scriptstyle{\,\alpha_{6}}$ \\
                     \vspace{-0.20cm}{ $\hspace{2.106cm}\!\!\;\!\!\;\!\big|$} \\
                   \raisebox{-0.4ex}[0pt]{$ \, \bulletline\!\!\;\!\bulletline\circline\circline\circline\circline\!\!\;\!\,\!\circ$}\\
                 \raisebox{1.1ex}[0pt]{${\hspace{0.45cm}\scriptstyle{\alpha_7}   \hspace{0.3cm} \alpha_1\hspace{0.2cm}\alpha_{2} \hspace{0.2cm}\alpha_{3} \hspace{0.2cm}\alpha_{4}\hspace{0.2cm}\alpha_{5}} $}
              \end{tabular}
                   \\ \cline{2-2}
  \raisebox{3.5ex}[0pt]{ $\mbox{C} 4)$} &   \begin{tabular}{l}  \vspace{-0.168cm} $\hspace{2.15cm}\!\circ \scriptstyle{\,\alpha_{7}}$ \\
                     \vspace{-0.20cm}{ $\hspace{2.106cm}\!\!\;\!\!\;\!\big|$} \\
                   \raisebox{-0.4ex}[0pt]{$ \, \bulletline \circline\circline\circline\circline\circline\!\!\;\!\,\!\circ$}\\
                 \raisebox{1.1ex}[0pt]{${\hspace{-0.01cm}\scriptstyle{\alpha_8}   \hspace{0.2cm} \alpha_1\hspace{0.2cm}\alpha_{2} \hspace{0.2cm}\alpha_{3} \hspace{0.2cm}\alpha_{4}\hspace{0.2cm}\alpha_{5} \hspace{0.2cm}\alpha_{6}} $}
              \end{tabular}
              \\ \hline
  { $\mbox{C} 5)$} &    \begin{tabular}{l}  \vspace{-0.166cm} $\hspace{2.06cm}\!\circ \scriptstyle{\,\alpha_{4}}$ \\
                     \vspace{-0.20cm}{ $\hspace{2.025cm}\!\!\;\!\!\;\!\big|$} \\
                   \raisebox{-0.4ex}[0pt]{$ \hspace{-0.025cm}\bulletline\!\;\!\!\bulletline\!\;\!\!\bulletline \circline\circline\circline\!\!\;\!\,\!\circ$}\\
                 \raisebox{1.1ex}[0pt]{${\hspace{0.981cm}\scriptstyle{  \alpha_6}\hspace{0.2cm}\alpha_{5} \hspace{0.2cm}\alpha_{3} \hspace{0.2cm}\alpha_{2}\hspace{0.21cm}\alpha_{1}} $}
              \end{tabular}
                 \\ \cline{1-2}
  \end{tabular}
 \begin{tabular}{|c|c|c|}
     \hline
      & Dynkin diagram of $\Delta$   \\
     \hline\hline
        &    \begin{tabular}{l}  \vspace{-0.166cm} $\hspace{2.06cm}\!\circ \scriptstyle{\,\alpha_{3}}$ \\
                     \vspace{-0.20cm}{ $\hspace{2.025cm}\!\!\;\!\!\;\!\big|$} \\
                   \raisebox{-0.4ex}[0pt]{$ \hspace{-0.025cm}\bulletline\!\;\!\!\bulletline\!\;\!\!\bulletline \circline\circline\circline\!\!\;\!\,\!\bullet$}\\
                 \raisebox{1.1ex}[0pt]{${\hspace{0.981cm}\scriptstyle{  \alpha_6}\hspace{0.2cm}\alpha_{1} \hspace{0.2cm}\alpha_{2} \hspace{0.2cm}\alpha_{4}\hspace{0.21cm}\alpha_{5}} $}
              \end{tabular}
                  \\ \cline{2-2}
    &    \begin{tabular}{l}  \vspace{-0.166cm} $\hspace{2.06cm}\!\circ \scriptstyle{\,\alpha_{4}}$ \\
                     \vspace{-0.20cm}{ $\hspace{2.025cm}\!\!\;\!\!\;\!\big|$} \\
                   \raisebox{-0.4ex}[0pt]{$ \hspace{-0.025cm}\bulletline\!\;\!\!\bulletline\circline \circline\circline\circline\!\!\;\!\,\!\bullet$}\\
                 \raisebox{1.1ex}[0pt]{${\hspace{0.49cm}\scriptstyle{  \alpha_7}\hspace{0.2cm}\alpha_{1}\hspace{0.2cm}\alpha_{2} \hspace{0.2cm}\alpha_{3} \hspace{0.2cm}\alpha_{5}\hspace{0.21cm}\alpha_{6}} $}
              \end{tabular}
                  \\ \cline{2-2}
   \raisebox{3.8ex}[0pt]{ $\mbox{C} 7)$} &    \begin{tabular}{l}  \vspace{-0.166cm} $\hspace{2.06cm}\!\circ \scriptstyle{\,\alpha_{5}}$ \\
                     \vspace{-0.20cm}{ $\hspace{2.025cm}\!\!\;\!\!\;\!\big|$} \\
                   \raisebox{-0.4ex}[0pt]{$ \hspace{-0.025cm}\bulletline\circline\circline \circline\circline\circline\!\!\;\!\,\!\bullet$}\\
                 \raisebox{1.1ex}[0pt]{${\hspace{-0.01cm}\scriptstyle{  \alpha_8}\hspace{0.2cm}\alpha_{1}\hspace{0.2cm}\alpha_{2} \hspace{0.2cm}\alpha_{3} \hspace{0.2cm}\alpha_{4}\hspace{0.2cm}\alpha_{6}\hspace{0.21cm}\alpha_{7}} $}
              \end{tabular}
                   \\ \cline{2-2}
  &    \begin{tabular}{l}  \vspace{-0.168cm} $\hspace{2.081cm}\!\circ \scriptstyle{\,\alpha_{6}}$ \\
                     \vspace{-0.20cm}{ $\hspace{2.038cm}\!\!\;\!\!\;\!\big|$} \\
                   \raisebox{-0.4ex}[0pt]{$ \, \circline\circline\circline\circline\circline\circline\!\!\;\!\,\!\bullet$}\\
                 \raisebox{1.1ex}[0pt]{${\hspace{-0.01cm}\scriptstyle{\alpha_1}\hspace{0.2cm}\alpha_2 \hspace{0.2cm} \alpha_3\hspace{0.2cm}\alpha_{4} \hspace{0.2cm}\alpha_{5} \hspace{0.2cm}\alpha_{7}\hspace{0.21cm}\alpha_{8}} $}
              \end{tabular}
          \\ \hline
          &  \begin{tabular}{l} \raisebox{-0.4ex}[0pt]{$\bulletline\,\typebbb\!\!\;\!\!-\!\!\!-\!\bullet$}\\
                 \raisebox{1.1ex}[0pt]{${\hspace{-0.01cm}\scriptstyle{\alpha_4}\hspace{0.2cm}\alpha_1\hspace{0.25cm}\alpha_2\hspace{0.2cm}\alpha_{3}} $}
                 \end{tabular}
                  \\ \cline{2-2}
\raisebox{2.7ex}[0pt]{ $\mbox{C} 9)$} &  \begin{tabular}{l} \raisebox{-0.4ex}[0pt]{$\typeb\!\!\;\!\!-\!\!\!-\!\bullet$}\\
                 \raisebox{1.1ex}[0pt]{${\hspace{-0.01cm}\scriptstyle{\alpha_1}\hspace{0.2cm}\alpha_2\hspace{0.25cm}\alpha_3\hspace{0.2cm}\alpha_{4}} $}
                 \end{tabular}
                   \\ \hline
 { $\mbox{C} 10)$}  &     \begin{tabular}{l} \raisebox{-0.4ex}[0pt]{$\typec\!\!\;\!\!-\!\!\!-\!\bullet$}\\
                 \raisebox{1.1ex}[0pt]{${\hspace{-0.01cm}\scriptstyle{\alpha_1}\hspace{0.2cm}\alpha_2\hspace{0.25cm}\alpha_3\hspace{0.2cm}\alpha_{4}} $}
                 \end{tabular}
                    \\ \hline
          \end{tabular}
  \end{table}

 \begin{remark}
  In Table \ref{tabrelativeposi}, we have treated   bases of type $E_6$ and $E_7$ as  subsets of a base of type $E_8$ canonically.
  $Dyn(\Delta_P)$  is  always given by a unique  case in Table \ref{tabrelativeposi} except when $\Delta$ is of $E_6$-type together with $r=5$. In this
  exceptional  case,  both
             {\upshape C5)} and {\upshape C7)} occur and we can choose either of them.
   Note $2\leq r<n$.  The case of   $G_2$-type does not occur there.
 \end{remark}

 In \cite{leungli33}, the grading map $gr$ was defined recursively by using the Peterson-Woodward comparison formula together with the given ordering on $\Delta_P$. Here we will define $gr$ as below, following the suggestion of the referee of \cite{leungli33} (see also Remark 2.10 therein).

  \begin{defn}\label{defnofgrading} Let us choose the ordering of $\Delta_P$  as given in Table  \ref{tabrelativeposi}.
    For each $1\leq j\leq r$, we denote $\Delta_j:=\{\alpha_1, \cdots, \alpha_j\}$. Set $\Delta_0:=\emptyset$ and $\Delta_{r+1}:=\Delta$. Denote by $P_i:=P_{\Delta_i}$ the parabolic subgroup corresponding to $\Delta_i$ for all $0\leq i\leq r+1$.
      Recall that we have denoted by $\{\mathbf{e}_1, \cdots, \mathbf{e}_{r+1}\}$ the standard basis of $\mathbb{Z}^{r+1}$. Define a grading map $gr$ by
      {\upshape{\begin{align*}gr:\,\, &\,\, W\times Q^\vee \longrightarrow \mathbb{Z}^{r+1};\\
                               &\!\!\!(w, \lambda)\mapsto gr(w, \lambda)=
       \sum_{i=1}^{r+1} \Big(\big|\mbox{Inv}(w)\cap (R_{P_i}^+\setminus R_{P_{i-1}}^+)\big|+\sum_{\beta \in R_{P_i}^+\setminus R_{P_{i-1}}^+}\langle \beta, \lambda\rangle\Big) \mathbf{e}_i.
       \end{align*}
       }}
  Say $gr(u, \eta)=\sum_{i=1}^{r+1}a_i\mathbf{e}_i$. Let $1\leq j\leq k\leq r+1$.  As usual, we define
    $$|gr(u, \eta)|:=\sum_{i=1}^{r+1}a_i, \qquad gr_{[j, k]}(u, \eta):=\sum_{i=j}^{k}a_i\mathbf{e}_i.$$
      As a known fact, we have (see also the proof of Proposition \ref{propgradingcoincide} for detailed explanations) $$gr(w, 0)= \sum_{j=1}^{r+1}\ell(w_j)\mathbf{e}_j,$$
      where   $w_j  \in W_{P_j}^{P_{j-1}}$ are the unique elements such  that $w=w_{r+1}w_r\cdots w_1$.

 \end{defn}

 We will  show the next conjecture of the referee of \cite{leungli33}.
\begin{thm}\label{thmdefnscoincide}
   The two grading maps   by Definition \ref{defnofgrading} above and   by Definition 2.8 of \cite{leungli33} coincide with each other.
\end{thm}
  Because of  the coincidence, Proposition \ref{propmainthm} holds with respect to the grading map $gr$. Namely for any Schubert classes $\sigma^u, \sigma^v$ of $QH^*(G/B)$, if $q_\lambda \sigma^w$ occurs in the quantum multiplication $\sigma^u\star \sigma^v$, then
    $$gr(w, \lambda)\leq gr(u, 0)+gr(v, 0).$$
  The proof of Theorem \ref{thmdefnscoincide} will be given in section \ref{sectiongeneralcase}, which is completely independent of section 3.
   Due to the coincidence,  the proofs of several main results in \cite{leungli33} may be simplified substantially.
   We can describe the explicit gradings of all the simple coroots as follows, which were obtained by direct calculations using Definition \ref{defnofLeungLi33} of the grading map $gr'$.  (See section 3.5 of  \cite{leungli33}   for more details on the calculations.)

  \begin{prop}\label{propgradingofeachQQQ}
 Let $\alpha\in \Delta$. We simply denote $gr(\alpha^\vee):=gr(\mbox{id}, \alpha^\vee)$.
  \begin{enumerate}
    \item  $gr({\alpha^\vee})=2\mathbf{e}_{r+1}, \mbox{if } Dyn(\{\alpha\}\cup\Delta_P)$ { is disconnected}.
    \item  $gr({\alpha^\vee})=(1+j)\mathbf{e}_{j}+(1-j)\mathbf{e}_{j-1}$,   if $\alpha=\alpha_j$   with $j\leq r-1$   where  $0\cdot \mathbf{e}_0:=\mathbf{0}$.
    \item $gr({\alpha^\vee})$ is given in Table \ref{tabgradingQQ},  if $\alpha=\alpha_r  \mbox{ or } \alpha_{r+1}$.
{\upshape     \begin{table}[h]
  \caption{\label{tabgradingQQ}   
    }
\hspace{-0cm}  \begin{tabular}{|c|c|r|}
     \hline
     &$gr({\alpha_r^\vee})$   & $gr({\alpha_{r+1}^\vee})$\hspace{1.2cm} \\
     \hline\hline
       $\mbox{C} 1B)$         &  $2r\mathbf{e}_r-(2r-2)\mathbf{e}_{r-1}$ &   $(2r+1)\mathbf{e}_{r+1}-r\mathbf{e}_r-\sum\limits_{j=1}^{r-1}\mathbf{e}_j$ \\ \cline{1-3}
      { $\mbox{C} 1C)$}      &   $(r+1)\mathbf{e}_r-(r-1)\mathbf{e}_{r-1}$ &  $(2r+2)\mathbf{e}_{r+1}-(r+1)\mathbf{e}_r-\sum\limits_{j=1}^{r-1}\mathbf{e}_j$ \\ \hline
   { $\mbox{C} 2)$}      &   $ 2(r-1)\mathbf{e}_{r}+(2-r)(\mathbf{e}_{r-1}+\mathbf{e}_{r-2})$ &   $2r\mathbf{e}_{r+1}+(1-r)\mathbf{e}_r-\sum\limits_{j=1}^{r-1}\mathbf{e}_j$  \\ \hline
       &    & $(\mbox{for } r=6)\,\,\,  18\mathbf{e}_7-11\mathbf{e}_6-\sum\limits_{j=1}^{5}\mathbf{e}_j$ \\ \cline{3-3}
    \raisebox{3.5ex}[0pt] { $\mbox{C} 4)$}  &   \raisebox{3.5ex}[0pt] {{\upshape $(3r-7)\mathbf{e}_{r}+(3-r)\sum\limits_{j=r-3}^{r-1}\mathbf{e}_{j}$} } &$(\mbox{for } r=7)\,\,\,  29\mathbf{e}_8-21\mathbf{e}_7-\sum\limits_{j=1}^{6}\mathbf{e}_j$   \\ \hline
  { $\mbox{C} 5)$}
               &      &\\ \cline{1-1}
        { $\mbox{C} 7)$}        & \raisebox{1.5ex}[0pt]{$ 2(r-1)\mathbf{e}_{r}+(2-r)(\mathbf{e}_{r-1}+\mathbf{e}_{r-2})$} & \raisebox{1.5ex}[0pt]{$ {\textstyle ({r^2-r\over 2}+2)\mathbf{e}_{r+1}-{r^2-r\over 2}\mathbf{e}_r}$ \hspace{0.4cm}}    \\ \cline{1-3}
   { $\mbox{C} 9)$}
               &  $2r\mathbf{e}_r-(2r-2)\mathbf{e}_{r-1}$ &  $(r^2+2)\mathbf{e}_{r+1}-r^2\mathbf{e}_r$ \hspace{1.0cm} \\ \hline
 { $\mbox{C} 10)$}
               & {  $4\mathbf{e}_3-2\mathbf{e}_{2}$ } &   $ 8\mathbf{e}_{4}-6\mathbf{e}_3$ \hspace{1.4cm}  \\ \hline

          \end{tabular}
   \end{table}
}
    \item The remaining cases happen when   there are   two nodes adjacent to $Dyn(\Delta_P)$,  namely the  node   $\alpha_{r+1}$ and the other node, say    $\alpha_{r+2}$.  Then we have  either of the followings.
     \begin{enumerate}
       \item  $gr({\alpha_{r+2}^\vee})= 2r\mathbf{e}_{r+1}+(1-r)\mathbf{e}_r-\sum_{j=1}^{r-1}\mathbf{e}_j$, which holds  if {\upshape $\mbox{C} 7)$} occurs and $r\leq 6$;
        \item $gr({\alpha_{r+2}^\vee})= 5\mathbf{e}_3-2\mathbf{e}_2-\mathbf{e}_1$, which holds  if {\upshape $\mbox{C} 9)$} occurs and $r=2$.
     \end{enumerate}
    \end{enumerate}
  In particular, we have $|gr({\alpha^\vee})|=2$ for any $\alpha\in \Delta$.
 \end{prop}

\section{Proof of Theorem \ref{thmfirstmainthm}}


Recall that we have defined a grading map $gr: W\times Q^\vee\rightarrow \mathbb{Z}^{r+1}$. For convenience, for any   $q_\lambda\sigma^w\in QH^*(G/B)[q_1^{-1}, \cdots, q_n^{-1}]$, we will also use the following notation $$gr(q_\lambda \sigma^w):=gr(w, \lambda).$$
 The injective map $\psi_{\Delta, \Delta_P}: QH^*(G/P)\rightarrow QH^*(G/B)$ induces a  natural map
    $QH^*(G/P)\rightarrow Gr^{\mathcal{F}}(QH^*(G/B))$. That is, $q_{\lambda_P}\sigma^w\mapsto \overline{\psi_{\Delta, \Delta_P}(q_{\lambda_P}\sigma^w)}\in Gr_{\mathbf{a}}^\mathcal{F}\subset Gr^{\mathcal{F}}(QH^*(G/B))$, where $\mathbf{a}=gr(\psi_{\Delta, \Delta_P}(q_{\lambda_P}\sigma^w))$.
 We state the next proposition, which extends  Proposition 3.24 of \cite{leungli33} in  the case of parabolic subgroups $P$ such that $\Delta_P$ is not of type $A$.
 \begin{prop}\label{propwelldefinedgrading}
    For any $q_{\lambda_P}\sigma^w\in QH^*(G/P)$, we have $$gr_{[1, r]}(\psi_{\Delta, \Delta_P}(q_{\lambda_P}\sigma^w))=\mathbf{0}.$$
 \end{prop}
   \noindent Hence, $\mathbf{a}\in\mathbb{Z}\mathbf{e}_{r+1}$. That is,  the map $\Psi_{r+1}$ as   in Proposition \ref{propanothermainthm} is well-defined.  We can further show
 \begin{prop}\label{propinjectionofvectsp}
      $\Psi_{r+1}$ is an injective map  of vector spaces. Furthermore, $\Psi_{r+1}$ is surjective if and only if
      $r=2$ and either case {\upshape C1$B$)} or case {\upshape C9)} occurs.
 \end{prop}

\noindent We shall also show
\begin{prop}\label{propmorphismofalgebras}
   $\Psi_{r+1}$ is a morphism of algebras. That is,    for any $q_{\lambda_P}, q_{\mu_P}, \sigma^{v'}, \sigma^{v''}$ in $QH^*(G/P)$, we have
  \begin{enumerate}
    \item   $\Psi_{r+1}(\sigma^{v'}\star_P\sigma^{v''})=\Psi_{r+1}(\sigma^{v'})\star  \Psi_{r+1}(\sigma^{v''})$;
     \item        $\Psi_{r+1}(q_{\lambda_P} \star_P \sigma^{v'})
                  = \Psi_{r+1}(q_{\lambda_P})\star \Psi_{r+1}(\sigma^{v'})$;
      \item   $\Psi_{r+1}(q_{\lambda_P}\star_Pq_{\mu_P})=\Psi_{r+1}(q_{\lambda_P})\star  \Psi_{r+1}(q_{\mu_P})$.
  \end{enumerate}
 \end{prop}

 \noindent To achieve the above proposition, we will need to show the vanishing of a lot of Gromov-Witten invariants occurring in certain quantum products  in $QH^*(G/B)$, and will need to
      calculate  certain   Gromov-Witten invariants, which   turn out to be equal to 1.

 Clearly,    Theorem \ref{thmfirstmainthm} follows immediately  from the combination of the above propositions.
   The rest of this section is devoted to the proofs of these propositions. Here  we would like to remind our readers of the following notation conventions:
  \begin{enumerate}
    \item[(a)]{\itshape Whenever referring to an element $\lambda_P$ in $Q^\vee/Q_P^\vee$, by   $\lambda_B$ we always mean the unique Peterson-Woodward lifting in $Q^\vee$ defined in  Proposition \ref{propcomparison}.} Namely, $\lambda_B\in Q^\vee$ is the unique element  that satisfies $\lambda_P=\lambda_B+Q_P^\vee$ and
                $\langle \alpha, \lambda_B\rangle  \in \{0, -1\}$ for all $\alpha\in R^+_P$.
    \item[(b)] {\itshape We simply denote  $\tilde P:=P_{r-1}$.} Namely, we have $\Delta_{\tilde P}:=\{\alpha_1, \cdots, \alpha_{r-1}\}$.

    \item[(c)] {\itshape Whenever an element in $\lambda\in Q^\vee$  is given first, we always denote $\lambda_P:=\lambda+Q^\vee_P\in Q^\vee/Q^\vee_P$
             and $\tilde \lambda_{\tilde P}:=\lambda+ Q^\vee_{\tilde P}\in Q^\vee/Q^\vee_{\tilde P}$. } Note that   the three elements $\lambda, \lambda_B$ and $\tilde \lambda_B$ (which is the Peterson-Woodward lifting of $\tilde \lambda_{\tilde P}$) are all  in $Q^\vee$, and they may be distinct with each other in general.
  \end{enumerate}

\subsection{Proofs of Proposition \ref{propwelldefinedgrading} and Proposition \ref{propinjectionofvectsp}}\label{subsecproofpropvec}
In analogy with \cite{leungli33}, we introduce the following notion with respect to the given pair $(\Delta, \Delta_P)$.
\begin{defn}
 An element $\lambda\in Q^\vee$ is called a \textbf{virtual null coroot}, if  $\langle\alpha, \lambda\rangle=0$ for all $\alpha\in \Delta_P$. An element $\mu_P$ in $Q^\vee/Q^\vee_P$ is called a  \textbf{virtual null coroot}, if   its Peterson-Woodward lifting $\mu_B\in Q^\vee$ is a virtual null coroot.
\end{defn}
By the definition of $gr$, every virtual null coroot $\lambda$ satisfies
$$gr_{[1, r]}(q_\lambda)=\mathbf{0}.$$

\begin{example}\label{exampnullcoroot}
   Suppose   $\alpha\in \Delta$ satisfies  that $Dyn(\{\alpha\}\cup\Delta_P)$    is disconnected. Then $\alpha\in\Delta\setminus \Delta_P$, and $\alpha^\vee$ is a virtual null coroot. Furthermore, for $\lambda_P:=\alpha^\vee+Q^\vee_P\in Q^\vee/Q^\vee_P$, we have
    $\lambda_B=\alpha^\vee$. Therefore $\lambda_P$ is also a   virtual null coroot.
\end{example}

 \begin{lemma}\label{lemmanullcoroot}
  Given  $\lambda_P, \mu_P\in Q^\vee/Q_P^\vee$, we denote $\kappa_P:=\lambda_P+\mu_P$.
    If $\mu_P$ is a virtual null coroot, then
       we have   $\kappa_B= \lambda_B+ \mu_B$. Consequently,   $$\mathcal{L}:=\{\eta_P\in Q^\vee/Q^\vee_P~|~ \eta_P \mbox{ is a virtual null coroot}\}$$ is   a sublattice of $Q^\vee/Q^\vee_P$.

        \end{lemma}
\begin{proof}
   Clearly, $\kappa_P=\kappa_B+Q^\vee_P$, and  we have
                $\langle \alpha, \kappa_B\rangle=\langle \alpha, \lambda_B\rangle  \in \{0, -1\}$ for all $\alpha\in R^+_P$. Thus the  statement follows from the uniqueness of the lifting.
  \end{proof}

We will let $\mathcal{L}_B$ denote the set of virtual null coroots in $Q^\vee$:
  $$\mathcal{L}_B:=\{\lambda\in Q^\vee~|~ \langle \alpha, \lambda\rangle=0, \forall \alpha\in\Delta_P\}.$$
Denote by $\Lambda^\vee$ the set of coweights of $G$ and by $\Lambda_P^\vee$ the set of coweights of the derived subgroup $(L, L)$ of the Levi factor $L$ of $P$. Denote by $\{\omega_{1, P}^\vee, \cdots, \omega_{r, P}^\vee\}$ the fundament coweights in $\Lambda_P^\vee$ dual to $\{\alpha_1, \cdots, \alpha_r\}$.
 Denote by $\partial \Delta_P$ the simple roots in $\Delta\setminus\Delta_P$ which are adjacent to $\Delta_P$.
 The next uniform description of the quotient  $\big(Q^\vee/Q^\vee_P\big)/\mathcal{L}$ is provided by the referee.
 \begin{prop}\label{propvirtualnullroot}
   The   quotient    $\big(Q^\vee/Q^\vee_P\big)/\mathcal{L}$ is isomorphic to the  subgroup of $\Lambda_P^\vee/Q_P^\vee$ generated by
     $$\{\omega_{i, P}^\vee~|~ \alpha_i \mbox{ is adjacent to }\partial \Delta_P\}.$$
\end{prop}
\begin{proof}
  Recall that $\Lambda_P^\vee$ is the set of integral valued linear form on $Q^\vee_P$. In particular, there is a natural morphism $R:Q^\vee\to \Lambda^\vee_P$
  obtained by restriction: $R(\lambda)=\lambda|_{Q^\vee_P}$. We have $R(Q^\vee_P)=Q^\vee_P$.

  Furthermore, this map factors through the quotient $Q^\vee/\mathcal{L}_B$ and the induced map is injective. In particular, the quotient
   $\big(Q^\vee/Q^\vee_P\big)/\mathcal{L}$ is isomorphic to the image $R(Q^\vee/Q_P^\vee)$. Since for $\alpha\in \Delta$ with $\Delta_P\cup \{\alpha\}$ disconnected we have $\alpha^\vee\in \mathcal{L}_B$, it follows that $R(Q^\vee/Q_P^\vee)$ is the subgroup of $\Lambda_P^\vee/Q_P^\vee$ generated by
    $R(\alpha^\vee)$ for $\alpha\in \partial \Delta_P$. But $R(\alpha^\vee)=-\omega_{i, P}^\vee$ for $\alpha_i$ adjacent to $\alpha$ and the result follows.
\end{proof}
\begin{remark}
  The group $\Lambda_P^\vee/Q_P^\vee$ is a finite abelian group. It  is the center of simply-connected cover of $(L, L)$, and is generated by the cominuscule coweights. One recovers this way the groups $\big(Q^\vee/Q^\vee_P\big)/\mathcal{L}$ in Table \ref{tabvirtualnullcoroot}.
\end{remark}

\renewcommand{\arraystretch}{1.2}
   {\upshape     \begin{table}[h]
  \caption{\label{tabvirtualnullcoroot}   
    }
\hspace{-0cm}
    \begin{tabular}{|c|c||c|c|}
     \hline
     \multicolumn{2}{|c||}{}  & $\mu_B$ & $\big(Q^\vee/Q^\vee_P\big)/\mathcal{L}$ \\  \hline \hline
      \multicolumn{2}{|c||}{\upshape C1$B$)} &$2\alpha_{r+1}^\vee+\big(\alpha_r^\vee+2\sum_{j=1}^{r-1}\alpha_j^\vee\big)$ &$\mathbb{Z}/2\mathbb{Z}$\\ \hline
         \multicolumn{2}{|c||}{\upshape C1$C$)} &$\alpha_{r+1}^\vee+\big(\sum_{j=1}^{r}\alpha_j^\vee\big)$ &$\{\mbox{id}\}$\\  \hline
      \multicolumn{2}{|c||}{\upshape C2)} & $2\alpha_{r+1}^\vee+\big(\alpha_r^\vee+\alpha_{r-1}^\vee+2\sum_{j=1}^{r-2}\alpha_j^\vee\big)$& $\mathbb{Z}/2\mathbb{Z}$\\ \hline
        & $r=6$ &   $3\alpha_7^\vee+\big(4\alpha_1^\vee+5\alpha_2^\vee+6\alpha_3^\vee+4\alpha_4^\vee+2\alpha_5^\vee+3\alpha_6^\vee\big)$ &  $\mathbb{Z}/3\mathbb{Z}$ \\ \cline{2-4}
         \raisebox{1.5ex}[0pt]{\upshape C4)}& $r=7$ &   $2\alpha_{8}^\vee+\big(3\alpha_1^\vee+4\alpha_2^\vee+5\alpha_3^\vee+6\alpha_4^\vee+4\alpha_5^\vee+2\alpha_6^\vee+3\alpha_7^\vee\big)$ &  $\mathbb{Z}/2\mathbb{Z}$ \\  \hline
        \multicolumn{2}{|c||}{\upshape C5)} & $4\alpha_6^\vee+\big(5\alpha_5^\vee+6\alpha_3^\vee+4\alpha_2^\vee+2\alpha_1^\vee+3\alpha_4^\vee\big)$& $\mathbb{Z}/4\mathbb{Z}$\\ \hline
       &  &   $2\alpha_5^\vee+\big(\alpha_1^\vee+2\alpha_2^\vee+ \alpha_3^\vee+ 2\alpha_4^\vee\big)$ &   \\ \cline{3-3}
        & \raisebox{1.5ex}[0pt]{$r=4$} &   $2\alpha_6^\vee+\big(2\alpha_1^\vee+2\alpha_2^\vee+ \alpha_3^\vee+ \alpha_4^\vee\big)$ & \raisebox{1.5ex}[0pt]{$\mathbb{Z}/2\mathbb{Z}\times \mathbb{Z}/2\mathbb{Z}$} \\ \cline{2-4}
          &  & $2\alpha_6^\vee+\alpha_7^\vee+\big(2\alpha_1^\vee+3\alpha_2^\vee+ 4\alpha_3^\vee+ 2\alpha_4^\vee+3\alpha_5^\vee\big)$  &   \\ \cline{3-3}
        & \raisebox{1.5ex}[0pt]{$r=5$} &    $2\alpha_7^\vee+\big(2\alpha_1^\vee+2\alpha_2^\vee+ 2\alpha_3^\vee+ \alpha_4^\vee+ \alpha_5^\vee\big)$ &  \raisebox{1.5ex}[0pt]{$\mathbb{Z}/4\mathbb{Z}$} \\ \cline{2-4}
       \raisebox{1.5ex}[0pt] {\upshape C7)}   &  &   $2\alpha_7^\vee+\big(\alpha_1^\vee+2\alpha_2^\vee+ 3\alpha_3^\vee+ 4\alpha_4^\vee+2\alpha_5^\vee+3\alpha^\vee_6\big)$ &   \\ \cline{3-3}
        & \raisebox{1.5ex}[0pt]{$r=6$} &   $2\alpha_8^\vee+\big(2\alpha_1^\vee+2\alpha_2^\vee+ 2\alpha_3^\vee+ 2\alpha_4^\vee+ \alpha_5^\vee+ \alpha^\vee_6\big)$ & \raisebox{1.5ex}[0pt]{$\mathbb{Z}/2\mathbb{Z}\times \mathbb{Z}/2\mathbb{Z}$} \\ \cline{2-4}
     &  {$r=7$} &   $4\alpha_8^\vee+\big(2\alpha_1^\vee+4\alpha_2^\vee+ 6\alpha_3^\vee+ 8\alpha_4^\vee+10\alpha_5^\vee+5\alpha^\vee_6+7\alpha_7^\vee\big)$ & {$\mathbb{Z}/4\mathbb{Z}$} \\ \hline
         &  &    $2\alpha_4^\vee+\big(2\alpha_1^\vee+ \alpha_2^\vee\big)$&  \\ \cline{3-3}
          & \raisebox{1.5ex}[0pt]{$r=2$} &  $\alpha_3^\vee+\big(\alpha_1^\vee+\alpha_2^\vee\big)$  &   \\ \cline{2-3}
    \raisebox{1.5ex}[0pt]{\upshape C9)}& $r=3$ &   $2\alpha_4^\vee+\big(2\alpha_1^\vee+4\alpha_2^\vee+ 3\alpha_3^\vee\big)$ & \raisebox{1.5ex}[0pt]{$\mathbb{Z}/2\mathbb{Z}$} \\  \cline{1-3}
         \multicolumn{2}{|c||}{\upshape C10)} &$2\alpha_4^\vee+\big(\alpha_1^\vee+2\alpha_2^\vee+ 3\alpha_3^\vee\big)$ &  \\ \hline
  \end{tabular}
   \end{table}
  }
  Recall that  each   monomial $q_{\lambda}=q_1^{a_1}\cdots q_n^{a_n}$ corresponds to a coroot $\lambda=\sum_{j=1}^na_j\alpha_j^\vee$.
Given a sequence $I=[i_1, i_2, \cdots, i_m]$,     we   simply   denote  by $s_{i_1i_2\cdots i_m}$  or  $s_I$  the product $s_{i_1}s_{i_2}\cdots s_{i_m}$, and define   $|I|:= m$.
\begin{prop}\label{propNONvirtualnullcoroot}
The virtual coroot lattice $\mathcal{L}_B$ is generated by the virtual null roots   $\mu_B\in Q^\vee$ given in Table \ref{tabvirtualnullcoroot}.

For each case in Table \ref{tabNONvirtualnullcoroot}, the corresponding   coroot $\lambda$ satisfies $\langle \alpha_k, \lambda\rangle =-1$ for the given  number $k$ in the table, and  $\langle \alpha_j, \lambda\rangle =0$ for all $j\in \{1,\cdots, r\}\setminus\{k\}$. 

  Furthermore, 
  we have    $\psi_{\Delta, \Delta_P}(q_{\lambda_P})=q_{\lambda}\sigma^{u}$ with $q_\lambda$ and $u$ being shown in Table \ref{tabNONvirtualnullcoroot} as well  (which implies $\lambda=\lambda_B$).
 In  particular,   each  $u$   is of the form $s_Is_{r-1}s_{r-2}\cdots s_1$,  $s_Is_J$ or  $s_I$  where   $I$ (resp. $J$) is a sequence of  integers ending with $r$ (resp. $r-1$) in the table.  The grading $gr(\sigma^u)$ is then given by $|I|\mathbf{e}_r+\sum_{i=1}^{r-1}\mathbf{e}_i$, $|I|\mathbf{e}_r+|J|\mathbf{e}_{r-1}$ and   $|J|\mathbf{e}_r$ respectively.

 \end{prop}
 \renewcommand{\arraystretch}{1.2}
  {\upshape     \begin{table}[h]
  \caption{\label{tabNONvirtualnullcoroot}   
    }
\hspace{-0.cm}
    \begin{tabular}{|c|c||c|c|c|}
     \hline
     \multicolumn{2}{|c||}{}  & $q_{\lambda}$ &  $u$ & 
       $k$ \\ \hline \hline
      \multicolumn{2}{|c||}{\upshape C1$B$)} & $q_{r+1}$ &$s_{12\cdots r}s_{r-1}s_{r-2}\cdots s_1$& 
       $1$ \\ \hline
         \multicolumn{2}{|c||}{\upshape C2)} & $q_{r+1}$ & $s_{12\cdots (r-2)r}s_{r-1}s_{r-2}\cdots s_1$& 
         $1$\\ \hline
            &  &$q_{7}$ &  $s_{54362132436}s_5s_4\cdots s_1$ & 
            $1$ \\ \cline{3-5}
      {\upshape C4)}   & \raisebox{1.5ex}[0pt]{$r=6$} &$q_{7}^2q_1^2q_2^2q_3^2q_4q_6$ & $s_{12346325436}s_{12345}$& 
        $5$\\ \cline{2-5}
       & $r=7$ & $q_{8}$ & $s_{123475436547234512347}s_{6}s_5s_4s_3s_2s_1$  &  
        $1$ \\ \hline
         \multicolumn{2}{|c||}{}  &$q_{6}$ &  $s_{4352132435}$ & 
         $5$\\ \cline{3-5}
      \multicolumn{2}{|c||}{\upshape C5)}    &$q_{6}^2q_5^2q_3^2q_4q_2$ & $s_{1235}s_{4}s_3s_2 s_1$& 
       $1$\\ \cline{3-5}
       \multicolumn{2}{|c||}{}  & $q_{6}^3q_5^3q_3^3q_4q_2^2q_1$ & $s_{532435}s_{1234}$  & 
       $4$\\ \hline
         &  & $q_5$ &  $s_{423124}$ &   
         $4$\\ \cline{3-5}
        &  {$r=4$} & $q_6$ & $s_{124}s_3s_2s_1$ &  
        $1$\\ \cline{3-5}
       &  & $q_5q_6q_1q_2q_4$ & $s_{324}s_{123}$ & 
       $3$ \\ \cline{2-5}
     &  & $q_6$ &  $s_{4352134235}$ &  
     $5$ \\ \cline{3-5}
        &  {$r=5$} & $q_7$ & $s_{1235}s_4s_3s_2s_1$ & 
        $1$\\ \cline{3-5}
       &  & $q_6q_7q_1q_2q_3q_5$ & $s_{534235}s_{1234}$ &  
       $4$  \\ \cline{2-5}
 \raisebox{1.5ex}[0pt]{\upshape C7)}&  & $q_7$ &  $s_{645342132643546}$ &  
 $6$ \\ \cline{3-5}
        &  {$r=6$} & $q_8$ & $s_{12346}s_5s_4s_3s_2s_1$ & 
        $1$ \\ \cline{3-5}
       &  & $q_7q_8q_1q_2q_3q_4q_6$ & $s_{5463243546}s_{12345}$ &  
       $5$\\ \cline{2-5}
 &  & $q_8$ &  $s_{657456345723456123457}$ &   
 $7$ \\ \cline{3-5}
        &  {$r=7$} & $q_8^2q_2q_3^2q_4^3q_5^4q_6^2q_7^3$ & $s_{123457}s_6s_5s_4s_3s_2s_1$ &
        $1$\\ \cline{3-5}
       &  & $q_8^3q_1q_2^2q_3^3q_4^4q_5^5q_6^2q_7^4$ & $ s_{756457345623457}s_{123456}$ & 
       $6$ \\ \hline
            &  {$r=2$} & $q_4$ &   $s_{12}s_1$ & 
            $1$  \\ \cline{2-5}
    \raisebox{1.5ex}[0pt]{\upshape C9)}& $r=3$ &$q_4q_2q_3$ & $s_{123}s_2s_1$ &  
    $1$  \\  \cline{1-5}
         \multicolumn{2}{|c||}{\upshape C10)} & $q_4$ &$s_{323123}$ &
         $3$\\ \hline
 \end{tabular}
   \end{table}
 }

 \begin{proof}
     Assume that case   {\upshape C1$B)$} occurs, then   we have a unique  $\mu_B=2\alpha_{r+1}^\vee+(\alpha_r^\vee+2\sum_{j=1}^{r-1}\alpha_j^\vee)$ and a unique $\lambda=\alpha_{r+1}^\vee$ in the tables. Clearly, $\langle \alpha, \mu_B\rangle=0$ for all $\alpha\in \Delta_P$.  Thus $\mu_B$ is a virtual null coroot, and it is
      the   expected Peterson-Woodward lifting of $\mu_P:=\mu_B+Q^\vee_P=2\alpha_{r+1}^\vee+Q^\vee_P$. Hence, $\mu_P$ is a virtual null coroot in $Q^\vee/Q_P^\vee$ by definition.
   It follows from
       Example \ref{exampnullcoroot} that all the elements in the sublattice $\mathcal{L}'$ generated by $\{\mu_P\}\cup\{\alpha\in \Delta~|~ Dyn(\{\alpha\}\cup \Delta_P) \mbox{ is disconnected}\}$ are  virtual null coroots.
       Clearly, $\big(Q^\vee/Q^\vee_P\big)/\mathcal{L}'\cong \mathbb{Z}/2\mathbb{Z}$.
       Since $\mathcal{L}'\subset\mathcal{L}\subset Q^\vee/Q^\vee_P$, we have a surjective morphism
         $\big(Q^\vee/Q^\vee_P\big)/\mathcal{L}'\to   \big(Q^\vee/Q^\vee_P\big)/\mathcal{L}\cong \mathbb{Z}/2\mathbb{Z}$. Hence, this is an isomorphism, and    $\mathcal{L}=\mathcal{L}'$.

   It is clear that for $k=1$, we have $\langle \alpha_k, \lambda\rangle=-1$ and $\langle \alpha_j, \lambda\rangle =0$ for all $j\in\{1, \cdots, r\}\setminus\{k\}$. Note that $\Delta_P$ is of $B_r$-type, and that any positive root $\gamma\in R_P^+$ is of the from $\varepsilon\alpha_1+\sum_{j=2}^rc_j\alpha_j$ where $\varepsilon\in\{0, 1\}$ (see e.g. \cite{bour}). Thus $\langle \gamma, \lambda\rangle \in\{0, -1\}$.
   Hence, $\lambda=\lambda_B$ is the Peterson-Woodward lifting of $\lambda_P$. Consequently, $\lambda_P$ is not a virtual null coroot, as $\lambda_B$ is not. That is, our claim holds.

    By definition,   $\psi_{\Delta, \Delta_P}(q_{\lambda_P})=q_{\lambda}\sigma^{w_Pw_{P'}}$ where $\Delta_{P'}=\Delta_P\setminus \{\alpha_k\}$ in this case. Note that
  $w_Pw_{P'}$ is the unique element of maximal length in $W^{P'}_P$, whose length is equal to $|R_P^+|-|R_{P'}^+|$.  In order to show $u=s_1s_2\cdots s_r \cdot s_{r-1}s_{r-2}\cdots s_1$ coincides with $w_Pw_{P'}$, it suffices to show: (1) the above expression of $u$ is reduced of expected length; (2) $u\in W^{P'}_P$, i.e., $u(\alpha)\in R^+$ for all $\alpha\in \Delta_{P'}$. 
 Indeed, in the case of  C1$B$),   $\Delta_{P'}=\{\alpha_2, \cdots, \alpha_r\}$ is of $B_{r-1}$-type.   For $1\leq j\leq r$,   $s_1\cdots s_{j-1}(\alpha_j)=\alpha_1+\cdots+\alpha_j\in R^+$. For
   $r-1\geq i\geq 1$,  $s_1\cdots s_{r}s_{r-1}\cdots s_{i+1}(\alpha_i)=\alpha_1+\cdots+\alpha_i+2\alpha_{i+1}+\cdots +2\alpha_r\in R^+$. Thus the expression of $u$ is reduced, and $\ell(u)=r+r-1= r^2-(r-1)^2= |R_P^+|-|R_{P'}^+|$.  For all $2\leq j\leq r$, we note $u(\alpha_j)=\alpha_j\in R^+$. Therefore both (1) and (2) hold. 

The expression   $u=s_Is_{r-1}s_{r-2}\cdots s_1$, where   $I=[1, 2, \cdots, r]$,    is reduced. Thus the subexpression $s_I=s_1\cdots s_r$ is also reduced.  Clearly, $s_j\in W^{P_{j-1}}_{P_j}$, and $s_I(\alpha_j)\in R^+$ for all $\alpha_j\in \Delta_{P_{r-1}}=\{\alpha_1, \cdots, \alpha_{r-1}\}$, which implies $s_I\in   W^{P_{r-1}}_P$.
 Hence,  $gr(\sigma^u)=\ell(s_I)\mathbf{e}_r+\sum_{i=1}^{r-1}\ell(s_i)\mathbf{e}_i=|I|\mathbf{e}_r+\sum_{i=1}^{r-1}\mathbf{e}_i$.

The arguments for the remaining cases are all the same.
  \end{proof}

\begin{remark}
  We obtain both tables using case by case analysis, which gives an alternative  proof of Proposition \ref{propvirtualnullroot}  by studying the quotient $(Q^\vee/Q^\vee_P)/\mathcal{L}'$ first.
\end{remark}

\begin{lemma} \label{lemmagrading33cases}
 Let  $\lambda_P, \mu_P\in Q^\vee/Q_P^\vee$.  Write   $\psi_{\Delta, \Delta_P}(q_{\lambda_P})=q_{\lambda_B}\sigma^{u}$.
 If $\mu_P$ is a virtual null coroot, then we have
   $$\psi_{\Delta, \Delta_P}(q_{\mu_P})=q_{\mu_B} \quad\mbox{and}\quad \psi_{\Delta, \Delta_P}(q_{\lambda_P+\mu_P})=q_{\lambda_B+\mu_B}\sigma^{u}.$$
Consequently, we have
  $$gr_{[1, r]}(\psi_{\Delta, \Delta_P}(q_{\mu_P}))=\mathbf{0}\quad \mbox{and}\quad gr_{[1, r]}(\psi_{\Delta, \Delta_P}(q_{\lambda_P+\mu_P}))= gr_{[1, r]}(\psi_{\Delta, \Delta_P}(q_{\lambda_P})).$$
\end{lemma}

\begin{proof}
  For $\kappa_P\in Q^\vee/Q^\vee_P$, by definition we have  $\psi_{\Delta, \Delta_P}(q_{\kappa_P})=q_{\kappa_B}\sigma^{w_Pw_{P'}}$ with $\Delta_{P'}=\{\alpha\in \Delta_P~|~\langle \alpha, \kappa_B\rangle=0\}$.
  If $\kappa_P=\mu_P$, then   $\Delta_{P'}=\Delta_{P}$ since $\mu_P$ is a virtual null coroot. Thus $w_Pw_{P'}=\mbox{id}$ and consequently   $\psi_{\Delta, \Delta_P}(q_{\mu_P})=q_{\mu_B}$.
   If $\kappa_P=\lambda_P+\mu_P$, then $\kappa_B=\lambda_B+\mu_B$. Write
    $\Delta_{P'}=$ $\{\alpha\in \Delta_P~|~\langle \alpha, \kappa_B\rangle=0\}=\{\alpha\in \Delta_P~|~\langle \alpha, \lambda_B\rangle=0\}$. That is, we have  $u=w_Pw_{P'}$ and
        $\psi_{\Delta, \Delta_P}(q_{\kappa_P})=q_{\kappa_B}\sigma^{u}$. The two identities on the gradings are then a direct consequence.
 \end{proof}

 \bigskip

 \begin{proof}[Proof of Proposition \ref{propwelldefinedgrading}]
 The initial proof used case by case analysis with Table \ref{tabNONvirtualnullcoroot}.
 Here we provide a uniform proof from the referee.

  To prove the statement using Lemma \ref{lemmagrading33cases}, we only need to prove that
 $$gr_{[1, r]}(\psi_{\Delta,\Delta_P}(q_{\alpha^\vee+Q^\vee_P}))=\mathbf{0}$$
 for $\alpha\in \partial \Delta_P$.

 Let $\alpha\in\partial \Delta_P$ and let $\alpha_i$ be the unique element in $\Delta_P$ adjacent to $\alpha$. We have $\langle \alpha_i, \alpha^\vee\rangle\neq 0$. By definition we have $gr_{[1, r]}(\psi_{\Delta,\Delta_P}(q_{\alpha^\vee+Q^\vee_P}))=gr_{[1, r]}(w_P^{P'}, \alpha^\vee)$ and
 $$gr_{[1, r]}(w_P^{P'}, \alpha^\vee)=\sum_{j=1}^r\Big(|\mbox{Inv}(w_P^{P'})\cap (R_j^+\setminus R_{j-1}^+)|+\sum_{\beta\in R_j^+\setminus R_{j-1}^+}\langle \beta, \alpha^\vee\rangle\Big)\mathbf{e}_j.$$
We first remark that the above grading does only depend on the restriction of $\alpha^\vee$ to $\Delta_P$ so on $R(\alpha^\vee)=-\omega_{i, P}^\vee$ as defined in the proof of Proposition \ref{propvirtualnullroot}. For $w\in W_P$ and $\lambda\in \Lambda^\vee_P$ we define
 $$gr_{[1, r]}(w, \lambda)=\sum_{j=1}^r\Big(|\mbox{Inv}(w_P^{P'})\cap (R_j^+\setminus R_{j-1}^+)|+\sum_{\beta\in R_j^+\setminus R_{j-1}^+}\langle \beta, \lambda\rangle\Big)\mathbf{e}_j.$$
For $\mathbf{a}=\sum_{j=1}^r a_j\mathbf{e}_j$, define $$\|\mathbf{a}\|:=\sum_{j=1}^r|a_j|.$$
Next remak (see Corollary 3.13 of \cite{CMP-affinesymm}) that for $w\in W_P$ and $\lambda\in \Lambda_P^\vee$, we have
   $$\ell(wt_\lambda)=\|gr_{[1, r]}(w, \lambda)\|$$
where $\ell$ denotes the length function on $\widetilde{W}_{\scriptsize\mbox{aff}}$ the extended affine Weyl group and where we consider the element $wt_\lambda$ as an element of the extended affine Weyl group $\widetilde{W}_{\scriptsize\mbox{aff}}$ (see Definition 3.9 of \cite{CMP-affinesymm}).

Now for $P'$ defined by $\Delta_{P'}=\{\beta\in \Delta_P~|~ \langle \beta, \omega_{i, P}^\vee\rangle =0\}$, the element
  $$\tau_i:=w_P^{P'}t_{-\omega_{i, P}^\vee}$$
is the element $\tau_i$ defined on page 9 of \cite{CMP-affinesymm}. In particular this element satisfies $\ell(\tau_i)=0$ (since this element is in the stabiliser of the fundamental alcove, see also page 5 of \cite{lamshi}). As a consequence we  get
$$\|gr_{[1, r]}(\psi_{\Delta,\Delta_P}(q_{\alpha^\vee+Q^\vee_P}))\|=0\quad\mbox{and}\quad gr_{[1, r]}(\psi_{\Delta,\Delta_P}(q_{\alpha^\vee+Q^\vee_P}))=\mathbf{0}.$$
 \end{proof}

 To prove  Proposition \ref{propinjectionofvectsp}, we  need the next lemma. 
\begin{lemma}[Lemma 4.1 (1) of \cite{leungli33}]\label{lemmauniquegradinginPB} For any $\mathbf{d}=\sum_{i=1}^{r-1}d_i\mathbf{e}_i\in \mathbb{Z}^{r-1}\times \{(0, 0)\}\subset\mathbb{Z}^{r+1}$, there exists a unique $(w, \eta)\in W\times Q^\vee$ such that $gr(q_{\eta}\sigma^w)=\mathbf{d}$.
\end{lemma}

\bigskip
\begin{proof}[Proof of Proposition \ref{propinjectionofvectsp}]
  Since $\psi_{\Delta, \Delta_P} $ is injective, so is $\Psi_{r+1}$.

 For a nonzero element $\overline{q_\mu\sigma^w}\in Gr_{(r+1)}^{\mathcal{F}}$, we write $w=vu$ where $v\in W^P$ and $u\in W_P$, and write  $\mu=\mu'+\mu''$, where $\mu'\in \bigoplus_{i=r+1}^n\mathbb{Z}_{\geq 0}\alpha_i^\vee$ and $\mu''\in \bigoplus_{i=1}^r\mathbb{Z}_{\geq 0}\alpha_i^\vee$. Note $gr_{[1, r]}(q_{\mu'}\sigma^v)\in \bigoplus_{i=1}^r\mathbb{Z}_{\leq 0}\mathbf{e}_i$ and  $\mathbf{0}=gr_{[1, r]}(q_\mu \sigma^w)=gr_{[1, r]}(q_{\mu''}\sigma^u)+gr_{[1, r]}(q_{\mu'}\sigma^v)$. Thus  we have $gr_{[1, r]}(q_{\mu''}\sigma^u)\in \bigoplus_{i=1}^r\mathbb{Z}_{\geq 0}\mathbf{e}_i$.
  Setting
    $\lambda_P:=\mu+Q^\vee_P$, we have   $\psi_{\Delta, \Delta_P}(q_{\lambda_P}\sigma^v)=q_{\lambda_B}\sigma^{vw_Pw_{P'}}$ with $\lambda_B=\mu'+\lambda''$ for some $\lambda''\in   \bigoplus_{i=1}^r\mathbb{Z}_{\geq 0}\alpha_i^\vee$.
    Note  $gr_{[1, r]}(q_{\lambda_B} \sigma^{vw_Pw_{P'}})=gr_{[1, r]}(q_{\lambda''}\sigma^{w_Pw_{P'}})+gr_{[1, r]}(q_{\mu'}\sigma^v)$, and it is equal to $\mathbf{0}$ by Proposition \ref{propwelldefinedgrading}. Hence,  $\mathbf{d}:=gr_{[1, r]}(q_{\lambda''}\sigma^{w_Pw_{P'}})=gr_{[1, r]}(q_{\mu''}\sigma^u)\in   \bigoplus_{i=1}^r\mathbb{Z}_{\geq 0}\mathbf{e}_i$. Thus the map $\Psi_{r+1}$ is surjective as soon as there is a unique element
    $q_{\mu''}\sigma^u$ of grading $\mathbf{d}$.

  Suppose  $r=2$ and either case C1$B$) or case C9) occurs. Note   $gr(q_1)=2\mathbf{e}_1$, $gr(q_2)=4\mathbf{e}_2-2\mathbf{e}_1$, and
       $u=u_2u_1$ for a unique $u_1\in W_{P_1}=\{1, s_1\}$ and $u_2\in  W_P^{P_1}=\{1, s_2, s_1s_2, s_2s_1s_2\}$.
    Note for given $d_1, d_2\geq 0$, the next equalities  $$d_1\mathbf{e}_1+d_2\mathbf{e}_2=gr_{[1, r]}(q_{a_1\alpha_1^\vee+a_2\alpha_2^\vee}\sigma^u)=a_1 \cdot gr(q_1)+a_2\cdot gr(q_2)+ \ell(u_1) \mathbf{e}_1+\ell(u_2)\mathbf{e}_2$$
    determine  a unique $(a_1, a_2, \ell(u_1), \ell(u_2))\in \mathbb{Z}_{\geq 0}\times \mathbb{Z}_{\geq 0}\times \{0, 1\}\times \{0, 1, 2, 3\}$. The pair $(\ell(u_1), \ell(u_2))$ further determines a unique $(u_1, u_2)\in W_{P_1}\times W^{P_1}_{P}$. Hence, $q_{\mu'' } \sigma^u=q_{\lambda''}\sigma^{w_Pw_{P'}}$ follows from the uniqueness.

   In order to show $\Psi_{r+1}$ is not surjective for the remaining cases, it suffices to consider the virtual null roots $\mu_P$ in Proposition \ref{propvirtualnullroot}, for which we note $\Psi_{r+1}(q_{\mu_P})=\overline{q_{\mu_B}}$. The point is to show $gr_{[r, r]}(q_r)\leq  \ell(w_Pw_{P_{r-1}})\mathbf{e}_r$. Once this is done, we show the existence of $q_{r}^{a_r-1}\sigma^{u_r}$ satisfying $gr_{[r, r]}(q_{r}^{a_r-1}\sigma^{u_r})=gr_{[r, r]}(q_{\lambda_B}\sigma^{w_Pw_{P'}})$, where $u_r\in W_{P}^{P_{r-1}}$, and $a_r$ denotes the power of $q_r$ in the monomial $q_{\lambda_B}$. Then we apply Lemma \ref{lemmauniquegradinginPB} to construct  an element in $W_{P_{r-1}}\times \bigoplus_{i=1}^{r-1}\mathbb{Z}_{\geq 0}\alpha_{i}^\vee$ with  grading   $gr_{[1, r-1]}(q_{\lambda_B}\sigma^{w_Pw_{P'}})-gr_{[1, r-1]}(q_{r}^{a_r-1}\sigma^{u_r})$. In this way, we obtain an element of the  same grading as $gr(q_{\lambda_B}\sigma^{w_Pw_{P'}})$ that is not in the image of  $ \Psi_{r+1}$. Precise arguments are given as follows.

     For case C1$C$), we have  $\mu_P=\alpha_{r+1}^\vee+Q_P^\vee$ and $\mu_B=\alpha_1^\vee+\cdots +\alpha_{r+1}^\vee$. Note
       $gr_{[r, r]}(q_{r})=(r+1)\mathbf{e}_r$, and $w_Pw_{P_{r-1}}$  is the longest element in $W^{P_{r-1}}_{P}$, which is of length $\ell(w_Pw_{P_{r-1}})=|R_{P}^+|-|R_{P_{r-1}}^+|=r^2- {(r-1)r\over2}={r^2+r\over 2}\geq r+1$. Hence, there exists $u_r\in W^{P_{r-1}}_{P}$ of length $r+1$.
       Note for each $1\leq j\leq r-1$, $u_j:=s_j\in W_{P_j}^{P_{j-1}}$. Thus
        $gr(q_{r+1}\sigma^{u_r\cdots u_2s_1})=(2r+2)\mathbf{e}_{r+1}=gr(q_1\cdots q_{r+1})$. However, $q_{r+1}\sigma^{u_r\cdots u_2u_1}\not\in \psi_{\Delta, \Delta_P}(QH^*(G/P))$.

     For case C1$B$) with $r\geq 3$, we have  $\mu_P=2\alpha_{r+1}^\vee+Q_P^\vee$ and $\mu_B=2\alpha_{r+1}^\vee+\alpha_r^\vee+2\sum_{j=1}^{r-1}\alpha_j^\vee$.  Note
       $gr_{[r, r]}(q_{r})=2r\mathbf{e}_r$, and  $\ell(w_Pw_{P_{r-1}})= {r(r+1)\over 2}\geq 2r$. Hence, there exists $u_r\in W^{P_{r-1}}_{P}$ of length $2r$.  Set $u_j=s_{j-1}s_j$ for $2\leq j\leq r-1$.  Then
      $gr_{[1, r]}(q_{r+1}^2q_1\sigma^{u_ru_{r-1}\cdots u_2})=\mathbf{0}$ so that $\overline{q_{r+1}^2q_1\sigma^{u_ru_{r-1}\cdots u_2}}\in Gr_{(r+1)}^\mathcal{F}$.
    However, $q_{r+1}^2q_1\sigma^{u_ru_{r-1}\cdots u_2}\not\in \psi_{\Delta, \Delta_P}(QH^*(G/P))$.

    The arguments for the remaining cases are also easy and similar.
\end{proof}

\subsection{Proof of Proposition \ref{propmorphismofalgebras} (1)}\label{subsectprop111} For $v', v''\in W^P$,    we note
$$\Psi_{r+1}(\sigma^{v'})\star \Psi_{r+1}(\sigma^{v''})-\Psi_{r+1}(\sigma^{v'}\star_P\sigma^{v''})=
  \sum N_{v', v''}^{w, \lambda} \overline{q_{\lambda}\sigma^{w}},$$ the summation over those $q_{\lambda}\sigma^w\in QH^*(G/B)$ satisfying
     $gr_{[1, r]}(q_\lambda \sigma^w)=\mathbf{0}$ and ${q_{\lambda}\sigma^{w}}\not\in \psi_{\Delta, \Delta_P}(QH^*(G/P))$. It suffices  to show the vanishing of all the coefficients $N_{v', v''}^{w, \lambda}$   (if any). In particular, it is already done,  if $\Psi_{r+1}$ is an isomorphism of vector spaces.
     Therefore,   if $r=2$, then  both      C1$B$) and C9) could be excluded in  the rest of this subsection.

To do this, we will  use the same idea occurring in section 3.5 of \cite{leungli33}. Namely, we consider the fibration $G/B\rightarrow G/\tilde{P}$ where    $\Delta_{\tilde P}=\{\alpha_1, \cdots, \alpha_{r-1}\}$.
Set $\varsigma:=r-1$, and note that $\Delta_{\tilde P}$ is of $A$-type satisfying the assumption on the ordering as in \cite{leungli33}.    Using Definition \ref{defnofgrading} with respect to $(\Delta,  \Delta_{\tilde P})$, we have a grading map     $$\tilde {gr}: W\times Q^\vee\longrightarrow \mathbb{Z}^{\varsigma+1}=\bigoplus\nolimits_{i=1}^r\mathbb{Z}\mathbf{e}_i\hookrightarrow \mathbb{Z}^{r+1},$$
which satisfies the next obvious property  $$\tilde{gr}_{[1, r-1]}={gr}_{[1, r-1]}.$$
  Consequently, we obtain a filtration $\tilde{\mathcal{F}}$ on  $QH^*(G/B)$ and a (well-defined) induced map $\tilde \Psi_{\varsigma+1}: QH^*(G/\tilde P)\rightarrow Gr^{\tilde{\mathcal{F}}}_{(\varsigma+1)}\subset  Gr^{\tilde {\mathcal{F}}}(QH^*(G/B))$ as well. Furthermore, all the results of \cite{leungli33} hold  with respect to the fibration $G/B\rightarrow G/\tilde{P}$. In particular, we have the next proposition (which follows immediately from Theorem 1.6 of \cite{leungli33}).

\begin{prop}\label{propreductiontor-1}
 Let $\tilde u, \tilde v\in W^{\tilde P}$ and  $\tilde w\in W_{\tilde P}$. In $Gr^{\tilde {\mathcal{F}}}(QH^*(G/B))$, we have
   {\upshape $$(1) \,\,\, \overline{\sigma^{\tilde u}}\star \overline{\sigma^{\tilde v}}=\overline{\psi_{\Delta, \Delta_{\tilde P}}(\sigma^{\tilde u}\star_{\tilde P}\sigma^{\tilde v})}\,\,;\qquad (2)\,\,\,
         \overline{\sigma^{\tilde u}}\star  \overline{\sigma^{\tilde w}}=\overline{\sigma^{\tilde u\tilde w}}.$$
}
 \end{prop}

\begin{lemma}\label{lemmagradingreduction11}
   For any $u, v\in W^P$, we have in $QH^*(G/B)$ that
       $$\sigma^{u}\star \sigma^v=\sum N_{u, v}^{w, \lambda}q_\lambda\sigma^{w}+\sum N_{u, v}^{w', \lambda'}q_{\lambda'}\sigma^{w'}+ \sum N_{u, v}^{w'', \lambda''}q_{\lambda''}\sigma^{w''},$$
     where the first summation is over those $q_\lambda\sigma^w\in \psi_{\Delta, \Delta_P}(QH^*(G/P))$,  the second summation is over those
     $q_{\lambda'}\sigma^{w'}\in \psi_{\Delta, \Delta_{\tilde P}}(QH^*(G/\tilde P))\setminus \psi_{\Delta, \Delta_P}(QH^*(G/P))$, and the third summation is over those $q_{\lambda''}\sigma^{w''}$ satisfying  ${gr_{[1, r-1]}(q_{\lambda''}\sigma^{w''})<\mathbf{0}}$.
\end{lemma}
\begin{proof}
    Since $\Delta_{\tilde P}\subset \Delta_P$, we have $u, v\in W^{\tilde P}$. By Proposition \ref{propreductiontor-1}(1), we have
    $$\sigma^{u}\star \sigma^v=\sum_{q_\lambda\sigma^w\in \psi_{\Delta, \Delta_{\tilde P}}(QH^*(G/\tilde P))} N_{u, v}^{w, \lambda}q_\lambda\sigma^{w}+ \sum_{\tilde{gr}_{[1, r-1]}(q_{\lambda''}\sigma^{w''})<\mathbf{0}} N_{u, v}^{w'', \lambda''}q_{\lambda''}\sigma^{w''}.$$
     If $q_\lambda\sigma^w\in \psi_{\Delta, \Delta_P}QH^*(G/P)$, then $\lambda=\lambda_B$ is the Peterson-Woodward lifting of $\lambda_P:=\lambda+Q^\vee_P$ and $w=w_1w_Pw_{P'}$ with $w_1$ being the minimal length representative of the coset $wW_P$. Since $R^+_{\tilde P}\subset R^+_P$, $\lambda_B=\tilde \lambda_B$ is also the lifting of $\tilde\lambda_{\tilde P}:=\lambda+Q^\vee_{\tilde P}$. Note that
       $\Delta_{\tilde P'}=\{\alpha\in \Delta_{\tilde P}~|~\langle \alpha, \tilde\lambda_B\rangle=0\}\subset \{\alpha\in \Delta_{P}~|~\langle \alpha, \lambda_B\rangle=0\}=\Delta_{P'}$.
        Thus $\mbox{Inv}(w_Pw_{P'})=R^+_P\setminus R^+_{P'}=\big(R_{\tilde P}^+\setminus R_{\tilde P'}^+\big)\bigsqcup \big((R_{P}^+\setminus R_{\tilde P}^+)\setminus (R_{P'}\setminus R^+_{\tilde P'})\big)$.
      Hence, we have $w_Pw_{P'}=w_2w_{\tilde P}w_{\tilde P'}$ where $w_2$ is the minimal length representative of the coset   $w_Pw_{P'}W_{\tilde P}$ \big(for which we have $\mbox{Inv}(w_2)=(R_{P}^+\setminus R_{\tilde P}^+)\setminus (R_{P'}\setminus R^+_{\tilde P'}))$. Note $w_1w_2\in W^{\tilde P}$. Thus we have $q_\lambda\sigma^w=\psi_{\Delta, \Delta_{\tilde P}}(q_{\tilde\lambda_{\tilde P}}\sigma^{w_1w_2})\in \psi_{\Delta, \Delta_{\tilde P}}(QH^*(G/\tilde P))$. Therefore the statement follows by noting   $\tilde {gr}_{[1, r-1]}=gr_{[1, r-1]}$.
\end{proof}

Due to the above lemma,  it remains to show  that  for any element  $q_{\lambda'}\sigma^{w'}$ in $\psi_{\Delta, \Delta_{\tilde P}}(QH^*(G/P_{r-1}))\setminus \psi_{\Delta, \Delta_P}(QH^*(G/P))$,
          either $N_{u, v}^{w', \lambda'}=0$ or $gr(q_{\lambda'}\sigma^{w'})<0$ holds. The latter claim could be further simplified as  $gr_{[r, r]}(q_{\lambda'}\sigma^{w'})<\mathbf{0}$, by noting $gr_{[1, r-1]}(q_{\lambda'}\sigma^{w'})=\mathbf{0}$. For this purpose, we need the next main result of \cite{leungliQtoC}, which is in fact an application of \cite{leungli33} in the special case of $P/B\cong \mathbb{P}^1$.
  For each $\alpha\in \Delta$, we
 define a map $\mbox{sgn}_\alpha: W\rightarrow \{0, 1\}$   by $\mbox{sgn}_\alpha(w):=1$ if
  $\ell(w)-\ell(ws_\alpha)>0$, and $0$ otherwise.

\begin{prop}[Theorem 1.1 of \cite{leungliQtoC}]\label{propQtoCthm}
   Given $u, v, w\in W$ and  $\lambda\in Q^\vee$, we have
   \begin{enumerate}
      \item   $N_{u, v}^{w, \lambda}=0$  unless  {\upshape $\mbox{sgn}_\alpha(w)+\langle \alpha, \lambda\rangle \leq \mbox{sgn}_\alpha(u)+\mbox{sgn}_\alpha(v)$} for all $\alpha\in \Delta.$
     \item  Suppose    {\upshape $ \mbox{sgn}_\alpha(w)+\langle \alpha, \lambda\rangle =\mbox{sgn}_\alpha(u)+\mbox{sgn}_\alpha(v)=2$} for some  $\alpha\in \Delta$, then
           {\upshape  $$N_{u, v}^{w, \lambda}=N_{us_\alpha, vs_\alpha}^{w, \lambda-\alpha^\vee}=
                 \begin{cases} N_{u, vs_\alpha}^{ws_\alpha, \lambda-\alpha^\vee}, &   i\!f    \mbox{ sgn}_\alpha(w)=0 \\
                               \vspace{-0.3cm}   & \\
                         N_{u, vs_\alpha}^{ws_\alpha, \lambda}, & i\!f   \mbox{ sgn}_\alpha(w)=1 \,\, {}_{\displaystyle .}  \end{cases}$$
            }
   \end{enumerate}

\end{prop}

\begin{cor}\label{corcoefforprop11}
  Let $u, v\in W^P$. Suppose    $N_{u, v}^{w, \lambda}\neq 0$ for some  $w\in W$ and $\lambda\in Q^\vee$.  Then we have
  \begin{enumerate}
    \item   $\langle\alpha, \lambda\rangle\leq 0$ for all $\alpha\in \Delta_P$;
    \item  Set $\lambda_P:=\lambda+Q^\vee_P$ and denote by $w_1$ the minimal length representative of the coset $wW_P$.    If $\lambda=\lambda_B$ and $gr_{[1, r]}(q_\lambda \sigma^w)=\mathbf{0}$, then $q_\lambda \sigma^w=\psi_{\Delta, \Delta_P}(q_{\lambda_P}\sigma^{w_1})$.
  \end{enumerate}
\end{cor}

\begin{proof}
  Assume $\langle\alpha, \lambda\rangle>0 $ for some $\alpha\in \Delta_P$, then we have
    $\mbox{sgn}_\alpha(u)+\mbox{sgn}_\alpha(v)=0<\langle\alpha, \lambda\rangle\leq \mbox{sgn}_\alpha(w)+\langle\alpha, \lambda\rangle$. Thus $N_{u, v}^{w, \lambda}=0$ by Proposition \ref{propQtoCthm} (1), contradicting with the hypothesis.

     Since $N_{u, v}^{w, \lambda}\neq 0$,    we have $\mbox{sgn}_\alpha(w)=0$ for any $\alpha\in \Delta_{P'}=\{\beta\in \Delta_P~|~\langle \beta, \lambda_B\rangle=0\}$, following from  Proposition \ref{propQtoCthm} (1) again; that is, $w(\alpha)\in R^+$. Thus $w\in W^{P'}$ and consequently $w=w_1w_2$ for a
     unique $w_2\in W^{P'}_{P}$. Since $gr_{[1, r]}(q_\lambda\sigma^{w_Pw_{P'}})=gr_{[1, r]}(\psi_{\Delta, \Delta_P}(q_{\lambda_P}))=\mathbf{0}=gr_{[1, r]}(q_\lambda\sigma^{w_1w_2})=gr_{[1, r]}(q_\lambda\sigma^{w_2})$, we have $gr_{[1, r]}(w_Pw_{P'})$ $=gr_{[1, r]}(w_2)$. Since $w_2, w_{P}w_{P'}\in W_P$,  $gr_{[r+1, r+1]}(w_Pw_{P'})=\mathbf{0}=gr_{[r+1, r+1]}(w_2)$. Therefore $\ell(w_2)=|gr(w_2)|=|gr(w_Pw_{P'})|=\ell(w_Pw_{P'})$. Hence,  $w_2=w_Pw_{P'}$ by the uniqueness of elements of maximal length in $W^{P'}_P$. Thus the statement follows.
\end{proof}

\begin{lemma}\label{lemmacoealpharrrnegative}
 Let $u, v\in W^P$. Suppose $N_{u, v}^{w, \lambda}\neq 0$ for some  $w\in W$ and $\lambda\in Q^\vee$.
 Assume
 $gr_{[1, r]}(q_\lambda\sigma^w)=\mathbf{0}$ and $\lambda\neq \lambda_B$ where $\lambda_P:=\lambda+Q^\vee_P$. Then we have
     $$gr_{[r, r]}(q_\lambda)<(|R^+_{\tilde P}\cup R^+_{\hat P}|-|R^+_P|)\mathbf{e}_r$$ where
  $\Delta_{\hat P}:=\{\alpha\in \Delta_P~|~ \langle \alpha, \lambda\rangle=0\}$.

\end{lemma}

\begin{proof}  Write  $\lambda=\sum_{j=1}^na_j\alpha_j^\vee$,   $gr_{[r, r]}(q_{r})=x\mathbf{e}_r$ and    $gr_{[r, r]}(q_{r+1})=y\mathbf{e}_r$. Whenever $r+2\leq n$, we denote    $gr_{[r, r]}(q_{r+2})=z\mathbf{e}_r$. Note $gr_{[r, r]}(q_\lambda)=(xa_r+ya_{r+1}+za_{r+2})\mathbf{e}_r$ (where $z=0$ unless case C7) occurs with $r\leq 6$). Let  $\varepsilon_j=-\langle \alpha_j, \lambda\rangle, j=1, \cdots, r$.

 Note that Proposition \ref{propmorphismofalgebras} holds with respect to $QH^*(G/\tilde P)$ and $\tilde{gr}_{[1, \varsigma]}(q_\lambda\sigma^w)=gr_{[1, r-1]}(q_\lambda\sigma^w)=\mathbf{0}$. Hence, $\lambda$ is the unique Peterson-Woodward lifting of $\lambda+Q^\vee_{\tilde P}\in Q^\vee/Q_{\tilde P}^\vee$ to $Q^\vee$.
       Thus $\langle \gamma, \lambda\rangle \in \{0, -1\}$ for all $\gamma\in R^+_{\tilde P}$. Consequently,
         $\varepsilon_j= 0$ for all $j$ in $\{1, \cdots, r-1\}$ with at most one exception, and if there exists such an exception, say $k$, then $\varepsilon_k=1$. Furthermore, we have $\varepsilon_r\geq 0$, by  noting
 $N_{u, v}^{w, \lambda}\neq 0$ and using Corollary \ref{corcoefforprop11}.

 Assume case C1$B$) (resp. case C1$C$)) occurs, then we have $-2y=x=2r$ (resp. $-y=x=r+1$) and $z=0$. In this case, we note $a_{r+1}+{x\over y} a_r=\sum_{j=1}^r j\varepsilon_j$ (resp. ${r\over 2}\varepsilon_r+\sum_{j=1}^{r-1}j\varepsilon_j$ where $\varepsilon_r=2a_{r-1}-2a_r$ is even). If $\varepsilon_r>0$, then we have $-(ya_{r+1}+xa_r)\geq -y r>{r(r+1)\over 2}=r^2-{(r-1)r\over 2}=|R^+_P|-|R^+_{\tilde P}|\geq |R^+_P|-|R^+_{\tilde P}\cup R^+_{\hat P}|$.
 If $\varepsilon_r=0$, then there exists such an exception $k$ with $2\leq k \leq r-1$ (resp. $1\leq k\leq r-1$).  \big(For case C1$B$), each positive root in $R_P^+$ is of the form $\gamma=\varepsilon \alpha_1+\sum_{j=2}^rb_j\alpha_j$ where $\varepsilon=0$ or $1$. If $k=1$, it would imply that $\lambda=\lambda_B$ is the Peterson-Woodward lifting of $\lambda_P$, contradicting with the hypothesis.\big)
  Consequently, we have
   $|R^+_P|-|R^+_{\tilde P}\cup R^+_{\hat P}| =|R^+_P|-|R^+_{\tilde P}|-|R^+_{P_{\Delta_P\setminus\{\alpha_k\}}}|+ |R^+_{P_{\Delta_{\tilde P}\setminus\{\alpha_k\}}}|=r^2-{(r-1)r\over 2}-\big({(k-1)k\over 2}+(r-k)^2\big)+\big({(k-1)k\over 2}+{(r-k-1)(r-k)\over2}\big)=kr-{k(k-1)\over 2}<-yk=-(ya_{r+1}+xa_r)$.

   Assume case C2) occurs, then we have $-2y=x=2(r-1)$ and $z=0$. Note $a_{r+1}+{x\over y} a_r={r\over 2}\varepsilon_r+{r-2\over 2}\varepsilon_{r-1}+\sum_{j=1}^{r-2} j\varepsilon_j$, and $\varepsilon_r-\varepsilon_{r-1}=2a_{r-1}-2a_r\equiv 0 (\mod 2)$.
         If $\varepsilon_r>0$, then   $-(ya_{r+1}+xa_r)\geq (r-1)({r\over 2}\varepsilon_r+{r-2\over 2}\varepsilon_{r-1}+0)\geq (r-1)^2>{r(r-1)\over 2}=|R^+_P|-|R^+_{\tilde P}|\geq |R^+_P|-|R^+_{\tilde P}\cup R^+_{\hat P}|$.
 If $\varepsilon_r=0$, then   there exists such an exception $k$ with $2\leq k \leq r-2$ (since $\lambda\neq \lambda_B$).
  Consequently, we have
   $|R^+_P|-|R^+_{\tilde P}\cup R^+_{\hat P}| =kr-{k(k+1)\over 2}<k(r-1)=-(ya_{r+1}+xa_r)$.

For the remaining cases, the arguments  are all similar, and the details will be given in section \ref{subsecAPPlemmaproof}.
\end{proof}

\bigskip

\begin{proof}[Proof of Proposition \ref{propmorphismofalgebras} (1)] Since $QH^*(G/B)$ is an $S$-filtered algebra,   we have $gr_{[1, r]}(q_\lambda\sigma^w)\leq gr_{[1, r]}(\sigma^{v'})+gr_{[1, r]}(\sigma^{v''})=\mathbf{0}$ if $N_{v', v''}^{w, \lambda}\neq 0$.
  Due to Lemma \ref{lemmagradingreduction11}, it is sufficient  to show $gr_{[1, r]}(q_{\lambda}\sigma^w)<\mathbf{0}$ whenever both $N_{v', v''}^{w, \lambda}\neq 0$ and $q_\lambda\sigma^w\in \psi_{\Delta, \Delta_{\tilde P}}(QH^*(G/\tilde P))\setminus \psi_{\Delta, \Delta_P}(QH^*(G/P))$ hold. For the  latter hypothesis, we only need to check that   either of the following holds:  (a) $gr_{[1, r]}(q_\lambda\sigma^w)=\mathbf{0}$,  $\lambda\neq \lambda_B$;  (b)  $gr_{[1, r]}(q_\lambda\sigma^w)=\mathbf{0}$,  $\lambda= \lambda_B$, $w\neq w_1w_Pw_{P'}$ where $w_1$ is the minimal length representative of the coset $wW_P$. If (b) holds, then it is done by Corollary \ref{corcoefforprop11} (2). Write $w=w_1w_2$ where $w_1\in W^P$ and  $w_2\in W_P$. By Proposition \ref{propQtoCthm} (1), we conclude $w_2(\alpha)\in R^+$ whenever $\alpha\in \Delta_{\hat P}=\{\beta\in \Delta_P~|~ \langle \beta, \lambda\rangle=0\}$. Thus $w_2\in W^{\hat P}_{P}$. Hence, $gr_{[r, r]}(\sigma^{w_2})=
    |\mbox{Inv}(w_2)\cap (R_P^+\setminus R_{\tilde P}^+)|\mathbf{e}_r\leq |(R_P^+\setminus R_{\hat P}^+)\cap (R_P^+\setminus R_{\tilde P}^+)|\mathbf{e}_r=(|R_P^+|-|R_{\hat P}^+\cup R_{\tilde P}^+|)\mathbf{e}_r$. Thus if (a) holds, then the statement follows as well,   by noting
  $gr_{[r, r]}(q_\lambda\sigma^w)=gr_{[r, r]}(q_\lambda)+gr_{[r, r]}(\sigma^{w_2})$ and using
   Lemma \ref{lemmacoealpharrrnegative}.
\end{proof}

\subsection{Proof of Proposition \ref{propmorphismofalgebras} (2)}

 The statement to prove is a direct consequence of the next proposition.

 \begin{prop}\label{propfor22uv}
   Let     $u\in W_P$ and    $v\in W^P$. In $QH^*(G/B)$, we have
     $$\sigma^v\star \sigma^{u}=\sigma^{v u}+\sum_{w, \lambda}b_{w, \lambda}q_\lambda\sigma^w$$
   with $gr(q_\lambda\sigma^w)<gr(\sigma^{v u})$ whenever $b_{w, \lambda}\neq 0$.
 \end{prop}
 \begin{remark} Proposition \ref{propfor22uv} here    extends    Proposition 3.23 of \cite{leungli33} in the case of parabolic subgroups $P$ such that  $\Delta_P$ is not of type $A$.
In Proposition 3.23 of \cite{leungli33}, the same property for   $\sigma^v\star\sigma^{s_j}$ was discussed, under the assumptions that $\Delta_P$ is    of type $A$, $s_j\in W_P$ and $v\in W^P$.
     By modifying the proof therein slightly, the assumption ``$v\in W^P$" could be generalized  to  ``$v\in W$ with $gr_{[j, j]}(v)<j\mathbf{e}_j$".
 \end{remark}

\bigskip

 \begin{proof}[Proof of Proposition \ref{propmorphismofalgebras} (2)]
  This follows immediately from Proposition \ref{propfor22uv}:
     $$\Psi_{r+1}(q_{\lambda_P}\star_P\sigma^{v'})= \overline{q_{\lambda_B}\sigma^{v'w_Pw_{P'}}}=\overline{q_{\lambda_B}}\star \overline{\sigma^{v'}}\star\overline{\sigma^{w_Pw_{P'}}}
=\Psi_{r+1}(q_{\lambda_P})\star \Psi_{r+1}(\sigma^{v'}).$$
\end{proof}

To show Proposition \ref{propfor22uv}, we prove some lemmas first.
 \begin{lemma}\label{lemmavvuuvirtual}
   Let $v\in W^P$ and  $u\in W_P$. Take any   $w\in W$ and $\lambda\in Q^\vee$ satisfying $gr(q_\lambda\sigma^w)=gr(\sigma^u)+gr(\sigma^v)$.
     If $\lambda$ is a virtual null coroot, then we have
     $$N_{v, u}^{w, \lambda}=\begin{cases}
        1,&\mbox{ if }(w, \lambda)=(vu, 0)\\ 0,&\mbox{ otherwise}
     \end{cases}.$$
\end{lemma}

\begin{proof}
Write $w=w_1w_2$ where $w_1\in W^P$ and $w_2\in W_P$. Take a reduced expression $w_2=s_{i_1}\cdots s_{i_m}$ (i.e., $\ell(w_2)=m$).
Since $v\in W^P$ and $\lambda$ is a virtual null coroot, we have $\mbox{sgn}_{\alpha}(v)=0=\langle \alpha, \lambda\rangle$ for all $\alpha\in \Delta_P$. Note $\alpha_{i_j}\in \Delta_P$  and  $\mbox{sgn}_{\alpha_{i_j}}(w_1s_{i_1}\cdots s_{i_j})=1$ for all $1\leq j\leq m$.  Applying the tuple $(u, v, w, \lambda, \alpha)$ of Proposition \ref{propQtoCthm} (2) to the case $(us_{i_m}\cdots s_{i_{j+1}}, vs_{i_j}, w_1s_{i_1}\cdots s_{i_j}, \lambda+\alpha_{i_j}^\vee, \alpha_{i_j})$,
 we have $N_{v, us_{i_m}\cdots s_{i_{j+1}}}^{w_1s_{i_1}\cdots s_{i_j}, \lambda}=N_{v, us_{i_m}\cdots s_{i_{j+1}}s_{i_j}}^{w_1s_{i_1}\cdots s_{i_{j-1}}, \lambda}$ if $\ell(us_{i_m}\cdots s_{i_{j+1}}s_{i_j})=\ell(us_{i_m}\cdots s_{i_{j+1}})-1$, or $0$ otherwise. Hence, we have
  $N_{v, u}^{w, \lambda}=N_{v, u\cdot w_2^{-1}}^{w_1, \lambda}$ if $\ell(u\cdot w_2^{-1})=\ell(u)-\ell(w_2)$, or $0$ otherwise.
  Note $\ell(u)=|gr(u)|=|gr_{[1, r]}(u)+gr_{[1, r]}(v)|=|gr_{[1, r]}(w)|=|gr(w_2)|=\ell(w_2)$. When  $\ell(u\cdot w_2^{-1})=0$, we have
    $u\cdot w_2^{-1}=$id, and consequently  $N_{v, {\scriptsize\mbox{id}}}^{w_1, \lambda}=1$ if $(w_1, \lambda)=(v, 0)$, or 0 otherwise. Thus the statement follows.
 \end{proof}

 Recall $\Delta_{\tilde P}=\{\alpha_1, \cdots, \alpha_{r-1}\}$, whose Dynkin diagram is of type $A_{r-1}$.
It is easy to see the  next combinatorial  fact (see e.g. Lemma 2.8 and Remark 2.9 of \cite{leungliQtoC}).
\begin{lemma}\label{lem-typeAcoroot}
   Let $\lambda\in Q^\vee$ be a nonzero effective coroot, i.e.,   $\lambda=\sum_{i=1}^na_i\alpha^\vee\neq 0$ satisfies $a_i\geq 0$ for all $i$.
   Then there exists   $\alpha\in \Delta$ such that $\langle \alpha, \lambda\rangle>0$. Furthermore if $a_i=0$ for $i=r, r+1, \cdots, n$, and if there exists only one such $\alpha$, then $\langle \alpha, \lambda\rangle>1$.
\end{lemma}

\begin{lemma}\label{lem-combinationPWlifting}
    Let $\lambda\in Q^\vee$, and $\lambda_B$ be the Peterson-Woodward lifting of $\lambda+Q_{P}^\vee$. If $\lambda$
     is the Peterson-Woodward lifting of $\lambda+Q_{\tilde P}^\vee$, then
     either    $\lambda-\lambda_B$ or $\lambda_B-\lambda$ is effective.
   Furthermore if   $\lambda-\lambda_B=\sum_{j=1}^rc_j\alpha_j^\vee\neq 0$, then  the coefficient   $c_r\neq 0$.
\end{lemma}

\begin{proof}
  If follows from the definition of a  Peterson-Woodward lifting that  $\langle \alpha, \lambda_B\rangle=0$ (resp. $\langle \alpha, \lambda\rangle=0$) for all
  $\alpha\in \Delta_P$ (resp. $\Delta_{\tilde P}$)  with at most one exception, and if such exception $\alpha_k$ (resp. $\alpha_{k'}$) exists, then
    $\langle \alpha, \lambda_B\rangle=-1$  (resp. $\langle \alpha_{k'}, \lambda\rangle=1$).
  When there does not exist such exception, we denote $k=k'=n+1$ for notation conventions.

We may assume  $\langle \alpha_r, \lambda-\lambda_B\rangle\geq 0$ (otherwise we consider $\lambda_B-\lambda$).
Then  $\lambda-\lambda_B\in Q^\vee_P$ is given by the  difference between a  dominate coweights and a fundamental coweight in $\Lambda_P^\vee$.   Therefore it is well known that $\lambda-\lambda_B$ is either a nonpositive combination or a nonnegative combination of $\alpha_1^\vee, \cdots, \alpha_{r}^\vee$. (For instance, we can prove this by direct calculations using Table 1 of \cite{hum}).

Now we assume $c_j\geq 0$ for all $j$ (otherwise we consider $\lambda_B-\lambda$).
Since $\langle \alpha_i, \lambda-\lambda_B\rangle$, $i=1, \cdots, r-1$, are all nonpositive with at most an exception of value $1$, we conclude
 $c_r>0$. Otherwise, it would make a contradiction with the second half of the statement of Lemma \ref{lem-typeAcoroot}.
\end{proof}

Recall that $\partial \Delta_P$ denotes the set of simple roots in $\Delta\setminus \Delta_P$ which are adjacent to $\Delta_P$.
\begin{lemma}\label{lem-vanishfornegativecombination}
    Let $v\in W^P, u\in W_P$ and  $w\in W$. Let $\lambda\in Q^\vee$ be effective,  and  $\lambda_B$ be the Peterson-Woodward lifting of $\lambda+Q_{P}^\vee$.
    If      $\lambda-\lambda_B=\sum_{j=1}^rc_j\alpha_j^\vee$ satisfies $c_r<0$ and $c_j\leq 0$ for all $j$, then
     $N_{v, u}^{w, \lambda}=0$.
\end{lemma}

\begin{proof}
  Let $\alpha\in \Delta\setminus (\Delta_P\cup \partial \Delta_P)$. Then $\mbox{sgn}_\alpha(u)=0$ and $vs_\alpha\in W^P$.
  If $\langle \alpha, \lambda\rangle>0$, then the coefficient of $\alpha^\vee$ in $\lambda$ must be positive. By Proposition (1),
   we have $N_{v, u}^{w, \lambda}=0$ unless $\mbox{sgn}_\alpha(v)=\langle \alpha, \lambda\rangle=1-\mbox{sgn}_\alpha(w)=1$. Furthermore when this holds,
    we have $N_{v, u}^{w, \lambda}=N_{vs_\alpha, u}^{ws_\alpha, \lambda-\alpha^\vee}$ by Proposition \ref{propQtoCthm} (2).
    Clearly, $\lambda-\lambda_B=(\lambda-\alpha^\vee)-(\lambda-\alpha^\vee)_B$. Therefore by induction, we can assume  $\langle \alpha, \lambda\rangle\leq 0$ for all $\alpha\in \Delta\setminus (\Delta_P\cup \partial \Delta_P)$.

 The boundary $\partial \Delta_P$ consists one or two nodes. We assume $\partial \Delta_P=\{\alpha_{r+1}\}$ first.
 Then by Lemma \ref{lem-typeAcoroot}, we have $\langle \alpha_{r+1}, \lambda_B\rangle \geq 1$.

 Assume that  $\alpha_r$ is adjacent to $\alpha_{r+1}$, which happens in cases C5), C7) with $r=7$, C9) with $r=3$, and C10). Then $\langle \alpha_{r+1}, \lambda\rangle= \langle \alpha_{r+1}, \lambda_B\rangle+c_r \langle \alpha_{r+1}, \alpha_r^\vee\rangle\geq 2$. Since $\mbox{sgn}_{\alpha_{r+1}}(u)=0$, we have
  $N_{v, u}^{w, \lambda}=0$ by Proposition \ref{propQtoCthm} (1).

  Assume that $\alpha_1$ is adjacent to $\alpha_{r+1}$. This happens in cases C1B), C1C), C2) and C4). If $\lambda_B$ is of the form $a\mu_B+\alpha_{r+1}^\vee+\sum_{j={r+2}}^na_j\alpha_j^\vee$, then by the hypotheses $\langle \alpha_j, \lambda_B\rangle \leq 0$ for all $j\geq r+2$, and the precise description of $\mu_B$ in Table \ref{tabvirtualnullcoroot}, we can easily conclude that $\langle\alpha_{r+1}, \lambda_B-\alpha_{r+1}^\vee\rangle \geq 0$. Hence, we obtain  $N_{v, u}^{w, \lambda}=0$ again by the same arguments above.
  If $\lambda_B$ is not of the aforementioned form, then $\lambda_B$ is the  combination of a virtual null coroot and a non-simple coroot in Table   \ref{tabvirtualnullcoroot} (or zero coroot). Such $\lambda_B$ satisfies $\langle \alpha_i, \lambda_B\rangle =0$ for all $1\leq i\leq r_a$, where $\alpha_{r_a}$ is the simple root adjacent to $\alpha_r$. Let $m$ be the minimum of the set $\{i~|~ 1\leq i\leq r_a, c_{i+1}<0\}$ if nonempty, or $m:=r_a$ otherwise.
  Then $\langle \alpha_m, \lambda\rangle = \langle \alpha_m, \sum_{i=1}^rc_i\alpha_i^\vee\rangle>0$.
  Since $\mbox{sgn}_{\alpha_m}(v)=0$,  $N_{v, u}^{w, \lambda}=0$ unless $\mbox{sgn}_{\alpha_m}(u)=\langle \alpha_m, \lambda\rangle=1- \mbox{sgn}_{\alpha_m}(w)=1$. When this holds, we have $N_{v, u}^{w, \lambda}=N_{v, us_{m}}^{ws_m, \lambda-\alpha_m^\vee}$ with $us_m\in W_P$ and
 $ \lambda-\alpha_m^\vee=\lambda_B+(-1)\cdot \alpha_m^\vee+\sum_{j=1}^rc_j\alpha_j^\vee$, by Proposition \ref{propQtoCthm}.
 Hence, by reduction, we can assume $c_1<0$. Consequently, we have $\langle \alpha_{r+1}, \lambda\rangle \geq 2$, and then obtain  $N_{v, u}^{w, \lambda}=0$.

 Now we assume $\partial \Delta_P=\{\alpha_{r+1}, \alpha_{r+2}\}$. That is, case C7) with $4\leq r\leq 6$, or case C9) with $r=2$ occurs.
If $\langle \alpha_{r+1}, \lambda_B\rangle>0$, then we are done by the same arguments as above, since  $\alpha_r$ is adjacent to $\alpha_{r+1}^\vee$ and $c_r<0$.
 If $\langle \alpha_{r+1}, \lambda_B\rangle\leq 0$, then
 $\langle \alpha_{r+2}, \lambda_B\rangle > 0$.
  If $\lambda_B=\tau+\alpha_{r+2}^\vee$ with $\tau$ a virtual null coroot, then we    conclude
   $\langle \alpha_{r+2}, \tau\rangle \geq  0$. (For instance when case 7) with $r=6$ occurs, we have
    $\tau= a  \mu_B^{(1)}+b\mu_B^{(2)}$, where $\mu_B^{(1)}, \mu_B^{(2)}$ denote the corresponding two coroots in Table \ref{tabvirtualnullcoroot}, and $a, b\geq 0$. We have  $\langle \alpha_{7}, \lambda_B\rangle=a-b\leq 0$ and $\langle \alpha_{8}, \lambda_B\rangle=2b-a+2> 0$. This implies $2b-a\geq0$. The arguments for the remaining cases are similar.)
  If $\lambda_B$ is not of the aforementioned form, then by Table \ref{tabNONvirtualnullcoroot} we conclude that  $\langle \alpha_i, \lambda_B\rangle =0$ for all $1\leq i\leq r_a$, where $\alpha_{r_a}$ is the simple root of $\Delta_P$ adjacent to $\alpha_r$. Therefore, we are done by the same arguments as above.
\end{proof}

\bigskip

\begin{proof}[Proof of Proposition \ref{propfor22uv}]
    Since $QH^*(G/B)$ is an $S$-filtered algebra,  by Lemma \ref{lemmavvuuvirtual}, we have
   $$\sigma^{v}\star \sigma^{  u}=\sigma^{v  u}+\sum N_{v,   u}^{w, \lambda}q_\lambda\sigma^w+ \sum b_{w', \lambda'}q_{\lambda'}\sigma^{w'}.$$ Here
     $gr(q_{\lambda'}\sigma^{w'})<gr(\sigma^v)+gr(\sigma^{\tilde u})$. The first summation is over those $q_\lambda\sigma^w$ satisfying both
      $$\mbox{(i) } \lambda=\sum\nolimits_{j=1}^na_j\alpha_j^\vee  \mbox{ is not a virtual null coroot}, 
                 \mbox{ where } a_j\geq 0\mbox{ for all }j,$$
       and $\mbox{(ii) } gr(q_\lambda\sigma^w)=gr(\sigma^v)+gr(\sigma^{ u})$. The hypothesis (ii)  is     equivalent to
       $$\mbox{(ii)}' \,\, gr_{[1, r]}(q_\lambda\sigma^w)=gr_{[1, r]}(\sigma^{  u}),$$ following from the dimension constraint of Gromov-Witten invariants $N_{v, u}^{w, \lambda}$ (see also Lemma 3.11 of \cite{leungli33}) and the assumption that $v\in W^P$.

       By Proposition \ref{propreductiontor-1} (1), we conclude that elements in the first summation also satisfy
       $$\mbox{(iii) }\sigma^wq_\lambda=\psi_{\Delta, \Delta_{\tilde P}}(q_{\tilde {\lambda}_{\tilde P}})  \mbox{ where }\tilde {\lambda}_{\tilde P}:=\lambda+Q^\vee_{\tilde P}.$$
          Therefore, it suffices to show $N_{v, \tilde u}^{w, \lambda}=0$ whenever all (i), (ii)$'$ and (iii) hold.

 Let $\lambda_B$ denote the Peterson-Woodward lifting of $\lambda+Q_P^\vee$. By Lemma \ref{lem-combinationPWlifting},
 the coefficients $c_j$ of   $\lambda-\lambda_B=\sum_{j=1}^rc_j\alpha_j^\vee$ are all nonpositive or all nonnegative, and $c_r\neq 0$ due to (i).
    If $c_r<0$, then we are done by Lemma \ref{lem-vanishfornegativecombination}. Therefore we assume $c_r>0$ in the following.

      Since all $c_j\geq 0$, we  write $\lambda = \lambda_B+\sum_{i=1}^t\beta_i^\vee$.
      The set $\Delta_{P'}=\{\alpha\in \Delta_P~|~ \langle \alpha, \lambda_B\rangle =0\}$ coincides with either $\Delta_P$ or
         $\Delta_P\setminus\{\alpha_k\}$ for a unique $\alpha_k\in \Delta_P$ with $\langle \alpha_k, \lambda_B\rangle=-1$. Therefore we can further assume
   that    $\beta_i\in \Delta_P$,    $i=1, 2, \cdots, t$, satisfy $\langle \beta_j, \lambda_B\rangle +\langle \beta_j,  \sum_{i=j}^{t}\beta_i^\vee\rangle >0$ for all $1\leq j\leq t$. (This can be done: if $\langle \alpha, \lambda-\lambda_B\rangle >0$ holds for some  $\alpha$ in $\Delta_P$ distinct from $\alpha_k$, then we simply choose $\beta_1=\alpha$. If not, then   $\alpha=\alpha_k$, and the coefficient of $\alpha_k$ in the highest root of the root  subsystem $R_P$ is equal to 1. Hence we  conclude   $\langle \alpha, \lambda-\lambda_B\rangle \geq 2$.)

   Since $v\in W^P$,  $\mbox{sgn}_{\alpha}(v)=0$ for all $\alpha\in \Delta_P$.
    For each $1\leq j\leq t$, we have $N_{  v,   u s_{\beta_1}\cdots s_{\beta_{j-1}}}^{ws_{\beta_1}\cdots s_{\beta_{j-1}},\lambda_B+\sum_{i=j}^t\beta_j^\vee}=0$ unless  $\ell(  u s_{\beta_1}\cdots s_{\beta_j})=\ell(  u s_{\beta_1}\cdots s_{\beta_{j-1}})-1$,
      $\ell(w s_{\beta_1}\cdots  s_{\beta_j})=\ell( u s_{\beta_1}\cdots s_{\beta_{j-1}})+1$ and  $\langle \beta_j, \lambda_B\rangle +\langle \beta_j, \sum_{i=j}^{t}\beta_i^\vee\rangle=1$ all hold, by Proposition \ref{propQtoCthm} (1). Furthermore when all these hypotheses hold, we have
     $$N_{  v,   u s_{\beta_1}\cdots s_{\beta_{j-1}}}^{ws_{\beta_1}\cdots s_{\beta_{j-1}},\lambda_B+\sum_{i=j}^t\beta_j^\vee}=N_{  v,   u s_{\beta_1}\cdots s_{\beta_j}}^{ws_{\beta_1}\cdots s_{\beta_{j-1}}s_{\beta_j}, \lambda_B+\sum_{i=j+1}^t\beta_j^\vee}$$
      by
    applying the tuple $(u, v, w, \lambda, \alpha)$ of
   Proposition \ref{propQtoCthm} (2) to the tuple \\$(\tilde u s_{\beta_1}\cdots s_{\beta_{j-1}}, vs_{\beta_j},  w s_{\beta_1}\cdots s_{\beta_{j-1}}s_{\beta_j}, \lambda_B+\sum_{i=j}^t\beta_j^\vee, \beta_j)$.
Denote $u':=  u s_{\beta_1}\cdots s_{\beta_t}$. Combining all these, we have
 $$N_{v,   u}^{w, \lambda}=N_{v, u'}^{ws_{\beta_1}\cdots s_{\beta_t},  \lambda_B}$$ if all the hypotheses ($\ddagger$) hold:
  $$\ell(ws_{\beta_1}\cdots s_{\beta_t})=\ell(w)+t, \ell(u')=\ell(  u)-t, \langle \beta_j, \lambda_B\rangle +\langle \beta_j, \sum_{i=j}^{t}\beta_i^\vee\rangle=1, j=1, \cdots, t,$$
   or 0 otherwise. In particular if $\lambda_B=0$, then we are done since the   hypotheses on the step $j=t$ cannot hold.

   It suffices to show   $N_{v, u'}^{ws_{\beta_1}\cdots s_{\beta_t},  \lambda_B}=0$ under the hypotheses $(\ddagger)$ and $\lambda_B\neq 0$.
   If $\ell(u')=0$, then $u'=\mbox{id}$, and we are  done.  Assume $\ell(u')>0$ now.
For any $\eta\in Q_P^\vee$, we have
  $|gr_{[1, r]}(q_\eta)|=|gr(q_\eta)|=\langle 2\rho, \eta\rangle$, following from the definition.
 Due to (ii)$'$,
        $$|gr_{[1, r]}(w)|+|gr_{[1, r]}(q_{\lambda_B})|+2t=|gr_{[1, r]}(q_\lambda \sigma^w)|=|gr_{[1, r]}( u)|=\ell( u).$$
By
Proposition \ref{propwelldefinedgrading},   $-|gr_{[1, r]}(q_{\lambda_B})|=|gr_{[1, r]}(w_Pw_{P'})|=\ell(w_Pw_{P'})= |R^+_P|-|R^+_{ P'}|$. Combining both,  we have
     $$|gr_{[1, r]}(ws_{\beta_1}\cdots s_{\beta_t})|=|gr_{[1, r]}(w)|+t= \ell(u') -|gr_{[1, r]}(q_{\lambda_B})|= \ell(u')+ |R^+_P|-|R^+_{ P'}|.$$
 Thus there is $\alpha\in \Delta_{P'}$ such that $\mbox{sgn}_\alpha(ws_{\beta_1}\cdots s_{\beta_t})=1$ (otherwise, $ws_{\beta_1}\cdots s_{\beta_t}(\alpha)\in R^+$ for all $\alpha\in \Delta_{P'}$, which would imply
      $|gr_{[1, r]}(ws_{\beta_1}\cdots s_{\beta_t})|\leq |R^+_P|-|R^+_{ P'}|$).
  Since $v\in W^P$, $\mbox{sgn}_\alpha(v)=0$.   By Proposition \ref{propQtoCthm} (1), we have
      $N_{v, u'}^{ws_{\beta_1}\cdots s_{\beta_t},  \lambda_B}=0$ unless $\mbox{sgn}_\alpha(u')=1$, i.e., $\ell(u's_\alpha)=\ell(u')-1$. Furthermore when this holds,
       we have  $N_{v, u'}^{ws_{\beta_1}\cdots s_{\beta_t},  \lambda_B}=N_{v, u's_\alpha}^{ws_{\beta_1}\cdots s_{\beta_t}s_\alpha,  \lambda_B}$
       (by applying the tuple $(u, v, w, \lambda, \alpha)$ of Proposition \ref{propQtoCthm} (2) to $(vs_\alpha, u', ws_{\beta_1}\cdots s_{\beta_t}s_\alpha, \lambda_B+\alpha^\vee, \alpha)$).
      By induction, we conclude
           $N_{v, u'}^{ws_{\beta_1}\cdots s_{\beta_t},  \lambda_B}=0$ unless   both $u'\in W_{P'}$ and $\ell(ws_{\beta_1}\cdots s_{\beta_t}u'^{-1})=
           \ell(ws_{\beta_1}\cdots s_{\beta_t})-\ell(u')$ hold. Furthermore when both hypotheses hold, we have
            $N_{v, u'}^{ws_{\beta_1}\cdots s_{\beta_t},  \lambda_B}=N_{v, \scriptsize\mbox{id}}^{ws_{\beta_1}\cdots s_{\beta_t}u'^{-1},  \lambda_B}=0$ since $\lambda_B\neq 0$.
\end{proof}

\subsection{Proof of Proposition \ref{propmorphismofalgebras} (3)}\label{subsecproofforprop33}
The statement tells us that   the elements $\Psi_{r+1}(q_{\lambda_P})$  in $Gr^{\mathcal{F}}_{(r+1)}$ do behave like monomials. Due to Lemma \ref{lemmagrading33cases}, it suffices to show those $q_{\lambda}\sigma^u$ behave like the non-identity elements of the finite abelian group $(Q^\vee/Q^\vee_P)/\mathcal{L}$ as in Table \ref{tabNONvirtualnullcoroot}.
 For any one of the  cases C1$B$), C1$C$), C2) and C9), we use  the first virtual null coroot  $\mu_B$ in Table  \ref{tabvirtualnullcoroot} and the unique element $q_\lambda \sigma^u$ in Table \ref{tabNONvirtualnullcoroot}. Namely for the only exceptional case when C9) with $r=2$ occurs, there are two virtual null coroots, and we will use the one  $\mu_B= 2\alpha_4^\vee+2\alpha_1^\vee+ \alpha_2^\vee$.
For any one of these cases, we only need to use check one quantum multiplication as in the next proposition, which we assume first. The remaining cases require verifications of more quantum multiplications, which will be discussed in section 5.3.

\begin{prop}\label{propuuforC1BC2}
  Assume   {\upshape C1$B$)},   {\upshape C2)} or     {\upshape C9)} occurs.    In $QH^*(G/B)$, we have
    $$q_\lambda\sigma^u\star q_\lambda\sigma^u=q_{\mu_B}+\sum b_{w', \lambda'}q_{\lambda'}\sigma^{w'}$$
    with $gr(q_{\lambda'}\sigma^{w'})<gr(q_{\mu_B})$ whenever $b_{w', \lambda'}\neq 0$.
\end{prop}

\bigskip

\begin{proof}[Proof of Proposition \ref{propmorphismofalgebras} (3)] Let  $q_{\kappa_P}, q_{\kappa_P'}\in QH^*(G/P)$. If   case C1$C$) occurs, then we note
  $\psi_{\Delta, \Delta_P}(q_{\kappa_P})\star \psi_{\Delta, \Delta_P}(q_{\kappa_P'})=q_{\kappa_B}\star q_{\kappa_B'}=q_{\kappa_B+\kappa_B'}=\psi_{\Delta, \Delta_P}(q_{\kappa_P+\kappa_P'})$. Therefore,
$\Psi_{r+1}(q_{\kappa_P})\star \Psi_{r+1}(q_{\kappa_P'})=\Psi_{r+1}(q_{\kappa_P+\kappa_P'})$.
Assume that case C1$B$), C9) or C2) occurs now. If either $\kappa_P$ or $\kappa_P'$ is a virtual null coroot, then we are done, by using Lemma \ref{lemmagrading33cases}. Otherwise, by Proposition \ref{propNONvirtualnullcoroot} we have $\kappa_P=\tau_P+\lambda_P$ and $\kappa_P'=\tau_P'+\lambda_P$ for some virtual null coroots $\tau_P, \tau_P'$, and consequently $\kappa_P+\kappa_{P}'=\tau_P+\tau_P'+(\mu_B+Q^\vee_P)$. Here $\mu_B$ and   $\psi_{\Delta, \Delta_P}(q_{\lambda_P})=q_\lambda\sigma^u$ are  given in Table \ref{tabvirtualnullcoroot} and Table \ref{tabNONvirtualnullcoroot} respectively.   Hence, we have $\Psi_{r+1}(q_{\kappa_P})= \Psi_{r+1}(q_{\tau_P})\star \Psi_{r+1}(q_{\lambda_P})$ and $\Psi_{r+1}(q_{\kappa_P'})= \Psi_{r+1}(q_{\tau_P'})\star \Psi_{r+1}(q_{\lambda_P})$, by Lemma \ref{lemmagrading33cases}.
Using Proposition \ref{propuuforC1BC2}, we have
   $\Psi_{r+1}(q_{\lambda_P})\star \Psi_{r+1}(q_{\lambda_P})=\overline{q_\lambda\sigma^u}\star \overline{q_\lambda\sigma^u}=\overline{q_{\mu_P}}$.
   Hence, we have $\Psi_{r+1}(q_{\kappa_P})\star \Psi_{r+1}(q_{\kappa_P'})=\overline{q_{\tau_B}}\star\overline{q_{\tau_P'}}\star\overline{q_{\mu_B}}=\overline{q_{\tau_B+\tau_B'+\mu_B}}=\Psi_{r+1}(q_{\kappa_P+\kappa_P'})$.
   For the remaining cases in Table \ref{tabrelativeposi}, the statements follows from the arguments given in  section \ref{subsecAPPprop33uu}. Thus we are done.
\end{proof}

Now we prepare some lemmas in order to prove Proposition \ref{propuuforC1BC2}. The reduced expressions of the longest element $w_P$ in $W_P$ are not unique. There is a conceptual approach to construct $w_P$ of  the form  $w^{h\over 2}$ whenever $h$ is even (see e.g. Chapter 3 of \cite{humr}). Here $h$ denotes the Coxeter number of $W_P$, and it is equal to $2r$ (resp. $2r-2$) for $\Delta_P$ of type $B_r$   (resp. $D_r$).  The next lemma provides a special choice of the above $w\in W_P$.

\begin{lemma}\label{lemmalongestWforBD}
  For $\Delta_P$ of type $B_r$   or  $D_r$,   $(s_1\cdots s_r)^{h\over 2}$  is a reduced expression of the longest element $w_P$.
\end{lemma}
\begin{proof}
  It is easy to check that the given element maps all simple roots in $\Delta_P$ to negative roots, and note $\ell(w_P)=r^2$ (resp. $r(r-1)$). Thus the statement follows.
\end{proof}

Recall that for   $u$ in Table \ref{tabNONvirtualnullcoroot},   $\tilde u$ denotes  the minimal length representative of $uW_{\tilde P}$.

\begin{lemma}\label{lemmatildeUUnotinVV}
  Let $v=s_{\beta_1}\cdots s_{\beta_p}\in W_P$ be a reduced expression. Assume  {\upshape C1$B$)},  {\upshape C2)}  or {\upshape C9)}  occurs, then
     $\tilde u^{-1}\leqslant v$ if and only if there exists a subsequence $[i_1, \cdots, i_{{h\over 2}}]$  of $[1, \cdots, p]$ such that
      $[\beta_{i_1},\cdots, \beta_{i_{h\over 2}}]=[\alpha_r, \alpha_{{h\over 2}-1}, \cdots, \alpha_2, \alpha_1]$.
\end{lemma}

\begin{proof}
     Note $\ell(\tilde u^{-1})={h\over 2}$. It is a general fact that $\tilde u^{-1}\leqslant v$ if and only if there exists a subsequence $[i_1, \cdots, i_{{h\over 2}}]$  of $[1, \cdots, p]$ such that
        $\tilde u^{-1}=s_{\beta_{i_1}}\cdots s_{\beta_{i_{h\over 2}}}$.  Since the simple reflections in $\tilde u^{-1}=s_r s_{{h\over 2}-1}\cdots s_2s_1$ are distinct,  we conclude that the two sets $\{\alpha_r, \alpha_{{h\over 2}-1}, \cdots, \alpha_{2}, \alpha_1\}$ and $\{\beta_{i_1}, \cdots, \beta_{i_{h\over 2}}\}$ coincide with each other. Then  the coincidence of the corresponding two ordered sequences follows immediately from the
       obvious observation that $s_r s_{{h\over 2}-1}\cdots s_{j+1}s_j(\alpha_j)\in -R^+$ for all $j$.
\end{proof}

The next well-known fact works for arbitrary $\Delta_P$   (see e.g.  Theorem 3.17 (iv) of \cite{bgg}\footnote{The Schubert cohomology classes $\sigma^v$ are denoted as $P_{v^{-1}}$ in \cite{bgg}.}).

\begin{lemma}\label{lemmavanishingofcupproductwv}
   Let $w, v\in W_P$. If $w^{-1}\not\leqslant v^{-1}w_P$, then   $\sigma^w\cup\sigma^v=0$ in $H^*(P/B)$.
\end{lemma}

\begin{cor}\label{corvanishingtildeu2}
   For case {\upshape C1$B$)}, {\upshape C2)} or {\upshape C9)}, we have $\sigma^{\tilde u}\cup\sigma^{\tilde u}=0$ in $H^*(P/B)$.
\end{cor}

\begin{proof}
 By Lemma \ref{lemmalongestWforBD},
   $\tilde u^{-1}w_P$ is equal to $(s_1 s_{2}\cdots s_r)^{r-1}$ if case {\upshape C1$B$)} or C9) occurs, or equal to
       $s_{r-1}(s_1 s_{2}\cdots s_r)^{r-2}$   if case  {\upshape C2)} occurs (since $s_rs_{r-1}=s_{r-1}s_r$). Clearly, there does not exist  a subsequence $[i_1, \cdots, i_{h\over 2}]$   satisfying
    $[\alpha_{i_1}, \cdots, \alpha_{i_{h\over 2}}]=[\alpha_r, \alpha_{{h\over 2}-1},\cdots, \alpha_2, \alpha_1]$. Thus $\tilde u^{-1}\not \leqslant\tilde u^{-1}w_P$ by Lemma \ref{lemmatildeUUnotinVV}. Hence, the statement follows from Lemma \ref{lemmavanishingofcupproductwv}.
\end{proof}

\bigskip

\begin{proof}[Proof of Proposition \ref{propuuforC1BC2}]
 Due to the filtered algebra structure of $QH^*(G/B)$, we have
  $q_\lambda\sigma^u\star q_\lambda\sigma^u=\sum_{w, \eta}N_{u,u}^{w, \eta}q_{\eta+2\lambda}\sigma^w +\sum b_{w', \lambda'}q_{\lambda'}\sigma^{w'}$, where
              $gr(q_{\eta+2\lambda}\sigma^w
              )=2gr(q_{\lambda}\sigma^u)$ and   $gr(q_{\lambda'}\sigma^{w'})<2gr(q_{\lambda}\sigma^u)$.
         Since $gr_{[r+1, r+1]}(\sigma^u)=\mathbf{0}$, we conclude $w\in W_P$ and $\eta=\sum_{i=1}^rb_i\alpha_i^\vee$ where $b_i\geq 0$.
        Note   $gr_{[r, r]}(q_r)=h\mathbf{e}_r$ by Table \ref{tabgradingQQ}.
          Write $w=w_1w_2$ where $w_1\in W^{\tilde P}_P$ and $w_2\in W_{\tilde P}$. Using $gr_{[r, r]}$, we conclude
           $\ell(w_1)+b_rh=h.$
           If $b_r=0$, then
           $gr(q_{\eta}\sigma^{w_2})=gr_{[1, r-1]}(q_{\eta}\sigma^{w_1w_2})=2gr_{[1, r-1]}(u)=2\sum_{i=1}^{r-1}\mathbf{e}_i=gr(q_1\sigma^{v_{r-1}\cdots v_2})$
  where $v_j:=s_{j-1}s_j$ for each $2\leq j\leq r-1$. Thus we have $q_{\eta}\sigma^{w_2}=q_1\sigma^{v_{r-1}\cdots v_2}$ by Lemma \ref{lemmauniquegradinginPB}.
  In $Gr^{\tilde{\mathcal{F}}}(QH^*(G/B))$, we have $\overline{\sigma^{u}\star \sigma^u}=\overline{N_{u, u}^{w, \eta}q_1 \sigma^{w_1v_{r-1}\cdots v_1}+\mbox{ other   terms}}$. On the other hand, by Proposition \ref{propreductiontor-1}, we have
      $\overline{\sigma^{u}\star \sigma^u}=(\overline{\sigma^{\tilde u}}\star \overline{\sigma^{s_{r-1}\cdots s_1}})^2=\big(\overline{N_{\tilde u, \tilde u}^{w_1, 0}\sigma^{w_1} + \mbox{term 1}}\big)\star\big(\overline{\mbox{term }2}\big)$. Here term 1 is a nonnegative combination of the form $\psi_{\Delta, \Delta_{\tilde P}}(q_{\eta'_{\tilde P}}\sigma^{w_1'})$ with  either $\eta'_{\tilde P}\neq 0$ or $w_1\neq w_1'\in W^{\tilde P}$, and term 2 is a combination of elements $q_{\lambda''}\sigma^{v''}$ with $(v'', \lambda'')\in W_{\tilde P}\times Q^\vee_{\tilde P}$.
  Hence, $N_{u, u}^{w, \eta}\neq 0$ only if $N_{\tilde u, \tilde u}^{w_1, 0}\neq 0$, the latter of which is  the coefficient  of $\sigma^{w_1}$ in $\sigma^{\tilde u}\cup \sigma^{\tilde u}\in H^*(G/B)\subset QH^*(G/B)$.
  It is a general fact (following from the surjection $H^*(G/B)\rightarrow H^*(P/B)$) that $N_{\tilde u, \tilde u}^{w_1, 0}$ coincides with  the coefficient of $\sigma^{w_1}$ in $\sigma^{\tilde u}\cup \sigma^{\tilde u}\in H^*(P/B)$, and therefore it is equal to 0 by Corollary \ref{corvanishingtildeu2}.
     Thus $N_{u, u}^{w, \eta}=0$ whenever $b_r=0$. It remains to deal with the case $b_r=1$. By Lemma \ref{lemmauniquegradinginPB},  there
       is exactly one such term, which  turns out to be $q_\eta\sigma^{w}=q_{\mu_B-2\lambda}.$ Thus it suffices to show $N_{u, u}^{\scriptsize \mbox{id}, \mu_B-2\lambda}=1$.
       Using Proposition \ref{propQtoCthm} (2) repeatedly, we conclude the followings.
        \begin{enumerate}
             \item If case C1$B$) or C9) occurs, then $\eta=\mu_B-2\lambda=\alpha_r^\vee+2\sum_{i=1}^{r-1}\alpha_i^\vee$, and  we have
              $$N_{u, u}^{\scriptsize\mbox{id }, \eta}=N_{\tilde u, \tilde u s_{r-1}}^{s_1\cdots s_{r-1}s_1 \cdots s_{r-2}, q_{r-1}q_r}=N_{s_1\cdots s_{r-1}, s_1\cdots s_r s_{r-1}}^{s_1\cdots s_{r-1}s_1 \cdots s_{r-2}s_r, q_{r-1}}=N_{s_1\cdots s_{r-2}, s_1\cdots s_r}^{s_1\cdots s_{r-1}s_rs_1 \cdots s_{r-2}, 0}.$$
             \item If   C2) occurs, then $\eta=\mu_B-2\lambda=\alpha_r^\vee+\alpha_{r-1}^\vee+2\sum_{i=1}^{r-2}\alpha_i^\vee$, and  we have
              $$N_{u, u}^{\scriptsize\mbox{id }, \eta}=N_{\tilde us_{r-1}, \tilde u s_{r-1}s_{r-2}}^{s_1\cdots s_{r-2}s_1 \cdots s_{r-3}, q_{r-2}q_{r-1}q_r}=N_{s_1\cdots s_{r-2}, \tilde us_{r-1}s_{r-2}}^{s_1\cdots s_{r-2}s_1 \cdots s_{r-3}s_rs_{r-1}, q_{r-2}}=N_{s_1\cdots s_{r-3}, s_1\cdots s_r}^{s_1\cdots s_{r}s_1 \cdots s_{r-3}, 0}.$$
            \end{enumerate}
         Namely, we always have  $N_{u, u}^{\scriptsize\mbox{id }, \mu_B-2\lambda}=N_{u'', v''}^{v''u'', 0}$ with $u''=s_1\cdots s_{{h\over 2}-2}\in W_{P''}$ and $v''=s_1\cdots s_r\in W^{P''}$, where $\Delta_{P''}:=\{\alpha_1, \cdots, \alpha_{{h\over 2}-2}\}$. Thus it  is equal to  1 by Lemma \ref{lemmavvuuvirtual}.
     \end{proof}
\section{Conclusions for general $\Delta_P$} \label{sectiongeneralcase}
In this section, we allow $P/B$ to be reducible,  namely the Dynkin diagram $Dyn(\Delta_P)$ could be disconnected. We will first show the coincidence between  the grading map $gr$ defined in section 2.2 and the one introduced in \cite{leungli33}. Then we will refine the statement of Theorem 5.2 of \cite{leungli33}, and will sketch the proof of it.

Whenever referring to  the subset  $\Delta_P=\{\alpha_1, \cdots, \alpha_r\}$, in fact, we have already given an ordering on the $r$ simple roots in $\Delta_P$, in terms of $\alpha_i$'s.
As we can see in Definition \ref{defnofgrading}, the grading map $gr: W\times Q^\vee\rightarrow \mathbb{Z}^{r+1}$
depends only on such an ordering of $\alpha_i$'s  in $\Delta_P$, which has  nothing to do with the connectedness of $Dyn(\Delta_P)$. Therefore we can use the same definition even if $Dyn(\Delta_P)$ is disconnected. We want to show $gr$ coincides with the grading map given by Definition 2.8 (resp. 5.1) of \cite{leungli33} when $Dyn(\Delta_P)$ is connected (resp. disconnected).

Recall $\Delta_0:=\emptyset, \Delta_{r+1}:=\Delta$,  $\Delta_i:=\{\alpha_1, \cdots, \alpha_i\}$ for $1\leq i\leq r$, and $P_j:=P_{\Delta_j}$ for all $j$.
 Denote $\rho_j:={1\over 2}\sum_{\beta \in R^+_{P_j}} \beta$ where $\rho_0:=0$.
 Then for any $\lambda \in Q^\vee$, we have $gr(\mbox{id}, \lambda)=\sum_{j=1}^{r+1}\langle 2\rho_j-2\rho_{j-1}, \lambda\rangle\mathbf{e}_j$ by Definition \ref{defnofgrading}.

  \begin{lemma}\label{lemmagradingJJJ}
     For any $\alpha\in\Delta_j$, we have $gr_{[j+1, r+1]}(\mbox{id}, \alpha^\vee)=\mathbf{0}$ and $|gr(\mbox{id}, \alpha^\vee)|=2$.
  \end{lemma}
\begin{proof}
   It is well-known that $\rho_k$   equals  the sum of fundamental weights in the root subsystem $R_{P_k}$. That is, we have $\langle \rho_k, \alpha^\vee\rangle =1$ for any $\alpha\in \Delta_k$.
   Hence, for  $j\leq k\leq r+1$, we have $|gr_{[1, k]}(\mbox{id}, \alpha^\vee)|=\sum_{i=1}^{k}\langle 2\rho_i-2\rho_{i-1}, \alpha^\vee\rangle= \langle 2\rho_k, \alpha^\vee\rangle=2$. Thus if $i>j$, then $|gr_{[i, i]}(\mbox{id}, \alpha^\vee)|=|gr_{[1, i]}(\mbox{id}, \alpha^\vee)|-|gr_{[1, i-1]}(\mbox{id}, \alpha^\vee)|=2-2=0$.
 \end{proof}

  By abuse of notation, we still denote by $\psi_{\Delta_{j+1}, \Delta_j}$ the injective map
   $\psi_{\Delta_{j+1}, \Delta_j}: W_{P_{j+1}}^{P_j}\times Q_{P_{j+1}}^\vee/Q_{P_j}^\vee \longrightarrow  W\times Q^\vee
               $ induced from the Peterson-Woodward comparison formula. We recall  Definition 2.8   of \cite{leungli33} as follows.

\begin{defn}\label{defnofLeungLi33}
   Define a  {grading map}     $gr': W\times Q^\vee  \longrightarrow \mathbb{Z}^{r+1}$ associated to
            $\Delta_P=(\alpha_1, \cdots, \alpha_r)$ as follows.

  \begin{enumerate}
     \item   For $w\in W$, we take its (unique) decomposition $w=v_{r+1}\cdots v_1$ where $v_j\in W^{P_{j-1}}_{P_{j}}$. Then
     we define
          $gr'(w,  0)=\!\! \sum_{i=1}^{r+1}\ell(v_i)\mathbf{e}_{i}.$ 
    \item For   $\alpha\in \Delta$,  we can define all $gr'(\mbox{id}, \alpha^\vee)$ recursively  in the following way.
    Define $gr'(\mbox{id}, \alpha_1^\vee)=2\mathbf{e}_1$; for any $\alpha\in \Delta_{j+1}\setminus \Delta_j$,  we define
                $${}\hspace{0.4cm}gr'(\mbox{id},  {\alpha^\vee})= \big( \ell(w_{P_j}w_{P_j'})+2+\sum\nolimits_{i=1}^j2a_i\big)\mathbf{e}_{j+1}-gr'(w_{P_j}w_{P_j'}, 0)-
                 \sum\nolimits_{i=1}^ja_igr'(\mbox{id}, \alpha_i^\vee),$$
     where  $w_{P_j}w_{P_j'}$ and  $a_i$'s
     are defined by the image
          $\psi_{\Delta_{j+1}, \Delta_j}(\mbox{id}, \alpha^\vee+Q_{P_j}^\vee)=(w_{P_j}w_{P_j'}, \alpha^\vee\!\!+\! \sum\limits_{i=1}^j a_i\alpha_i^\vee)$.

      \item In general,   we define  $gr'(w, \sum_{k=1}^n{b_k}\alpha_k^\vee)=gr'(w, 0)+\sum_{k=1}^nb_kgr'(\mbox{id}, \alpha_k^\vee)$.
  \end{enumerate}

\end{defn}

 One of the main results of \cite{leungli33}, i.e., Proposition \ref{propmainthm}, tells us that
  the grading $gr'$ respects the quantum multiplication. Precisely  for any Schubert classes $\sigma^u, \sigma^v$ of $QH^*(G/B)$, if $q_\lambda \sigma^w$ occurs in the quantum multiplication $\sigma^u\star \sigma^v$, then
    $$gr'(w, \lambda)\leq gr'(u, 0)+gr'(v, 0).$$

\begin{prop}\label{propgradingcoincide}
  If $Dyn(\Delta_P)$ is connected, then $gr=gr'$.
\end{prop}

\begin{proof} While it is a general fact that $gr|_{W\times \{0\}}=gr'|_{W\times \{0\}}$, we   illustrate a little bit details here.
   For each $j$,       $v_{j}\cdots v_1\in W_{P_{j}}$ preserves $R_{P_j}$, and $v_{r+1}v_r\cdots v_{j+1}\in W^{P_j}$ maps $R_{P_j}^+$ (resp. $-R_{P_j}^+$) to $R^+$ (resp. $-R^+$). Thus for any $\beta\in R^+_{P_j}$, $w(\beta)\in-R^+$ if and only if
    $v_{j}\cdots v_1(\beta)\in -R^+_{P_j}\subset -R^+$. That is, we have $\ell(v_j\cdots v_1)=|\mbox{Inv}(v_{j}\cdots v_1)|=|\mbox{Inv}(w)\cap  R_{P_{j}}^+|$.
     Hence,
     $\ell(v_j)=\ell(v_j\cdots v_1)-\ell(v_{j-1}\cdots v_1)=|\mbox{Inv}(w)\cap  R_{P_{j}}^+|-|\mbox{Inv}(w)\cap  R_{P_{j-1}}^+|=|\mbox{Inv}(w)\cap  (R_{P_{j}}^+\setminus R_{P_{j-1}}^+)|$.

Note $\Delta_1=\{\alpha_1\}$ and $gr(\mbox{id}, \alpha_1^\vee)=2\mathbf{e}_1$ (by Lemma \ref{lemmagradingJJJ}). Assume the statement follows for   simple roots in  $\Delta_k$. For   $\alpha\in \Delta_{k+1}\setminus \Delta_k$,  say
          $\psi_{\Delta_{k+1}, \Delta_k}(\mbox{id}, \alpha^\vee+Q_{P_k}^\vee)=(w_{P_k}w_{P_k'}, \lambda)$ where $\lambda=\alpha^\vee \!+\! \sum_{i=1}^k a_i\alpha_i^\vee$. Then
     $gr'_{[1,k]}(\mbox{id}, \alpha^\vee)=-gr'_{[1, k]}(w_{P_k}w_{P_k'}, 0)-
                 \sum\nolimits_{i=1}^ka_igr'_{[1, k]}(\mbox{id}, \alpha_i^\vee)=-gr'_{[1, k]}(w_{P_k}w_{P_k'}, 0)-
                 \sum\nolimits_{i=1}^ka_i\sum_{j=1}^k\langle 2\rho_j-2\rho_{j-1},\alpha_i^\vee\rangle \mathbf{e}_j=
                 -gr'_{[1, k]}(w_{P_k}w_{P_k'}, 0)+gr_{[1, k]}(\mbox{id}, \alpha^\vee)-
                 \sum_{j=1}^k\langle 2\rho_j-2\rho_{j-1},\lambda\rangle \mathbf{e}_j$.
   Note $\langle \gamma, \lambda\rangle\in\{0,,-1\}$ for any $\gamma\in R_{P_k}^+$, and $\Delta_{P_k'}=\{\beta\in \Delta_k~|~ \langle \beta, \lambda\rangle =0\}$.  Thus we have $\langle \gamma, \lambda\rangle=-1 $ if   $\gamma\in R_{P_k}^+\setminus R_{P_k'}^+$, or 0 if $\gamma\in  R_{P_k'}^+$.
   Hence, for any $1\leq i\leq k$, we have $-\langle 2\rho_i-2\rho_{i-1},\lambda\rangle=-\!\!\!\!\!\!\!\!\!\!\!\sum\limits_{\gamma\in   R_{P_i}^+\setminus R_{P_{i-1}}^+}\!\!\!\!\langle \gamma,\lambda\rangle=-\!\!\!\!\!\!\!\!\!\!\!\!\sum\limits_{\gamma\in  (R_{P_i}^+\setminus R_{P_{i-1}}^+)\bigcap (R_{P_k}^+\setminus R_{P_k'}^+)}\!\!\!\!\!\langle \gamma,\lambda\rangle=-\!\!\!\!\!\!\!\!\!\!\!\!
    \sum\limits_{\gamma\in   (R_{P_i}^+\setminus R_{P_{i-1}}^+)\bigcap (R_{P_k}^+\setminus R_{P_k'}^+)}\!\!\!\!-1=\big|  (R_{P_i}^+\setminus R_{P_{i-1}}^+)\bigcap  (R_{P_k}^+\setminus R_{P_k'}^+)\big|=\big|  (R_{P_i}^+\setminus R_{P_{i-1}}^+ )\bigcap \mbox{Inv}(w_{P_k}w_{P_k'})\big|=|gr_{[i, i]}(w_{P_k}w_{P_k'}, 0)|$. Hence, $gr_{[1, k]}(\mbox{id}, \alpha^\vee)=gr_{[1, k]}'(\mbox{id}, \alpha^\vee)$. Thus we have $gr(\mbox{id}, \alpha^\vee)=gr'(\mbox{id}, \alpha^\vee)$, by noting $gr_{[k+2, r+1]}(\mbox{id}, \alpha^\vee)=\mathbf{0}=gr_{[k+2, r+1]}'(\mbox{id}, \alpha^\vee)$ and $|gr(\mbox{id}, \alpha^\vee)|=2=|gr'(\mbox{id}, \alpha^\vee)|$. Hence, the statement follows by induction on $k$.
 \end{proof}

When $Dyn(\Delta_P)$ is not connected, we use the same ordering on $\Delta_P$ as in section 5 of \cite{leungli33}. Namely, we
   write  $\Delta_P=\bigsqcup_{k=1}^m\Delta^{(k)}$ such that each
                $Dyn(\Delta_{(k)})$ is a connected component of  $Dyn(\Delta_P)$.
                  Clearly, $\Delta^{(k)}$'s are all of   $A$-type with at most one exception, say           $\Delta^{(m)}$ if it exists.
                     We fix a canonical order on  $\Delta_P$.   Namely, we say   $\Delta_P=(\Delta^{(1)},\cdots, \Delta^{(m)})=(\alpha_1, \cdots, \alpha_r)$ such that
   for each $k$,
    $\Delta^{(k)}=\{\alpha_{k, 1}, \cdots, \alpha_{k, r_k}\}$  satisfying
    (1) if    $\Delta^{(k)}$ is  of $A$-type,  then $Dyn(\Delta^{(k)})$ is given by     \begin{tabular}{l} \raisebox{-0.4ex}[0pt]{$  \circline\!\;\!\!\circ\cdots\, \circline\!\!\!\;\circ $}\\
                 \raisebox{1.1ex}[0pt]{${\hspace{-0.2cm}\scriptstyle{\alpha_{k,1}}\hspace{0.03cm}\alpha_{k, 2}\hspace{0.65cm}\alpha_{k, r_k} } $}
  \end{tabular}\!
   together with the same way   of  denoting  an ending point (by $\alpha_{k, 1}$ or $\alpha_{k, r_k}$) as in section 2.4 of \cite{leungli33};
     (2)  if   $\Delta^{(k)}$ is not of $A$-type,   then   $Dyn(\Delta^{(k)})$ is given in the way of Table \ref{tabrelativeposi}.
  We also denote the   standard basis of $\mathbb{Z}^{r+1}$ as
                 $\{\mathbf{e}_{1, 1},\cdots, \mathbf{e}_{1, r_1}, \cdots, \mathbf{e}_{m, 1}, \cdots, \mathbf{e}_{m, r_m}, \mathbf{e}_{m+1, 1}\}$.
                 In order words,  we have
                  $\mathbf{e}_{k, i}=\mathbf{e}_{i+\sum_{t=1}^{k-1}r_t}$ and $\alpha_{k, i}=\alpha_{i+\sum_{t=1}^{k-1}r_t}$ in terms of our previous notations of $\mathbf{e}_j$'s and $\alpha_j$'s respectively.

 Using Definition \ref{defnofLeungLi33} (resp. \ref{defnofgrading}) with respect to $\Delta^{(k)}$, we obtain a  grading map
       $$gr_{{(k)}}': W\times Q^\vee \longrightarrow \mathbb{Z}^{r_k+1}=\bigoplus\nolimits_{i=1}^{r_k+1}\mathbb{Z}\mathbf{e}_{k, i} $$
  (resp.     $gr_{{(k)}}: W\times Q^\vee \longrightarrow \mathbb{Z}^{r_k+1}=\bigoplus\nolimits_{i=1}^{r_k+1}\mathbb{Z}\mathbf{e}_{k, i}$).
             Note $W_P=W_1\times \cdots\times W_m$ where each $W_k$ is the Weyl subgroup generated by simple reflections from    $\Delta^{(k)}$.
 In particular for any $(w, \lambda) \in W_k\times (\bigoplus_{\alpha\in \Delta^{(k)}}\mathbb{Z}\alpha^\vee)\subset W\times Q^\vee$,   we have
          $gr_{{(k)}}'(w, \lambda)\in \bigoplus_{i=1}^{r_k}\mathbb{Z}\mathbf{e}_{k, i}\hookrightarrow \mathbb{Z}^{r+1}$ which we treat as an element of $\mathbb{Z}^{r+1}$ via the  natural  inclusion.
 Now we recall Definition 5.1 of \cite{leungli33} for general $\Delta_P$ as follows.

 \begin{defn}\label{defgradingforgeneral}
           We define a  {grading map} as follows, say  again   $gr': W\times Q^\vee  \longrightarrow \mathbb{Z}^{r+1}$ by abuse of notation.
  \begin{enumerate}
     \item Write
          $w=v_{m+1}v_m\cdots    v_1$ (uniquely), in which $(v_1, \cdots, v_m, v_{m+1})\in W_1\times \cdots \times W_m\times W^P$.
          Then $gr'(w, 0):= \ell(v_{m+1})\mathbf{e}_{m+1, 1}+\sum_{k=1}^m  gr_{{(k)}}'(v_k, 0)$.  
     \item For each $\alpha_{k, i}\in \Delta^{(k)}$, $gr'(\mbox{id}, q_{\alpha_{k, i}^\vee}) :=  gr_{{(k)}}'(\mbox{id}, q_{\alpha_{k, i}^\vee})$.
             For   $\alpha\in \Delta\setminus\Delta_P$,
              we write    $\psi_{\Delta, \Delta_P}(q_{\alpha^\vee+Q^\vee_P})=w_Pw_{P'}q_{\alpha^\vee}\prod\limits_{k=1}^m\prod\limits_{i=1}^{r_k}q_{{\alpha_{k, i}^\vee}}^{a_{k,i}}$ and then define
                $$gr'(\mbox{id}, {\alpha^\vee})=\big( \ell(w_Pw_{P'})+2+\sum\limits_{k=1}^m\sum\limits_{i=1}^{r_k}2a_{k,i}\big)\mathbf{e}_{m+1, 1}-
                       gr'(w_Pw_{P'}, 0)-\sum\limits_{k=1}^m\sum\limits_{i=1}^{r_k}a_{k,i}gr'(\mbox{id},  {\alpha_{k,i}^\vee}).$$
     \item In general,     $gr'(w, \sum_{\alpha\in \Delta} {b_\alpha}\alpha^\vee):= gr'(w, 0)+\sum_{\alpha\in \Delta}b_{\alpha}gr'(\mbox{id}, {\alpha^\vee})$.
  \end{enumerate}
\end{defn}
By abuse of notation, we denote $\pi_k$ for both of the natural projections
  $$\mathbb{Z}^{r_k+1}=\bigoplus\nolimits_{i=1}^{r_k+1}\mathbb{Z}\mathbf{e}_{k, i}\longrightarrow \bigoplus\nolimits_{i=1}^{r_k}\mathbb{Z}\mathbf{e}_{k, i}\quad
     \mbox{and}\quad  \mathbb{Z}^{r+1}=\bigoplus_{j=1}^{m+1}\bigoplus_{i=1}^{r_j}\mathbb{Z} \mathbf{e}_{j, i}\longrightarrow  \bigoplus_{i=1}^{r_k}\mathbb{Z} \mathbf{e}_{k, i}.$$

  \begin{lemma}\label{lemmarestriction}
     For   $1\leq k\leq m$, we have $\pi_k\circ gr=\pi_k \circ gr_{(k)}$ and $\pi_k\circ gr'=\pi_k \circ gr_{(k)}'$.
  \end{lemma}

 \begin{proof}
   It follows immediately from the definition that   $\pi_k\circ gr'(w, \alpha^\vee)=\pi_k \circ gr_{(k)}'(w, \alpha^\vee)$ for  $(w, \alpha^\vee)\in W\times \Delta^{(k)}$. For    $\beta\in    \Delta^{(\tilde k)}$ where $\tilde k\neq k$, we note
  $\langle \alpha, \beta^\vee\rangle =0$.
  Thus     $\psi_{\Delta, \Delta^{(k)}}(q_{\beta^\vee+Q^\vee_{P_{\Delta^{(k)}}}})=q_{\beta^\vee}$; furthermore if
     $\psi_{\Delta, \Delta_P}(q_{\gamma^\vee+Q^\vee_P})=w_Pw_{P'}q_{\gamma^\vee}\prod\limits_{k=1}^m\prod\limits_{i=1}^{r_k}q_{{\alpha_{k, i}^\vee}}^{a_{k,i}}$ where $\gamma\in \Delta\setminus\Delta_P$,
        then  $\psi_{\Delta, \Delta^{(k)}}(q_{\gamma^\vee+Q^\vee_{P_{\Delta^{(k)}}}})=w_{\check P}w_{\check P'}q_{\gamma^\vee} \prod\limits_{i=1}^{r_k}q_{{\alpha_{k, i}^\vee}}^{a_{k,i}}$ with $w_{\check P}w_{\check P'}$ given by the $W_k$-component of $w_{ P}w_{P'}$, implying $\pi_k\circ gr'(w_{\check P}w_{\check P'})=\pi_k\circ gr'(w_{P}w_{P'})$.
    Thus we have  $gr_{(k)}'(\mbox{id}, \beta^\vee)=2\mathbf{e}_{k, r_k+1} $,  $\pi_k \circ gr'(\mbox{id}, \beta^\vee)= \pi_k \circ gr_{(\tilde k)}'(\mbox{id}, \beta^\vee)=\mathbf{0}=\pi_k\circ gr_{(k)}'(\mbox{id}, \beta^\vee)$, and consequently   $\pi_k \circ gr'(\mbox{id}, \gamma^\vee)=  \pi_k\circ gr_{(k)}'(\mbox{id}, \gamma^\vee)$. Hence, $\pi_k\circ gr'=\pi_k\circ gr_{(k)}'$.

  Due to our notation conventions, we have $\mathbf{e}_j=\mathbf{e}_{k, i}$ for $j=i+\sum_{t=1}^{k-1}r_t$.  Thus
  $\pi_k\circ gr=\pi_k\circ gr_{(k)}$ follows immediately, by noting
  $R_{P_j}^+=R_{P_{\Delta^{(k)}_i}}^+\bigsqcup \big(\bigsqcup_{t=1}^{k-1}R_{P_{\Delta^{(k)}_i}}^+\big)$. 
  \end{proof}

\bigskip

\begin{proof}[Proof of Theorem \ref{thmdefnscoincide}] 
  For each $1\leq k\leq m$, we have  $gr_{(k)}=gr_{(k)}'$ by Proposition \ref{propgradingcoincide}. Thus
    $\pi_k\circ gr=\pi_k\circ gr'$ by Lemma \ref{lemmarestriction}. That is, we have
      $gr_{[1, r]}=gr'_{[1, r]}$. Note $|gr(w, 0)|=|gr'(w, 0)|=\ell(w)$ and $|gr(\mbox{id}, \alpha^\vee)|=|gr'(\mbox{id}, \alpha^\vee)|=2$ for any $\alpha\in\Delta$. Thus we have
       $|gr(w, \lambda)|=|gr'(w, \lambda)|$ for any $(w, \lambda)\in W\times Q^\vee$. Hence,  the statement follows.
\end{proof}

\bigskip

For general $\Delta_P$,   the subset
   $ \{gr(w, \lambda)~|~ q_{\lambda} {\sigma^w}\in QH^*(G/B)\}$ of   $\mathbb{Z}^{r+1}$, denoted as $S$   by abuse of notation, turns out again to be a  totally-ordered sub-semigroup of $\mathbb{Z}^{r+1}$. (The proof is similar to the one for Lemma 2.12 of \cite{leungli33} in the case when $Dyn(\Delta_P)$ is connected.)
In the same way as in section \ref{subsecMainResult},  we obtain
 an $S$-family    of subspaces of $QH^*(G/B)$;   it naturally extends to a $\mathbb{Z}^{r+1}$-family, and induces graded vector subspaces. Namely, by abuse of notation, we have
  $\mathcal{F}=\{F_{\mathbf{a}}\}$ with $ F_{\mathbf{a}}:= \bigoplus_{gr(w, {\lambda})\leq \mathbf{a}}\mathbb{Q}q_{\lambda}\sigma^w$;
   $Gr^{\mathcal{F}}(QH^*(G/B)):=\bigoplus_{\mathbf{a}\in \mathbb{Z}^{r+1}} Gr_\mathbf{a}^{\mathcal{F}}, \mbox{ where }
                Gr_{\mathbf{a}}^{\mathcal{F}}:=F_{\mathbf{a}}\big/\sum_{\mathbf{b}<\mathbf{a}}F_{\mathbf{b}}$; for each $1\leq j\leq r+1$,
     $Gr^{\mathcal{F}}_{(j)} :=\bigoplus_{i\in \mathbb{Z}}Gr_{i\mathbf{e}_{j}}^{\mathcal{F}}. $ In addition, we denote
       $$\mathcal{I}:=\bigoplus_{gr_{[r+1, r+1]}(w, \lambda)>\mathbf{0}}\mathbb{Q}q_\lambda\sigma^w\subset QH^*(G/B)$$
and        $$  \mathcal{A}:=\psi_{\Delta, \Delta_P}(QH^*(G/P))\oplus\mathcal{J} \quad \mbox{where}\quad\mathcal{J}:=F_{-\mathbf{e}_{r+1}}.$$
For each $1\leq j\leq r$, $X_j:=P_j/P_{j-1}$ is a Grassmannian (possibly of general type), and the quantum cohomology $QH^*(X_j)$ is therefore isomorphic to $H^*(X_j)\otimes \mathbb{Q}[t_j]$ as vector spaces.  Note $X_{r+1}:=P_{r+1}/P_r=G/P$.

Now we can restate Theorem \ref{thmgenthm11} in the introduction more precisely as follows.

\begin{thm}\label{thmgenthm22} $\mbox{}$
\begin{enumerate}
  \item  $QH^*(G/B)$ has an $S$-filtered algebra structure  with filtration $\mathcal{F}$, which naturally extends to a $\mathbb{Z}^{r+1}$-filtered algebra structure on         $QH^*(G/B)$.
   \item   $\mathcal{I}$ is an ideal  of $QH^*(G/B)$, and there is a canonical algebra isomorphism
      $$ QH^*(G/B)/\mathcal{I}  \overset{\simeq}{\longrightarrow}QH^*(P/B).$$
  \item   $\mathcal{A}$ is a subalgebra of $QH^*(G/B)$ and    $\mathcal{J}$ is an ideal of $\mathcal{A}$. Furthermore,
   there is a canonical algebra isomorphism (induced by $\psi_{\Delta, \Delta_P}$)
             $$ QH^*(G/P)\cong \mathcal{A}/\mathcal{J}.$$

  \item There is a canonical isomorphism of $\mathbb{Z}^{r}\times \mathbb{Z}_{\geq 0}$-graded algebras:
     $$Gr^{\mathcal{F}}(QH^*(G/B)[q_1^{-1}, \cdots, q_r^{-1}])\overset{\cong}{\longrightarrow}
    \big(\bigotimes_{j=1}^{r}QH^*(X_{j})[t_{j}^{-1}]\big)\bigotimes Gr^{\mathcal{F}}_{(r+1)}.$$
         There is also  an injective  morphism of graded algebras:
            $$\Psi_{r+1}:\,\,  \,\, QH^*(G/P)\hookrightarrow Gr^{\mathcal{F}}_{(r+1)},$$
           well defined by     $q_{\lambda_P}\sigma^w\mapsto \overline{\psi_{\Delta, \Delta_P}(q_{\lambda_P}\sigma^w)}$.
    Furthermore,  $\Psi_{r+1}$ is    an isomorphism if and only if either
            {\upshape (a) } $\Delta^{(k)}$'s are all  of $A$-type or {\upshape (b) } the only exception $\Delta^{(m)}$ is of $B_2$-type with $\alpha_r$ being a short simple root.
\end{enumerate}
\end{thm}


\begin{remark}
    Say  $\alpha_j\in \Delta^{(k)}$, i.e.,  $j=i+\sum_{p=1}^{k-1}r_p$ for some $1\leq i\leq r_k$. Whenever $(k, i)\neq (m, r_m)$, we have $X_j\cong \mathbb{P}^i$. If $\Delta^{(m)}=\{\alpha_{r-1}, \alpha_r\}$ is of type $B_2$ and $\alpha_r$ is a short simple root,
 then we have $X_{r-1}\cong \mathbb{P}^1$ and $X_r\cong \mathbb{P}^3$. In other words, \textit{$\Psi_{r+1}$ is an isomorphism if and only if
    all $X_j$ ($1\leq j\leq r$) are projective spaces. }

\end{remark}

\bigskip

\begin{proof}[(Sketch) Proof of Theorem \ref{thmgenthm22}]
   The quantum cohomology ring $QH^*(G/B)$ is generated by the divisor Schubert classes $\{\sigma^{s_1}, \cdots, \sigma^{s_n}\}$.
   The well-known  quantum Chevalley formula (see \cite{fw}) tells us
    $ \sigma^u\star\sigma^{s_{i}} =\sum_\gamma \langle \omega_i, \gamma^\vee\rangle \sigma^{us_\gamma}+\sum_\gamma \langle \omega_i, \gamma^\vee\rangle q_{\gamma^\vee}\sigma^{us_\gamma}$ where $u\in W$ is arbitrary, and $\{\omega_1, \cdots, \omega_n\}$ denote the fundamental weights.

    To show (1), it suffices to use induction on $\ell(u)$ and the positivity of Gromov-Witten invariants $N_{u, v}^{w, \lambda}$, together with the Key Lemma of \cite{leungli33} for general $\Delta_P$. Namely,
    we need to show  $gr(us_\gamma, 0)\leq gr(u, 0)+gr(s_i, 0)$ (resp. $gr(us_\gamma, \gamma^\vee)\leq gr(u, 0)+gr(s_i, 0)$) whenever the corresponding coefficient $\langle\omega_i, \gamma^\vee\rangle \neq 0$. Under this hypothesis,    the expected inequality will hold if we replace ``$gr$" by ``$gr_{(k)}$",  due to  the Key Lemma of \cite{leungli33} which works for any $\Delta^{(k)}$. Therein the proof of the Key Lemma is most complicated part of the paper. We used the notion of virtual null coroot to do some reductions, but still had to do a big case by case analysis.
    Hence, the   expected inequality  holds   if we replace ``$gr$" by ``$\pi_k\circ gr$" (for any $1\leq k\leq m$), due to Lemma \ref{lemmarestriction}.  That is, it holds when we replace ``$gr$" by ``$gr_{[1, r]}$". Thus the expected inequality holds by noting that  $|gr(us_\gamma, 0)|$  (resp. $|gr(us_\gamma, \gamma^\vee)|$) is equal to $|gr(u, 0)|+|gr(s_i, 0)|$.

    The proof of (2) is exactly the same as the proof of Theorem 1.3 in \cite{leungli33}. The quotient  $P/B$ is again a complete flag variety, and therefore teh quantum cohomology $QH^(P/B)$ is generated by the special Schubert classes $\sigma^{s_i}$, $i=1, \cdots, r$. The prove is done by showing that
    $ QH^*(G/B)/\mathcal{I}$ is generated by $\overline{\sigma^{s_i}}$, $i=1, \cdots, r$, respecting the same quantum Chevalley formula.

      Statement (3) is in fact a consequence of (4).

      The proof of (4) is similar to the above one for (1). Namely we reduce $gr$ to $\pi_k\circ gr=\pi_k\circ gr_{(k)}$. The expected statement will hold with respect to $gr_{(k)}$, by using either the corresponding results of \cite{leungli33} for $\Delta^{(k)}$  of type $A$ or
      Theorem \ref{thmfirstmainthm} when $\Delta^{(k)}$ is not of type $A$. The proof of the former case is much simpler than the latter one, although the ideas are similar.    Here we also need to use the  same observation that
        $|gr(w, \lambda)|=|gr(u, 0)|+|gr(v, 0)|$ whenever the Gromov-Witten invariant $N_{u, v}^{w, \lambda}$ in the quantum product $\sigma^{u}\star \sigma^v$ is nonzero.
\end{proof}

\section{Appendix} \label{exceptional}
\subsection{Proof of Lemma \ref{lemmacoealpharrrnegative} (Continued)}\label{subsecAPPlemmaproof}
Recall   $\varepsilon_j=-\langle \alpha_j, \lambda\rangle\geq 0$, $j=1, \cdots, r$.
$gr_{[r, r]}(q_{r})=x\mathbf{e}_r, gr_{[r, r]}(q_{r+1})=y\mathbf{e}_r$,  $gr_{[r, r]}(q_{r+2})=z\mathbf{e}_r$.
Define $(c_1, \cdots, c_r)$ by $$\sum_{\beta\in R_{P_r}\setminus R_{P_{r-1}}}\beta=-y\sum_{i=1}^r c_i\alpha_i,$$
which are   described in Table \ref{tabtabforlemma} by direct calculations. Therein we recall that the case C9) with $r=2$ has been excluded from the discussion.
By definition,  we have   $$gr_{[r, r]}(q_\lambda)=(xa_r+ya_{r+1}+za_{r+2})\mathbf{e}_r= (y\sum_{j=1}^r c_j\varepsilon_j )\mathbf{e}_r.$$
 \renewcommand{\arraystretch}{1.2}

     \begin{table}[h]
  \caption{\label{tabtabforlemma}   
    }
    \begin{tabular}{|c|c||c|c|c|c|}
     \hline
        \multicolumn{2}{|c||}{ }   & $(c_1, \cdots, c_r)$   & $(x, y, z)$  & { $-yc_r$}          &$|R^+_P|-|R^+_{\tilde P}|$\\ \hline \hline
             & {$r=6$}   &$(1, 2, 3, 2, 1, 2)$ & $(11, -11, 0)$ & $22$& 
        $21$\\ \cline{2-5}
       \raisebox{1.5ex}[0pt]{\upshape C4)}   & $r=7$   & $(1, 2, 3, 4, {8\over 3}, {4\over 3}, {7\over 3})$ &  $(14,-21,0)$ & $49$  &  
        $42$ \\ \hline
       \multicolumn{2}{|c||}{\upshape C5)}     & $({2\over 5}, {4\over 5}, {6\over 5}, {3\over 5}, 1)$  & &   &    \\ \cline{1-3}
         & {$r=4$} &  $({1\over 2}, 1, {1\over 2},  1)$   & &  & 
        \\ \cline{2-3}
         & $r=5$ &    $({2\over 5}, {4\over 5}, {6\over 5}, {3\over 5}, 1)$ & $(2r-2, {(1-r)r\over 2}, 1-r)$ & ${r(r-1)\over 2}$  &  
       ${r(r-1)\over 2}$ \\ \cline{2-3}
     \raisebox{1.5ex}[0pt]{\upshape C7)}& {$r=6$} &  $({1\over 3}, {2\over 3}, 1, {4\over 3}, {2\over 3}, 1)$  &  & & 
         \\ \cline{2-3}
         & $r=7$ &   $({2\over 7}, {4\over 7}, {6\over 7}, {8\over 7},{10\over 7}, {5\over 7}, 1)$  & & &  
          \\ \hline
       {\upshape C9)}& {$r=3$} &  $({1\over 3}, {2\over 3}, 1)$  & $(6, -9, 0)$ & $9$& 
        $6$\\ \hline
        \multicolumn{2}{|c||}{\upshape C10)} &   $({2\over 3}, {4\over 3}, 1)$ &$(4, -6, 0)$ & $6$  &  
        $6$ \\ \hline
       \end{tabular}
       \end{table}

  Therefore, if $\varepsilon_r>0$, then we have
  $$-(xa_r+ya_{r+1}+za_{r+2})=(-y)\cdot \sum_{j=1}^rc_j\varepsilon_j\geq (-y)\cdot c_r\cdot 1\geq |R^+_P|-|R^+_{\tilde P}| \geq |R^+_P|-|R^+_{\tilde P}\cup R^+_{\hat P}| .$$
  If case C9) with $r=3$ occurs, we are done.
 If   $-gr_{[r, r]}(q_\lambda)= (|R^+_P|-|R^+_{\tilde P}\cup R^+_{\hat P}|)\mathbf{e}_r$  held,
   then C5) or  C7)   occurs and all the above inequalities are   equalities. This  implies that $\varepsilon_r=1$ and $ \varepsilon_j=0, j=1, \cdots, r-1$. Therefore we have   $\langle \gamma, \lambda\rangle\in \{0, -1\}$ for any   $\gamma\in R_P^+$, by noting that
            $\gamma=\varepsilon \alpha_r+\sum_{i=1}^{r-1}c_i\alpha_i$ (where $\varepsilon\in \{0, 1\}$) for all these three cases.  That is, $\lambda=\lambda_B$ is the Peterson-Woodward lifting of $\lambda_P$, contradicting with the hypothesis. Hence, the statement follows if  $\varepsilon_r>0$.

Assume now  $\varepsilon_r=0$. Since $\lambda\neq \lambda_B$, we have  $\varepsilon_j=0$ for all $\alpha_j\in \Delta_P$ but exactly one exception, say $\alpha_k$.   In addition,
we have  $$-(xa_r+ya_{r+1}+za_{r+2})=-y c_k\varepsilon_k=-yc_k,$$
together  with the property that   the coefficient  $\theta_k$ in the highest root $\theta=\sum_{i=1}^r\theta_i\alpha_i$ of $R^+_P$ is not equal to $1$. (Otherwise, $\lambda$ would be the Peterson-Woodward lifting of $\lambda_P$, contradicting with the hypothesis.) Thus if $c_k\geq c_r$, then we have
$$-(xa_r+ya_{r+1}+za_{r+2})=-yc_k\geq -yc_r\geq |R^+_P|-|R^+_{\tilde P}| > |R^+_P|-|R^+_{\tilde P}\cup R^+_{\hat P}|.$$ Here the last inequality holds since $\alpha_r\in  R^+_{\hat P}\setminus R^+_{\tilde P}$. If $c_k<c_r$, then all possible $k$, together with $-yc_k$ and
  the number  $|R^+_P|-|R^+_{\tilde P}\cup R^+_{\hat P}| =|R^+_P|-|R^+_{\tilde P}|-|R^+_{P_{\Delta_P\setminus\{\alpha_k\}}}|+ |R^+_{P_{\Delta_{\tilde P}\setminus\{\alpha_k\}}}|$, are precisely given in Table \ref{tabtabforlemma22} by direct calculations.
  In particular, we also have $-(xa_r+ya_{r+1}+za_{r+2})=-y c_k>|R^+_P|-|R^+_{\tilde P}\cup R^+_{\hat P}|$.

  \renewcommand{\arraystretch}{1.2}
  \begin{center}
   \begin{table}[h]
  \caption{\label{tabtabforlemma22}   
    }
    \begin{tabular}{|c|c||c|c|c|}
     \hline
        \multicolumn{2}{|c||}{ }     &   $k$  & { $-yc_k$}          &$|R^+_P|-|R^+_{\tilde P}\cup R^+_{\hat P}|$\\ \hline \hline
             &   &$2$   & $42$& 
        $32$\\ \cline{3-5}
       \raisebox{1.5ex}[0pt]{\upshape C4)}   &  \raisebox{1.5ex}[0pt]{$r=7$} & $6$   & $28$  &  
        $27$ \\ \hline
       \multicolumn{2}{|c||}{\upshape C5)}   & $2$  & 8   &  7  \\ \hline
         &  $r=5$ & $2$    & 8 & 7 
        \\ \cline{2-5}
         & $r=6$ & $2$  & 10  & 9 
         \\ \cline{2-5}
     \raisebox{1.5ex}[0pt]{\upshape C7)}&  & $2$    & 12 & 11 
         \\ \cline{3-5}
         &  \raisebox{1.5ex}[0pt]{$r=7$} & 3  & 18  & 15 
          \\ \hline
       {\upshape C9)}& {$r=3$} & $2$    & $6$& 
        $5$\\ \hline
        \multicolumn{2}{|c||}{\upshape C10)}  & $1$  & $4$  &  
        $3$ \\ \hline
       \end{tabular}
     \end{table}
    \end{center}

\subsection{Proof of Proposition \ref{propmorphismofalgebras} (3) (Continued)}\label{subsecAPPprop33uu}

For each case,  we uniformly denote those $(q_\lambda, u)$  in Table \ref{tabNONvirtualnullcoroot} in order as $(q_{\lambda_i}, u_i)$'s, and denote by $\tilde u_i$   the minimal length representative of $u_iW_{\tilde P}$ as before (i.e., $\tilde u_i$ is given by a subexpression  $s_L$   of $u_i$ with the sequence ending with $r$). We also denote  those virtual null coroot(s)  $\mu_B$ in  Table \ref{tabvirtualnullcoroot} in order as $\mu_1, \mu_2$. Namely if there is a unique $\mu_B$, then we denote $\mu_1=\mu_2=\mu_B$ for convenience.

 Due to Lemma \ref{lemmagrading33cases} again,   it suffices to show all the equalities  in Table \ref{tabtabforprodUU}  hold in $Gr^{\mathcal{F}}(QH^*(G/B))$ for the corresponding cases.
  Note   $\overline{\sigma^{u_i}}\star \overline{\sigma^{u_j}}=\overline{\sum N_{u_i, u_j}^{w, \eta}q_{\eta}\sigma^{w}}$ where $gr(q_\eta\sigma^w)=gr(\sigma^{u_i})+gr(\sigma^{u_j})$. Consequently, we have $w\in W_P$, $\eta=\sum_{k=1}^rb_k \alpha_k^\vee$  and $\ell(\tilde u_i)+\ell(\tilde u_j)=b_r |gr_{[r, r]}(q_r)|+|gr_{[r, r]}(\sigma^w)|$. Since $0\leq |gr_{[r, r]}(\sigma^w)|\leq |R^+_P|-|R^+_{\tilde P}|$, we have $b^{\scriptsize\mbox{max}}\geq b_r\geq b^{\scriptsize\mbox{min}} \geq 0$ for certain integers $b^{\scriptsize\mbox{max}}, b^{\scriptsize\mbox{min}}$.
 Write $w=vw_2$ where $v\in W^{\tilde P}_P$ and $w_2\in W_{\tilde P}$.  Once $b_r$ is given,  both $\ell(v)$ and $(w_2, \eta)$ are fixed by the above equalities on gradings together with Lemma \ref{lemmauniquegradinginPB}.
    There is a unique term, say  $\overline{q_\vartheta\sigma^{\check w}}$, on the right hand side of each   expected identity in Table \ref{tabtabforprodUU},  where $\check w=\mbox{id}$, $u_2$  or $u_3$.
      It is easy to check that $q_\vartheta\sigma^{\check w}$ is of expected grading with $b_r(\vartheta)=b^{\scriptsize\mbox{max}}$.
       Thus if $b_r(\eta)= b^{\scriptsize\mbox{max}}$, then we have $\eta=\vartheta$ and $w=v\check w_2$ with $\ell(v)=\ell(\tilde v)$.  Here  $\tilde v\in W^{\tilde P}_P$ and $\check w_2\in W_{\tilde P}$ are given by  $\check w=\tilde v \check w_2$.
    In particular, we have $w=\mbox{id}$ if $\check w=\mbox{id}$.  Hence, in order to conclude the expected equality, it suffices to show
        \begin{enumerate}
          \item $N_{u_i, u_j}^{v\check w_2, \vartheta}=1$ if $v=\tilde v$, or 0 otherwise;
          \item  $N_{u_i, u_j}^{w, \eta}=0$ whenever $b^{\scriptsize\mbox{max}}>b_r(\eta)\geq b^{\scriptsize\mbox{min}}$. Similar to the proof in section \ref{subsecproofforprop33}, this claim follows from the next two:
              \begin{enumerate}
                \item $\sigma^{\tilde u_i}\cup \sigma^{\tilde u_j}=0$ in $H^*(P/B)$;
                \item $N_{u_i, u_j}^{w, \eta}=0$ whenever $b^{\scriptsize\mbox{max}}>b_r(\eta)\geq \max\{1, b^{\scriptsize\mbox{min}}\}$.
              \end{enumerate}
        \end{enumerate}

\renewcommand{\arraystretch}{1.3}

  \begin{table}[h]
  \caption{\label{tabtabforprodUU}   
    }
    \begin{tabular}{|c|c||l|c|c|}
     \hline
       \multicolumn{2}{|c||}{ }     &   {}\hspace{1.1cm}Equalities  & {  $N_{u_i, u_j}^{v\check w_2, \vartheta}$ }          &$ q_{\eta}$\\ \hline \hline

                   &     &   & &
                 $ q_7q_1^2q_2^2q_3^2q_4^2q_5$,  \\
       \raisebox{1.5ex}[0pt]{\upshape C4)\!\!}       &   \raisebox{1.5ex}[0pt]{$r=7$}  &   & &   $q_7^2q_1^2q_2^2q_3^2q_4^3q_5^2q_6
              $      \\ \cline{1-2}\cline{5-5}

            {\upshape C7)\!\!}    & $\!\!r\!=\!4;6\!\!\,{}$&   \raisebox{1.5ex}[0pt]{$\overline{\sigma^{u_1}}\star \overline{\sigma^{u_1}}=\overline{q_{\mu_1-2\lambda_1}}$}   &  \raisebox{1.5ex}[0pt]{$N_{u_1, {\scriptsize\mbox{id}}}^{u_1^{-1}, 0}$}  &   $q_4$;  $ q_6^2q_3q_4^2q_5 $          \\ \cline{1-2}\cline{5-5}

\multicolumn{2}{|c||}{\upshape C10)}   &   &       &  $ q_3^2q_2$   \\ \cline{1-5}

      \multicolumn{2}{|c||}{  }   & $\overline{\sigma^{u_2}}\star \overline{\sigma^{u_2}}=\overline{q_{\mu_2-2\lambda_2}}$  &  $N_{u_2, {\scriptsize\mbox{id}}}^{u_2^{-1}, 0}$ &          \\ \cline{3-4}
          \multicolumn{2}{|c||}{\raisebox{1.5ex}[0pt]{\upshape C5),\, C7)} } & $\overline{\sigma^{u_1}}\star \overline{\sigma^{u_2}}=\overline{q_{\lambda_3-\lambda_1-\lambda_2}\sigma^{u_3}}$ &  $N_{u_1s_rs_{r-2}\cdots s_2s_1,  s_{1}\cdots s_{r-1}}^{v\check w_2,\,\, 0}$ &   \raisebox{1.5ex}[0pt]{$\emptyset$}         \\ \cline{1-2} \cline{1-5}

           \multicolumn{2}{|c||}{{\upshape C5)\!\!} }   &   &   &   $ \emptyset $         \\ \cline{1-2}\cline{5-5}
        &  {$r=7$} &   \raisebox{1.5ex}[0pt]{$\overline{\sigma^{u_1}}\star   \overline{\sigma^{u_1}}=\overline{q_{\lambda_2-2\lambda_1}\sigma^{u_2}}$}  &    \raisebox{1.5ex}[0pt]{$N_{s_{r-1}\cdots s_1, s_{r-1}\cdots s_1}^{v\check w_2, 0}$} &  $q_7^2q_4q_5^2q_6$         \\ \cline{2-5}

          \raisebox{1.5ex}[0pt]{\upshape C7)\!\!}    &  {$r=5$}  & $\overline{\sigma^{u_1}}\star \overline{\sigma^{u_3}}=\overline{q_{\mu_1-\lambda_1-\lambda_3}}$   &  $N_{u_1, {\scriptsize\mbox{id}}}^{u_3^{-1}, 0}$  &     $q_5q_3q_4$       \\ \hline

               &     &  $\overline{\sigma^{u_1}}\star \overline{\sigma^{u_2}}=\overline{q_{\mu_1-\lambda_1-\lambda_2}}$    & $N_{u_1, {\scriptsize\mbox{id}}}^{u_2^{-1}, 0}$  &       $  q_6q_1q_2q_3q_4q_5$     \\ \cline{3-5}
      \raisebox{1.5ex}[0pt]{\upshape C4)\!\!}  &    \raisebox{1.5ex}[0pt]{$r=6$} & $\overline{\sigma^{u_1}}\star \overline{\sigma^{u_1}}=\overline{q_{\lambda_2-2\lambda_1}\sigma^{u_2}}$  & $N_{s_{54362345}, s_{54362345}}^{ v\check w_2, 0}$  &   $\emptyset$     \\ \cline{1-5}

          \end{tabular}
  \end{table}

 To show (1), we use Proposition \ref{propQtoCthm} (2) repeatedly. As a consequence,  we can conclude that $N_{u_i, u_j}^{v\check w_2, \vartheta}$ coincides with a classical intersection number given in Table \ref{tabtabforprodUU} as well, which is of the form either $N_{u_i, {\scriptsize\mbox{id}}}^{u_j^{-1}, 0}$ or  $N_{u', v'}^{v\check w_2, 0}$. The formal one is equal to 1 by checking $u_i=u_j^{-1}$ easily. Denote $\Delta_{\check P}:=\Delta_P\setminus\{\alpha_k\}$. For the latter one, it is easy to check that both $u',  v'$ are in $W^{\check P}_P$, where  $s_k$ denotes the last simple reflection in the reduced expression of  $\check w_2$. Thus  $N_{u', v'}^{v\check w_2, 0}=0$ unless $v\check w_2$ is in $W^{\check P}_P$ as well.  In addition, it is  easy to check that $\ell(u')+\ell(v')=\ell(v\check w_2)=\dim P/\check P$, and  that $u'$ is the minimal length representative of $W_{\check P}=w_P v'W_{\check P}$. Thus $u'$ is dual to $v'$ with respect to the canonical non-degenerated bilinear form on $H^*(P/\check P)$. Hence,  $N_{u', v'}^{v\check w_2, 0}=1$. (See e.g. section 3 of \cite{fw} for these well-known facts.) That is, (1) follows.  To illustrate the above reduction more clearly, we give a little bit more details for $N_{u_1, u_1}^{v\check w_2, \vartheta}$ in  case C4) with $r=6$. In this case, we have $u_1=s_{54362132436}s_5s_4s_3s_2s_1$ and $q_{\vartheta}=q_{\lambda_2-2\lambda_1}=q_1^2q_2^2q_3^2q_4q_6$.
  Proposition \ref{propQtoCthm} (2)  $u_1=s_{54362132436}s_5s_4s_3s_2s_1$ and $q_{\vartheta}=q_{\lambda_2-2\lambda_1}=q_1^2q_2^2q_3^2q_4q_6$.
Using Proposition \ref{propQtoCthm} (2), we can first deduce $N_{u_1, u_1}^{v\check w_2, \vartheta}=N_{u_1, s_{54362345}}^{v\check w_2 s_{12346321}, 0}$.
 Denote   $u':=s_{54362345}$. Note $\Delta_{\check P}=\{\alpha_1, \alpha_2, \alpha_3, \alpha_4, \alpha_6\}$, and $u'(\alpha)\in R^+$ for   all $\alpha\in \Delta_{\check P}$. Thus we can further deduce that
$N_{u_1, u'}^{v\check w_2 s_{12346321}, 0}=N_{u_1s_{12364321}, u'}^{v\check w_2, 0}$. That is, we have
   $N_{u_1, u_1}^{v\check w_2, \vartheta}=N_{u', u'}^{v s_{12345}, 0}$. Note $\dim P/\check P=36-20=16=\ell(\check w)$ and $\check w=s_{12346325436}s_{12345}\in W^{\check P}_P$. That is, $\check w$ is the (unique) longest element in
    $W^{\check P}_P$.
   Thus  $N_{u', u'}^{v s_{12345}, 0}=0$ unless $v\check w_2=\check w$. Note $w_{\check P}=s_{43621324361234123121}$, $\ell( u'w_{\check P} (u')^{-1})\leq 2\ell(u')+\ell(w_{\check P})=36=|R^+_P|$, and $u'w_{\check P} (u')^{-1}(\alpha)\in -R^+$ for all $\alpha\in \Delta_P$. Thus $w_P= u'w_{\check P} (u')^{-1}$. That is, $w_Pu'= u'w_{\check P}$. In other words, $\sigma^{u'}$ is dual to itself in $H^*(P/\check P)$. Hence, $N_{u', u'}^{\check w, 0}=1$.

   To show (a), we note $\tilde u_i\in W^{\tilde P}_P$. Thus in $H^*(P/B)$, we have $N_{\tilde u_i, \tilde u_j}^{w', 0}=0$ unless $w'\in W^{\tilde P}_P$. Consequently, we have $\sigma^{\tilde u_i}\cup \sigma^{\tilde u_j}=0$ in $H^*(P/B)$ if $\ell(\tilde u_i)+ \ell(\tilde u_j)>\dim P/\tilde P=|R^+_P|-|R^+_{\tilde P}|$. Hence, (a) follows immediately from either the above inequality or the combination of Lemma \ref{lemmavanishingofcupproductwv} and Corollary \ref{corvanishingtildeu2}, except for the case when C4) with $r=7$ occurs. In this exceptional case,
    we note $\ell(\tilde u_1)=21={1\over 2 }\dim P/\tilde P$. Hence, we have $\sigma^{\tilde u_1}\cup \sigma^{\tilde u_1}=N_{\tilde u_1, \tilde u_1}^{w', 0} \sigma^{w'}$ in $H^*(P/B)$, where $w'$ denotes the longest element in $W^{\tilde P}_P$. Equivalently, the same equality holds in $H^*(P/\tilde P)$, once we treat   the Schubert classes $\sigma^{\tilde u_1}, \sigma^{w'}$ as elements $H^*(P/\tilde P)$ canonically.  Then we have
     $N_{\tilde u_1, \tilde u_1}^{w', 0} \sigma^{w'}=\int_{[P/\tilde P]}\sigma^{\tilde u_1}\cup \sigma^{\tilde u_1}\cup \sigma^{\scriptsize\mbox{id}}=N_{\tilde u_1, {\scriptsize\mbox{id}}}^{u'', 0}=1$ if $\tilde u_1=u''$, or 0 otherwise. Here  $u''$ is the minimal length representative of the coset $w_P\tilde u_1W_{\tilde P}$. Note $\tilde u_1$ maps $R_{\tilde P}^+$ to $R^+_P$, and $w_P$ maps $R^+_P$ to $-R^+_P$. Hence, we have $w_P\tilde u_1=u''w_{\tilde P}$, i.e.,  $w_P=u''w_{\tilde P}(\tilde u_1)^{-1}$. By direct calculation,  we have     $\tilde u_1w_{\tilde P}(\tilde u_1)^{-1}(\alpha_1)=\alpha_1+\cdots+\alpha_5\not\in -R^+$.      Thus $\tilde u_1\neq u''$ and (a) follows.  Note $\tilde u_1\leqslant w_P\tilde u_1=u''w_{\tilde P}$ if and only if $\tilde u_1\leqslant u''$. Since $\ell(u'')=\ell(u''w_{\tilde P})-\ell(w_{\tilde P})=\ell(w_P\tilde u_1)-\ell(w_{\tilde P})=(63-21)-21=21=\ell(\tilde u_1)$ and $\tilde u_1\neq u''$, we can further conclude  $\tilde u_1\not\leqslant w_P\tilde u_1$.

     It remains to show (b). All the coroots $\eta$ satisfying the hypothesis of (b) are given   in terms of $q_\eta$ in the last column of Table \ref{tabtabforprodUU}   if it exists, or  ``$\emptyset$"   otherwise.
    \begin{enumerate}
      \item[1)] For case C4) with $r=7$, we  note $\mbox{sgn}_7(u_1)=0$. If $q_\eta=  q_7^2q_1^2q_2^2q_3^2q_4^3q_5^2q_6$, then we have $\langle \alpha_7, \eta\rangle=1$, and consequently $N_{u_1, u_1}^{w, \eta}=0$ by Proposition \ref{propQtoCthm} (1).
                Similarly, for   C4) with $r=6$ we have $N_{u_1, u_2}^{w, \eta}=0$ by considering $\mbox{sgn}_6$.

      \item [2)]   For case C4) with $r=7$ and  $q_\eta= q_7q_1^2q_2^2q_3^2q_4^2q_5$, we can first conclude $N_{u_1, u_1}^{w, \eta}=N_{u_1, u'}^{ws_{1234574321}, 0}$ by using Proposition \ref{propQtoCthm} (2) repeatedly. Here $u':=u_1s_{123474321}=s_{12347543654723456}$. Denote $\Delta_{\check P}:=\Delta_P\setminus \{\alpha_6\}$, and note $u'\in  W^{\check P}_P$. Thus we can further conclude $N_{u_1, u_1}^{w, \eta}=N_{u', u'}^{w, 0}$. Since $2\ell(u')=34>63-30=\dim P/\check P$, we have $N_{u', u'}^{w, 0}=0$. Similarly, (b) follows if  case C7)  occurs.

        \item [3)] For case C10), we can   conclude $N_{u_1, u_1}^{w, \eta}=N_{s_{321},s_{321}}^{w, 0}$. Therefore it is equal to 0, by noting $(s_{321})^{-1}w_P=s_{321}s_{321}\not\geqslant (s_{321})^{-1}$ and using Lemma \ref{lemmavanishingofcupproductwv}.
    \end{enumerate}
  Hence, (b) follows.

\section*{Acknowledgements}
The  author   is grateful to Naichung Conan Leung for valuable suggestions and constant encouragement.
He also thanks Ivan Chi-Ho Ip and Leonardo Constantin Mihalcea for useful discussions and advice. Especially, he  thanks the referee for his/her careful reading and very helpful comments. The author is partially supported by JSPS Grant-in-Aid for Young Scientists (B)  No. 25870175.
\bibliographystyle{amsplain}

\end{document}